\documentclass[twoside,11pt]{article}
\RequirePackage{epsfig}
\RequirePackage{amssymb}
\RequirePackage{natbib}
\RequirePackage{graphicx}
\let\oldRequirePackage=\RequirePackage
\renewcommand{\RequirePackage}[2][]{}
\usepackage{jmlr2e}
\let\RequirePackage=\oldRequirePackage

\usepackage{hyperref}
\hypersetup{
    colorlinks = true, linkcolor=blue,
}
\usepackage{cmap}
\usepackage[utf8x]{inputenc}
\usepackage[T1]{fontenc}
\usepackage[english]{babel}
\usepackage{amssymb,amsmath}
\usepackage{bbm}
\usepackage{microtype}
\usepackage{lmodern}
\usepackage{xcolor}
\usepackage{ifthen}
\usepackage{graphicx}
\usepackage{mathrsfs}
\graphicspath{{./pictures/}}

\newcommand{\markupdraft}[2]{
    \ifthenelse{\equal{#1}{color}}{\color{#2}}{}
    \ifthenelse{\equal{#1}{display}}{#2}{}
    \ifthenelse{\equal{#1}{todo}}{#2}{}
}
\renewcommand{\markupdraft}[3][]{}  

\def\N{{\mathbb{N}}}

\def\R{{\mathbb{R}}}

\newcommand{\1}{\mathbbm{1}}
\DeclareMathOperator{\sign}{sign}

\renewcommand{\geq}{\geqslant}
\renewcommand{\leq}{\leqslant}

\def\eps{\varepsilon}
\renewcommand{\epsilon}{\varepsilon}
\renewcommand{\phi}{\varphi}

\def\d{{\mathrm{d}}}

\newcommand{\deq}{\mathrel{\mathop:}=}
\newcommand{\eqd}{=\mathrel{\mathop:}}

\newcommand{\mc}[2]{\newcommand{#1}{\ensuremath{#2}}} 
\newcommand{\mco}[3][]{\newcommand{#2}[1][#1]{\ensuremath{#3}}} 
\mc{\Ptheta}{{P_\theta}}
\mc{\gradt}{\widetilde\nabla_{\!\theta}}
\mc{\thegoal}{\E_{\Pt}\Wf}
\mc{\T}{\mathsf{T}}
\mc{\mueff}{\mu_\textsc{w}}
\mc{\thetat}{{\theta^t}}
\mc{\Pt}{{P_\theta}}
\mc{\Ptt}{P_{\thetat}}
\mc{\Wf}{W_{\thetat}^f}
\mc{\Qtt}{Q_{\thetat}}
\mc{\Cm}{C}
\mc{\Csr}{C^{-1/2}}
\mc{\n}{d}
\mc{\I}{\mathrm{I}_\n}
\mc{\m}{m}
\mc{\x}{x}
\mc{\etam}{\eta_\mathrm{m}}
\mc{\etac}{\eta_\mathrm{c}}
\mc{\thetatt}{\theta^{t+1}}
\newcommand{\abs}[1]{\left|\mskip1mu#1\right|}
\newcommand{\norm}[1]{\|#1\|}

\newcommand{\eqigo}{\eqref{eq:IGOupdate}}
\newcommand{\refeq}[1]{\eqref{eq:#1}}
\mc{\back}{\hspace*{-5em}}%
\newcommand{\g}{\mathscr{G}}

\DeclareMathOperator*{\argmax}{arg\,max}

\mco[i]{\wi}{\widehat{w}_{#1}}
\mc{\wk}{\widehat{w}_k}

\newcommand{\tsum}{{\textstyle \sum}}

\DeclareMathOperator{\Cov}{Cov}
\DeclareMathOperator{\Var}{Var}

\newcommand{\deltat}{\ensuremath{\hspace{0.05em}\delta\hspace{-.06em}t\hspace{0.05em}}}

\newcommand{\E}{\mathbb{E}} 

\newcommand{\KL}[2]{\mathrm{KL}\!\left(#1 \,|\hspace{-.15ex}|\, #2\right)}
\DeclareMathOperator{\Ent}{Ent}

\newtheorem{thm}{Theorem}
\newtheorem{prop}[thm]{Proposition}
\newtheorem{lem}[thm]{Lemma}
\newtheorem{defi}[thm]{Definition}
\newtheorem{cor}[thm]{Corollary}

\newcommand{\notecolored}[3][]{\markupdraft{display}{{\color{#2}\noindent[Note (#1): #3]}}}
\newcommand{\newcolored}[3][]{{\markupdraft{color}{#2}#3}
    \ifthenelse{\equal{#1}{}}{}{\markupdraft{display}{{\color{yellow!70!black}[#1]}}}} 

\newcommand{\del}[2][]{{\markupdraft{display}{{\color{yellow!80!black}[removed: "#2"[#1]]}}}} 
\renewcommand{\del}[2][]{}
\newcommand{\new}[2][]{\newcolored[#1]{blue!80!black}{#2}}
\newcommand{\note}[2][]{\notecolored[#1]{green!75!black}{#2}}    
    
\newcommand{\NOTE}[2][]{\notecolored[#1]{magenta}{#2}}    
\newcommand{\TODO}[1][]{\markupdraft{todo}{{\color{red}\noindent==TODO: #1==}}}

\newcommand{\NDY}[1]{\NOTE[Yann]{#1}}

\newcommand{\niko}[1]{\notecolored[Niko]{green!50!blue}{#1}}

\jmlrheading{18}{2017}{1-65}{11/14; Revised 10/16}{4/17}{14-467}{Yann Ollivier, Ludovic Arnold, Anne Auger and Nikolaus Hansen}

\ShortHeadings{Information-Geometric Optimization}{Ollivier, Arnold, Auger and Hansen}
\firstpageno{1}

\begin{document}

\title{Information-Geometric Optimization Algorithms: A
Unifying Picture via Invariance Principles}

\author{\name Yann Ollivier \email yann.ollivier@lri.fr \\
       \addr CNRS \& LRI (UMR 8623), Universit{\'e} Paris-Saclay \\
       91405 Orsay, France
       \AND
       \name Ludovic Arnold \email arnold@lri.fr \\
       \addr Univ.\ Paris-Sud, LRI \\
       91405 Orsay, France
       \AND
       \name Anne Auger \email anne.auger@inria.fr\\
           \name Nikolaus Hansen \email nikolaus.hansen@inria.fr\\
       \addr Inria \& CMAP, Ecole polytechnique\\
       91128 Palaiseau, France
      }

\editor{Una-May O'Reilly}

\maketitle

\begin{abstract}
We present a canonical way to turn any smooth parametric family of
probability distributions on an arbitrary search space $X$ into a continuous-time black-box
optimization method on $X$, the \emph{information-geometric optimization} (IGO) method. Invariance as a major design principle
keeps the number of arbitrary choices to a minimum.
The resulting \emph{IGO flow} is the flow of an ordinary
differential equation
conducting the natural gradient
ascent of an adaptive, time-dependent transformation of the objective
function. It makes no particular assumptions on the objective function
to be optimized. 

The IGO method produces explicit IGO algorithms through time
discretization. It naturally recovers versions of known algorithms
and offers a systematic way to derive new ones. In
continuous search spaces, IGO algorithms take a form related to natural evolution strategies (NES). The cross-entropy method is recovered in a particular case with a large time step, and can be extended into a smoothed,
parametrization-independent maximum likelihood update (IGO-ML). 
When applied to the family of Gaussian distributions on
$\R^d$, the IGO framework recovers a version of the well-known CMA-ES algorithm and of xNES. For
the family of Bernoulli distributions on $\{0,1\}^d$, we recover the
seminal PBIL algorithm \new{and cGA}.  For the distributions of restricted Boltzmann
machines, we naturally obtain a novel algorithm for discrete optimization
on $\{0,1\}^d$. All these algorithms are natural instances of, and unified under, the single information-geometric optimization framework.

The IGO method achieves, thanks to its intrinsic formulation,
maximal invariance properties: invariance under reparametrization of the
search space $X$, under a change of parameters of the probability
distribution, and under increasing transformation of the function to be
optimized. The latter is achieved through an adaptive, quantile-based 
formulation of the objective.

Theoretical considerations strongly suggest that IGO algorithms are
essentially characterized by a minimal change of the distribution over
time. Therefore they have minimal loss
in diversity through the course of optimization, provided the initial diversity is high.
First experiments using restricted
Boltzmann machines confirm this insight. As a simple consequence, IGO seems to provide, from information theory, an
elegant way to \new{simultaneously}\del{spontaneously}
explore several valleys of a fitness landscape in a single run.
\end{abstract}


\section{Introduction}

%
Optimization problems are at the core of many disciplines. Given an objective
function $f: X\to\R$, to be optimized on some space $X$, the
goal of black-box optimization is to find solutions $x\in X$ with
small (in the case of minimization) value $f(x)$, using the least
number of calls to the function $f$. In a \emph{black-box} scenario,
 knowledge about the function $f$ is restricted to the
handling of a device (e.g., a simulation code) that delivers the value $f(x)$ for any input $x\in X$.
The search space $X$ may be finite, discrete infinite, or continuous.
However, optimization algorithms are often designed for a specific
type of search space, exploiting its specific structure.

One major design principle in general and in optimization in particular
is related to \emph{invariance}. Invariance extends performance
observed on a single function to an entire associated invariance class,
that is, it generalizes from a single problem to a class of problems. Thus it hopefully provides better robustness w.r.t.\ changes in
the presentation of a problem.
For continuous search spaces, invariance under
translation of the coordinate system is standard in optimization.
Invariance under general affine-linear changes of the coordinates 
has been---we believe---one of the keys to the success of 
the Nelder-Mead \emph{Downhill Simplex method} \citep{NelderMead:65}, 
of quasi-Newton methods \citep{deuflhard2011newton} and of the
\emph{covariance matrix adaptation evolution strategy}, CMA-ES
\citep{hansen2001,hansen2014principled}. 
While these relate to transformations in the search
space, another important invariance concerns the application of
monotonically increasing transformations to $f$, so that it is indifferent whether the function $f$, $f^{3}$ or $f\times |f|^{-2/3}$ is minimized. This way some
non-convex or non-smooth functions can be as ``easily'' optimised as convex ones. Invariance under $f$-transformation is not uncommon, 
e.g., for evolution strategies \citep{schwefel1995evolution} 
or pattern search methods \citep{Hooke:Jeeves:61,Torczon:97,NelderMead:65}; however
it has not always been recognized as an attractive feature.

Many stochastic optimization methods have been proposed to tackle
black-box optimization. The underlying (often hidden) principle of these
stochastic
methods is to iteratively update a probability distribution $P_\theta$
defined on $X$, parametrized by a set of parameters $\theta$. At a given
iteration, the distribution $P_{\theta}$ represents, loosely speaking,
the current belief about where solutions with the smallest values of the
function $f$ may lie. Over time, $P_{\theta}$ is expected to concentrate
around the minima of $f$. The update of the distribution involves
querying the function with points sampled from the current
probability distribution $P_{\theta}$. Although implicit in the
presentation of many stochastic optimization algorithms, this is the
explicit setting for  the wide family of 
\emph{estimation of distribution algorithms} (EDA)
\citep{larranaga2002estimation,Baluja:ICML1995,pelikan2002survey}. Updates
of the probability distribution often rely on heuristics
(nevertheless in \citealt{Toussaint2004} the possible interest of
information geometry to exploit the structure of probability
distributions for designing better grounded heuristics is pointed out). In addition, in the
EDA setting we can distinguish two theoretically founded approaches to
update $P_{\theta}$. 

First, the \emph{cross-entropy} method 
consists in  taking $\theta$ minimizing the Kullback--Leibler divergence
between $P_{\theta}$ and the indicator of the best points according to
$f$ \citep{CEMtutorial}.

Second, one can
transfer the objective function $f$ to the space of parameters $\theta$
by taking the average of $f$ under $P_{\theta}$, seen as a function of
$\theta$. This average is a
new function from an Euclidian space to $\R$ and is minimal when
$P_\theta$ is concentrated on minima of $f$. Consequently, $\theta$ can
be updated  by following a gradient
descent of this function with respect to $\theta$.
This has been done in various situations
such as $X=\{0,1\}^{d}$ and the family of Bernoulli measures
\citep{Berny:2000a} or of Boltzmann machines \citep{Berny2002}, or on
$X=\R^{d}$ for the family of Gaussian distributions \citep{Berny:Gaussian,Gallagher:2005}
or the exponential family also using second order information \citep{zhou2014gradient,zhou2017gradient}.

However, taking the ordinary gradient with respect to $\theta$
depends on the precise way a parameter $\theta$ is chosen to
represent the distribution $P_\theta$, and does not
take advantage of the Riemannian metric structure of families of
probability distributions. In the context of machine learning, Amari
noted the shortcomings of the ordinary gradient for families of probability
distributions \citep{Amari:1998} and proposed instead to use the natural
gradient with respect to the Fisher metric \citep{Rao1945,
Jeffreys1946,Amari2000book}. In the context of optimization, this natural
gradient has been used for exponential
families on $X=\{0,1\}^{d}$
\citep{MalagoGECCO2008InformationGeometryPerspectiveOfEDAs,
malago2011geometry} and for the family of Gaussian distributions on $X =
\R^{d}$ with so-called natural evolution strategies (NES,
\citealt{Wierstra2008, Sun:2009:ENE:1569901.1569976,Glasmachers2010,wierstra2014natural}). 

However, none of the previous attempts using gradient updates captures
the invariance under increasing transformations of the objective
function, which is instead, in some cases, enforced \emph{a posteriori} with heuristic arguments.

Building on these ideas, this paper overcomes the invariance problem of
previous attempts and provides a consistent, unified picture of optimization on arbitrary search spaces via invariance principles. More specifically, we
consider an arbitrary search space $X$, either discrete or continuous,
and a black-box optimization problem on $X$, that is, a function $f:X\to\R$ to be minimized. We assume that a
family of probability distributions $P_\theta$ on $X$ depending on a
continuous multicomponent parameter $\theta\in \Theta$ has been chosen. A
classical example is to take $X=\R^d$ and to consider the family of all
Gaussian distributions $P_\theta$ on $\R^d$, with $\theta=(m,C)$ the mean
and covariance matrix.  Another simple example is $X=\{0,1\}^d$ equipped
with the family of Bernoulli measures: $\theta=(\theta_i)_{1\leq i
\leq d}$ and $P_\theta(x)=\prod \theta_i^{x_i}(1-\theta_i)^{1-x_i}$ for
$x=(x_i)\in X$.

From this setting, \emph{information-geometric optimization} (IGO) can be defined in a natural way. At each
(continuous) time
$t$, we maintain a value $\theta^t$ of the parameter of the distribution. 
The function $f$ to be optimized is transferred to the parameter space
$\Theta$ by means of a suitable time-dependent transformation based on the
\emph{$P_{\theta^t}$-levels} of $f$ (Definition~\ref{def:quantiles}). The
\emph{IGO flow}, introduced in Definition~\ref{def:IGOflow}, follows the natural gradient of the expected value of this function of
$\theta^t$ in the parameter space $\Theta$, where
the natural gradient derives from the Fisher information metric.
The IGO flow is thus the flow of an ordinary differential equation in space
$\Theta$.
This continuous-time gradient flow is turned into a family of explicit
\emph{IGO algorithms} by
taking an Euler time discretization of the differential equation
and approximating the distribution $P_{\theta^t}$ by using samples.
From the start, the IGO flow is invariant under strictly increasing
transformations of $f$ (Proposition~\ref{prop:finv}); we also
prove that the sampling procedure is
consistent (Theorem~\ref{thm:consistency}). IGO algorithms share
their final algebraic form with the \emph{natural evolution strategies} (NES)
introduced in the Gaussian setting \citep{Wierstra2008,
Sun:2009:ENE:1569901.1569976,Glasmachers2010,wierstra2014natural}; the latter are thus recovered in the IGO framework
as an Euler approximation to a well-defined flow, without heuristic arguments.

The IGO method also has an equivalent description as an
\emph{infinitesimal maximum likelihood update} (Theorem~\ref{thm:IGOML});
this reveals a new property of the natural gradient and does not
require a smooth parametrization by $\theta$ anymore. This also
establishes a specific link (Theorem~\ref{thm:IGOCEM}) between IGO, the
natural gradient, and the
\emph{cross-entropy method} \citep{CEMtutorial}.

When we instantiate IGO
using the family of Gaussian
distributions on $\R^d$, we naturally obtain versions of the well-known
\emph{covariance matrix adaptation evolution strategy} (CMA-ES,
\citealt{hansen2001, hansen2004, jastrebski2006improving}) and of
\emph{natural evolution strategies}. With Bernoulli measures on the
discrete cube $\{0,1\}^d$, we recover (Proposition~\ref{prop:PBIL}) the well-known \emph{population-based incremental learning} (PBIL, \citealt{Baluja:ICML1995,Baluja:PBIL:1994})
and the \emph{compact genetic algorithm} (cGA, \citealt{harik1999compact}); 
this derivation of PBIL 
or cGA as
a natural gradient ascent appears to be new, and emphasizes
the common ground between continuous and discrete optimization.

From the IGO framework, it is (theoretically) immediate to build new optimization
algorithms using more complex families of distributions than Gaussian or
Bernoulli. As an illustration, distributions associated with restricted Boltzmann machines
(RBMs) provide a new but natural algorithm for discrete optimization on
$\{0,1\}^d$ which is able to handle dependencies between the bits (see also
\citealt{Berny2002}).  The probability distributions associated with RBMs
are multimodal; combined with the specific information-theoretic properties
of IGO that guarantee minimal change in diversity over time, this allows
IGO to reach multiple optima at once in a natural way, at least in a simple
experimental setup (Appendix~\ref{sec:RBM}). 

The IGO framework is built to achieve maximal \emph{invariance
properties}. First, the IGO flow is invariant under reparametrization of the
family of distributions
$P_\theta$, that is, it only depends on $P_\theta$ and not on the way we
write the parameter $\theta$ (Proposition~\ref{prop:thetainv}). This invariance under $\theta$-reparametrization is the main idea behind
\emph{information geometry} \citep{Amari2000book}. For instance, for Gaussian measures it
should not matter whether we use the covariance
matrix or its inverse or a Cholesky factor as
the parameter. This
limits the influence of encoding choices on the behavior of the
algorithm. Second, the IGO flow is invariant under a change of
coordinates in the search space $X$, provided that this change of coordinates globally
preserves the family of distributions $P_\theta$
(Proposition~\ref{prop:Xinv}).
For instance, for Gaussian
distributions on $\R^d$, this includes all affine changes of
coordinates. This means that the algorithm, apart from
initialization, does not depend on the
precise way the data is presented. Last, the IGO flow and IGO algorithms are invariant
under applying a strictly increasing function to $f$ (Proposition~\ref{prop:finv}).
Contrary to previous formulations using natural gradients
\citep{Wierstra2008,Glasmachers2010,Akimoto:ppsn2010}, this invariance is achieved from
the start. Such invariance
properties mean that we deal with \emph{intrinsic} properties of the
objects themselves, and not with the way we encode them as collections of
numbers in $\R^d$. It also means, most importantly, that we make a
minimal number of arbitrary choices.

\bigskip

In Section~\ref{sec:descr}, we define the IGO flow and the IGO algorithm.
We begin with standard facts about the definition and basic properties of
the natural gradient, and its connection with Kullback--Leibler
divergence and diversity. We then proceed to the detailed description of
the algorithm.

In Section~\ref{sec:properties}, we state some first mathematical
properties of IGO. These include monotone improvement of the objective
function, invariance properties, and the form of IGO for exponential families
of probability distributions.\del{, and the case of noisy objective
functions.}

In Section~\ref{sec:IGOML} we explain the theoretical relationships
between IGO, maximum likelihood estimates and the cross-entropy method.
In particular, IGO is
uniquely characterized by a weighted log-likelihood maximization
property (Theorem~\ref{thm:IGOML}).

In Section~\ref{sec:CMA-PBIL-IGO}, we apply the IGO framework to 
explicit families of probability distributions. Several 
well-known optimization algorithms are recovered this way. These include PBIL
(Sec.~\ref{sec:PBIL}) using Bernoulli distributions, and
versions of CMA-ES and other evolutionary algorithms such as
EMNA and xNES (Sec.~\ref{sec:gaussian}) using Gaussian 
distributions.

\new{Appendix~\ref{sec:discussion} discusses further aspects, perspectives and implementation matters of the IGO framework.}
In Appendix~\ref{sec:RBM}, we illustrate how IGO can be used to design new
optimization algorithms, by deriving the IGO
algorithm associated with restricted Boltzmann machines. We illustrate
experimentally, on a simple bimodal example, 
the specific influence of the Fisher information
matrix on the optimization
trajectories, and in particular on the diversity of the optima
obtained.
\new{Appendix~\ref{sec:moremath} details}\del{
The Appendix also contains further discussion, details for} a number of
further mathematical properties of IGO (such as its invariance properties
or the case of noisy objective functions). \new{Appendix~\ref{sec:proofs} contains the previously omitted}\del{, as well as the} proofs of the
mathematical statements.

\section{Algorithm Description}
\label{sec:descr}

We now present the outline of the algorithm. Each step is described in
more detail in the sections below.

The IGO flow can be seen as an \emph{estimation of
distribution algorithm}: at each time $t$, we maintain a probability
distribution $P_{\theta^t}$ on the search space $X$, where $\theta^t\in
\Theta$. The value of $\theta^t$ will evolve so that, over time,
$P_{\theta^t}$ gives more weight to points $x$ with better values of the
function $f(x)$ to optimize.

A straightforward way to proceed is to transfer $f$ from $x$-space to $\theta$-space: define a function $F(\theta)$ as the
$P_\theta$-average of $f$ and then do a gradient descent for
$F(\theta)$ in space $\Theta$ \citep{Berny:2000a,Berny2002,Berny:Gaussian,Gallagher:2005}.
This way, $\theta$ will converge to a point such that $P_\theta$ yields a
good average value of $f$. We depart from this approach in two ways:
\begin{itemize}
\item At each time, we replace $f$ with an adaptive transformation of $f$
representing how good or bad observed values of $f$ are \emph{relative to other
observations}. This provides invariance under all monotone transformations of
$f$.
\item Instead of the vanilla gradient for $\theta$, we use the so-called
\emph{natural gradient} given by the Fisher information matrix.
This reflects
the intrinsic geometry of the space of probability distributions,
as introduced by Rao and Jeffreys \citep{Rao1945, Jeffreys1946} and later
elaborated upon by Amari and others \citep{Amari2000book}. 
This provides
invariance under reparametrization of $\theta$ and, importantly,
minimizes the change of diversity of $P_\theta$.
\end{itemize}

The algorithm is constructed in two steps: we first give an ``ideal'' version, namely, a version in
which time $t$ is continuous so that the evolution of $\theta^t$ is given
by an ordinary differential equation in $\Theta$. Second, the actual algorithm
is a time discretization using a finite time step and Monte Carlo
sampling instead of exact $P_\theta$-averages.

\subsection{The Natural Gradient on Parameter Space}
\label{sec:grad}
\new{We recall suitable definitions of the vanilla and the natural gradient
and motivate using the natural gradient in the context of optimization. }

\subsubsection{About Gradients and the Shortest Path Uphill} Let $g$ be a smooth function from $\R^d$ to $\R$,
to be maximized. We first recall the interpretation of gradient ascent as
``the shortest path uphill''.

Let $y\in \R^d$. Define the vector $z$ by
\begin{equation}\label{eq:grad-def-const-norm}
z=\lim_{\epsilon\to 0} \argmax_{z,\,\norm{z}\leq
1} g(y+\eps z)\enspace.
\end{equation}
Then one can check that $z$ is the normalized gradient of $g$ at $y$:
$z_i=\frac{\partial g/\partial y_i}{\norm{\partial g/\partial y}}$.
(This holds only at points $y$ where the gradient of $g$ does not vanish.)


This shows that, for small $\deltat$, the well-known gradient ascent of $g$ given by
\[
y^{t+\deltat}_i=y^t_i+\deltat\,\tfrac{\partial g}{\partial y_i}
\]
realizes the largest increase of
the value of $g$, for a given step size $\norm{y^{t+\deltat}-y^t}$.


The relation \eqref{eq:grad-def-const-norm} depends on the choice of a
norm $\norm{\!\cdot\!}$ (the gradient of $g$ is given by $\partial g/\partial y_i$
only in an orthonormal basis).
If we use, instead of the
standard metric $\norm{y-y'}=\sqrt{\sum (y_i-y'_i)^2}$ on $\R^d$, a
metric $\norm{y-y'}_A=\sqrt{\sum A_{ij} (y_i-y'_i)(y_j-y'_j)}$ defined by
a positive definite matrix $A_{ij}$, then the gradient of $g$ with
respect to this metric is given by $\sum_j A^{-1}_{ij}
\frac{\partial g}{\partial y_i}$. This follows from the textbook
definition of 
gradients by $g(y+\eps
z)=g(y)+\eps\langle \nabla g,z\rangle_A +O(\eps^2)$ with
$\langle\cdot,\cdot\rangle_A$ the
scalar product associated with the matrix $A_{ij}$ \citep{SchwartzAnalyseII}.


It is possible to write the analogue of \eqref{eq:grad-def-const-norm}
using the
$A$-norm.
We then find that the gradient ascent associated with metric $A$ is given by
\[
y^{t+\deltat}=y^t+\deltat\,A^{-1}\,\tfrac{\partial g}{\partial y}
\]
for small $\deltat$ and
maximizes the increment of $g$ for a given $A$-distance
$\norm{y^{t+\deltat}-y^t}_A$---it realizes the steepest $A$-ascent.
Maybe this viewpoint clarifies the relationship between gradient and
metric: this steepest ascent property can actually be used as a
definition of gradients.

In our setting we want to use a gradient ascent in the parameter space
$\Theta$ of our distributions $P_\theta$. The ``vanilla'' gradient
$\frac{\partial }{\partial \theta_i}$ is associated with the metric
$\norm{\theta-\theta'}=\sqrt{\sum (\theta_i-\theta'_i)^2}$ and clearly
depends on the choice of parametrization $\theta$. Thus this metric,
and the
direction pointed by this gradient, are not
intrinsic, in the sense that they do not depend only on the
\emph{distribution} $P_\theta$. A metric depending on $\theta$ only
through the distributions $P_\theta$ can be defined as follows.

\subsubsection{Fisher Information and the Natural Gradient on Parameter Space}
Let $\theta,\theta'\in\Theta$ be two values of the distribution
parameter. A widely used way to define a ``distance'' between two
generic distributions\footnote{Throughout the text we \new{do not distinguish}\del{will conflate} a
probability distribution $P$, seen as a measure, and its density
with respect to some unspecified reference measure $\d x$, and so will
write indifferently $P(\d x)$ or $P(x)\d x$. The
measure-theoretic viewpoint allows for a unified treatment of the
discrete and continuous case.} $P_\theta$ and $P_{\theta'}$ is the
\emph{Kullback--Leibler divergence} from information theory, defined \citep{Kullback} as 
\[
\KL{P_{\theta'}}{P_\theta}=\int_x
\ln\frac{P_{\theta'}(\d x)}{P_\theta(\d x)}\,P_{\theta'}(\d x)\enspace.
\]
When $\theta'=\theta+\delta\theta$ is close to $\theta$, under mild
smoothness assumptions we can expand the
Kullback--Leibler divergence at second order in $\delta\theta$. This
expansion defines the
Fisher information matrix $I$ at $\theta$ \citep{Kullback}:
\begin{equation}
\label{eq:IKL}
\KL{P_{\theta+\delta\theta}}{P_\theta}=
\frac12 \sum I_{ij}(\theta)\,\delta\theta_i \delta\theta_j
+ O(\delta\theta^3)\enspace.
\end{equation}
An equivalent definition of the
Fisher information matrix is by the usual
formulas \citep{CoverThomas}
\[
I_{ij}(\theta)=\int_x \frac{\partial \ln P_\theta(x)}{\partial
\theta_i}\frac{\partial \ln P_\theta(x)}{\partial \theta_j}\,P_\theta(\d x)
=-\int_x \frac{\partial^2 \ln P_\theta(x)}{\partial\theta_i \, \partial \theta_j} \,
P_\theta(\d x)
\enspace.
\]

The Fisher information matrix defines a (Riemannian) metric on $\Theta$:
the distance, in this metric, between two very
close values of $\theta$ is given by the square root of twice the Kullback--Leibler divergence.
Since the Kullback--Leibler divergence depends only on $P_\theta$ and not
on the parametrization of $\theta$, this metric is intrinsic.

If $g:\Theta\to \R$ is a smooth function on the parameter
space, its \emph{natural gradient} 
\citep{Amari:1998} at $\theta$ is defined in accordance with the Fisher metric as
\[
(\widetilde\nabla_{\!\theta} \, g)_i=\sum_j I^{-1}_{ij}(\theta)\,\frac{\partial g(\theta)}{\partial
\theta_j}
\]
or more synthetically
\[
\widetilde\nabla_{\!\theta}\, g=I^{-1} \, \frac{\partial g}{\partial \theta}\enspace.
\]
\mco[]{\vanilla}{\frac{\partial#1}{\partial \theta}}%
From now on, we will use $\widetilde\nabla_{\!\theta}$ to denote the natural gradient
and $\frac{\partial}{\partial \theta}$ to denote the vanilla gradient.

By construction, the natural gradient descent is
intrinsic: it does not depend on the chosen parametrization $\theta$ of
$P_\theta$, so that it makes sense to speak of the natural gradient
ascent of a function $g(P_\theta)$.
The Fisher metric is essentially the only way to obtain this
property \cite[Section 2.4]{Amari2000book}.

Given that the Fisher metric comes from the Kullback--Leibler
divergence, the ``shortest path uphill'' property of gradients mentioned
above translates as follows (see also \citealt[Theorem~1]{Amari:1998}):
\begin{prop}
\label{prop:KLuphill}
The natural gradient ascent points in the direction $\delta\theta$
achieving the largest change of the objective function, for a given
distance between $P_\theta$ and $P_{\theta+\delta\theta}$ in
Kullback--Leibler
divergence. More precisely,
let $g$ be a smooth function on the parameter space $\Theta$. Let
$\theta\in\Theta$ be a point where $\widetilde\nabla g(\theta)$ does not
vanish. Then, if
\[
\delta\theta=
\frac{\widetilde\nabla
g(\theta)}{\rule{0cm}{2ex}\norm{\widetilde\nabla g(\theta)}}
\]
is the direction of the natural gradient of $g$ (with $\norm{\cdot}$ the
Fisher norm), we have
\[
\delta\theta=\lim_{\eps\to 0} \frac1\eps
\argmax_{
\begin{subarray}{c}
\delta\theta'\text{such that}
\\
\KL{P_{\theta+\delta\theta'}}{P_{\theta}}\leq \eps^2/2
\end{subarray}
}
g(\theta+\delta\theta').
\]
\end{prop}

Here we have implicitly assumed that the parameter space $\Theta$ is
such that no two points $\theta\in\Theta$ define the same
probability distribution, and the mapping $P_\theta\mapsto\theta$
is smooth.

\subsubsection{Why Use the Fisher Metric Gradient for Optimization?
Relationship to Diversity}
The first reason
for using the natural gradient is its reparametrization invariance, which
makes it the only gradient available in a general abstract setting
\citep{Amari2000book}.
Practically, this invariance also limits the influence of encoding
choices on the behavior of the algorithm (Appendix~\ref{sec:invariance}).
\del{Finally t}\new{T}he Fisher matrix can be also seen as an \emph{adaptive
learning rate} for different components of the parameter vector
$\theta_i$: components $i$ with a high impact on $P_\theta$ will be
updated more cautiously.

Another advantage comes from the relationship with Kullback--Leibler
distance in view of the ``shortest path uphill'' (see also
\citealt{Amari:1998}).
To
minimize the value of some function $g(\theta)$ defined on the parameter
space $\Theta$, the naive approach follows a gradient descent for
$g$ using the ``vanilla'' gradient
\[
\theta_i^{t+\deltat}=\theta_i^t+\deltat \tfrac{\partial g}{\partial
\theta_i}
\]
and, as explained above, this maximizes the increment of $g$ for a given
increment $\norm{\theta^{t+\deltat}-\theta^t}$. On the other hand, the Fisher
gradient
\[
\theta_i^{t+\deltat}=\theta_i^t+\deltat I^{-1} \tfrac{\partial
g}{\partial \theta_i}
\]
maximizes the increment of $g$ for a given Kullback--Leibler distance
$\KL{P_{\theta^{t+\deltat}}}{P_{\theta^t}}$.

In particular, if we choose an initial value $\theta^0$ such that
$P_{\theta^0}$ covers the whole space $X$ uniformly (or a wide portion,
in case $X$ is unbounded), the
Kullback--Leibler divergence between $P_{\theta^t}$ and $P_{\theta^0}$ is
the Shannon entropy of the uniform distribution minus the Shannon entropy
of $P_{\theta^t}$, and so this divergence measures the loss of diversity
of $P_{\theta^t}$ with respect to the uniform distribution.

\begin{prop}
\label{prop:maxent}
Let $g:\Theta\to\R$ be a regular function of $\theta$ and let $\theta^0$
such that $P_{\theta^0}$ is the uniform distribution on a finite space
$X$. Let $(\theta^t)_{t\geq 0}$ be the trajectory of the gradient ascent
of $g$ using the natural gradient. Then for small $t$ we have
\begin{equation}
\theta^t=\argmax_\theta \left\{
t\cdot g(\theta)+\Ent(P_\theta)
\right\}+o(t)
\end{equation}
where $\Ent$ is the Shannon entropy.
\end{prop}

(We have stated the proposition over a finite space to have a well-defined
uniform distribution. A short proof, together with the regularity
conditions on $g$, is given in Appendix~\ref{sec:proofs}.)

So following the natural gradient of a function $g$, starting at or close
to the uniform distribution, amounts to optimizing the function $g$ with \emph{minimal loss of
diversity, provided the initial diversity is large}. (This is valid, of
course, only at the beginning; once one gets too far from uniform, a
better interpretation is minimal \emph{change} of diversity.) On the other hand,
the vanilla gradient
descent does not satisfy Proposition~\ref{prop:maxent}: it optimizes $g$ with minimal change in the numerical values of the
parameter $\theta$, which is of little interest.

So arguably this method realizes the best trade-off between optimization
and loss of diversity. (Though, as can be seen from the detailed
algorithm description below, maximization of diversity occurs only
greedily at each step, and so there is no guarantee that after a given
time, IGO will provide the highest possible diversity for a given
objective function value.)

An experimental confirmation of the positive influence of the Fisher
matrix on diversity is given in Appendix~\ref{sec:RBM} below.

\subsection{IGO: Information-Geometric Optimization}
\label{sec:igo}
\new{We now introduce a quantile-based rewriting of the objective function. From applying the natural gradient on the rewritten objective we derive \emph{information-geometric optimization}. }

\subsubsection{Quantile Rewriting of $f$} Our original problem is to
minimize a function $f:X\to \R$. A simple way to turn $f$ into a function
on $\Theta$ is to use the expected value $-\E_{P_\theta} f$
\citep{Berny:2000a,Wierstra2008}, but expected values can be unduly
influenced by extreme
values and using them can be rather unstable \citep{whitley1989genitor};
moreover $-\E_{P_\theta} f$ is not invariant under increasing transformation
of $f$ (this invariance implies we can only \emph{compare} $f$-values, not sum them up).

Instead, we take an adaptive, quantile-based approach by first
replacing the function $f$ with
a monotone rewriting $W_{\theta^t}^f$, depending on the current parameter
value $\theta^t$, and then following the gradient of
$\E_{P_\theta} W_{\theta^t}^f$, seen as a function of $\theta$. A due
choice of $W_{\theta^t}^f$ allows to control
the range of the resulting values and achieves the desired invariance.
Because the rewriting $W_{\theta^t}^f$ depends on $\theta^t$, it might be viewed as an \emph{adaptive}
$f$-transformation.

The monotone rewriting entails that
if $f(x)$ is ``small'' then $W_\theta^f(x)\in\R$ is
``large'' and vice versa.  The quantitative meaning of ``small'' or ``large''
depends on $\theta\in \Theta$. To obtain the value of $W_\theta^ f(x)$ we compare 
$f(x)$ to the quantiles
of $f$ under the current distribution, as
measured by the
$P_\theta$-level fraction in which the value of $f(x)$ lies.

\begin{defi}
\label{def:quantiles}
The lower and upper $P_\theta$-$f$-levels of $x\in X$ are defined as
\begin{equation}
\label{eq:quantiles}
\begin{aligned}
q_\theta^<(x) &= \Pr\nolimits_{x'\sim P_\theta}(f(x')<f(x))
\\
q_\theta^\leq(x) &= \Pr\nolimits_{x'\sim P_\theta}(f(x')\leq f(x))\enspace.
\end{aligned}
\end{equation}

Let $w:[0;1]\to \R$ be a non-increasing function, the \emph{selection
scheme}.

The transform $W_\theta^f(x)$ of an objective function $f:X\to \R$ is defined as 
a function of the $P_\theta$-$f$-level of $x$ as
\begin{equation}\label{eq:f-replacement}
W_\theta^f(x) =
\begin{cases}
w(q_\theta^\leq(x))
&\text{if }
q_\theta^\leq(x)=q_\theta^<(x),
\\
\frac{1}{q_\theta^\leq(x)-q_\theta^<(x)}
\int_
{q=q_\theta^<(x)}
^
{q=q_\theta^\leq(x)} w(q)\,\d q
&\text{otherwise.}
\end{cases}
\end{equation}
\end{defi}

The \del{distribution}\new{level} functions \del{$q:\R\to[0,
1]$}\new{$q:X\to [0,1]$} reflect the probability to sample a better value than $f(x)$. They are monotone in $f$ (if $f(x_{1})
\leq f(x_{2})$ then $q_{\theta}^<(x_{1}) \leq q_{\theta}^<(x_{2})$, and
likewise for $q^\leq$) and invariant under
strictly increasing transformations of $f$.

A typical choice for $w$ is $w(q)=\mathbbm{1}_{q\leq q_0}$ for some fixed
value $q_0$, the \emph{selection quantile}. In what follows, we suppose
that a selection scheme \new{(weighting scheme) $w$} has been chosen once and for all.

As desired, the definition of $W_{\theta}^f$ is invariant under a
strictly increasing transformation of $f$.
For instance, the $P_\theta$-median of $f$ gets remapped to $w(\frac12)$.

Note that $\E_{x\sim P_\theta} W_\theta^f(x)$ is always equal to
$\int_0^1 w$, independently of $f$ and
$\theta$: indeed, by definition,
the $P_\theta$ inverse-quantile of a random point under $P_\theta$ is uniformly distributed
in $[0;1]$. In the following, our objective will be to maximize the
expected value of $W_\thetat^f$ over $\theta$, that is, to maximize  
\begin{equation}\label{eq:objective}
\E_{P_\theta} W_\thetat^f = \int W_{\theta^t}^f(x)\, P_{\theta}(\d x)
\enspace
\end{equation}
over $\theta$,
where $\thetat$ is fixed at a given step but will adapt over time.

Importantly, $W_\theta^f(x)$ can be estimated in practice: indeed, the 
$P_\theta$-$f$-levels $\Pr_{x'\sim P_\theta}(f(x')<f(x))$ can be estimated by taking
samples of $P_\theta$ and ordering the samples according to the value of
$f$ (see below). The estimate remains invariant under strictly increasing
$f$-transformations.

\subsubsection{The IGO Gradient Flow} At the most abstract level,
IGO is a continuous-time gradient flow in the parameter space $\Theta$,
which we define now. In practice, discrete time steps (a.k.a.\ iterations) are
used, and $P_\theta$-integrals are approximated through sampling, as 
described in the next section.


Let $\theta^t$ be the current value of the parameter at time $t$, and let
$\deltat\ll 1$. We define $\theta^{t+\deltat}$ in such a way as to
increase the $P_\theta$-weight of points where $f$ is small, while not
going too far from $P_{\theta^t}$ in Kullback--Leibler divergence. We use
the adaptive weights
$W_{\theta^t}^f$ as a way to measure which points have large or small
values. In accordance with \eqref{eq:objective}, this suggests taking the gradient ascent
\begin{equation}\label{eq:IGOexactupdate}
\theta^{t+\deltat}=\theta^t+\deltat \, \widetilde\nabla_{\!\theta} \int
W_{\theta^t}^f(x)\, P_{\theta}(\d x)
\end{equation}
where the natural gradient is suggested by Proposition~\ref{prop:KLuphill}.

Note again that we use $W_{\theta^t}^f$ and not $W_{\theta}^f$ in the
integral. So the gradient $\widetilde\nabla_{\!\theta}$ does not act on
the adaptive objective $W_{\theta^t}^f$. If we used
$W_{\theta}^f$ instead, we would face a paradox: right after a move,
previously good points do not seem so good any more since the
distribution has improved.  More precisely, $\int
W_{\theta}^f(x)\, P_{\theta}(\d x)$ is constant and always equal to the
average weight $\int_0^1 w$, and so the gradient would always vanish.

Using the log-likelihood trick $\widetilde\nabla P_\theta=P_\theta\,
\widetilde\nabla \!\ln P_\theta$ (assuming $P_\theta$ is smooth), we get an
equivalent expression of the update above as an integral under the
current distribution $P_{\theta^t}$; this is important for practical
implementation.
This leads to the following definition.

\begin{defi}[IGO flow]
\label{def:IGOflow}
The IGO flow is the set of continuous-time trajectories in space
$\Theta$, defined by the ordinary differential equation
\begin{align}\label{eq:IGOflow}
\frac{\d \theta^t}{\d t}
&=\widetilde\nabla_{\!\theta} \int W_{\theta^t}^f(x)\, P_{\theta}(x) \d x
\\\nonumber&=
\int
W_{\theta^t}^f(x)\,
\,\frac{\widetilde\nabla_{\!\theta} P_\theta(x)}{P_{\theta^t}(x)} P_{\theta^t}(x) \d x
\\&=
\label{eq:IGOsemiexplicit}
\int
W_{\theta^t}^f(x)\,
\widetilde\nabla_{\!\theta} \ln P_\theta(x) 
\,P_{\theta^t}(\d x)
\\&=\label{eq:IGOexplicit}
I^{-1}(\theta^t)\,
\int W_{\theta^t}^f(x)\,\frac{\partial \ln P_\theta(x)}{\partial
\theta} 
  \,P_{\theta^t}(\d x)\enspace.
\end{align}
where the gradients are taken at point $\theta=\theta^t$, and $I$ is the
Fisher information matrix.
\end{defi}

Natural evolution
strategies (NES, \citealt{Wierstra2008,Glasmachers2010,Sun:2009:ENE:1569901.1569976,wierstra2014natural}) feature
a related gradient \emph{descent} with $f(x)$ instead of $W_{\theta^t}^f(x)$.
The associated flow would read
\begin{equation}\label{eq:IGOflowf}
\frac{\d \theta^t}{\d t}=-\widetilde\nabla_{\!\theta} 
\int f(x)\, P_{\theta}(\d x) \enspace,
\end{equation}
where the gradient is taken at $\theta^{t}$ (in the sequel when not
explicitly stated, gradients in $\theta$ are taken at $\theta=\theta^{t}$).
However, in the end NESs always implement algorithms using
sample quantiles (via ``nonlinear \textit{fitness shaping}''), 
as if derived from the gradient ascent of $W_{\theta^t}^f(x)$.

The update
\eqref{eq:IGOsemiexplicit} is a weighted average of
``intrinsic moves'' increasing the log-likelihood of some points. We can
slightly rearrange the update as
\begin{align}
\frac{\d\theta^t}{\d t}
&=
\int \overbrace{W_{\theta^t}^f(x)}^{\back\text{preference
weight~~~}\back}\, \underbrace{\gradt \ln
P_\theta(x)}_{\back\text{intrinsic move to reinforce $x$
}\back} 
\,\overbrace{P_{\theta^t} (\d x)}^{\back\text{~~~current sample distribution}\back}\\
&=
\gradt  \int \underbrace{W_{\theta^t}^f(x)\ln
P_\theta(x)}_{\back\text{weighted log-likelihood}\back} 
\,P_{\theta^t} (\d x)\enspace,
\end{align}
which provides an interpretation for the IGO gradient
flow as a gradient ascent optimization of the weighted log-likelihood of the ``good
points'' of the current distribution. In the sense of Theorem~\ref{thm:IGOML} below,
IGO is in fact the ``best'' way to increase this log-likelihood.


For exponential families of probability distributions, we will see later
that the IGO flow rewrites as a nice derivative-free expression
\eqref{eq:expIGO}.

\subsubsection{IGO Algorithms: Time Discretization and Sampling} The
above is a mathematically well-defined continuous-time flow in
parameter space.  Its practical implementation involves three
approximations depending on two parameters $N$ and $\deltat$:
\begin{itemize}
\item the integral under $P_{\theta^t}$ is approximated using $N$ samples
taken from $P_{\theta^t}$; 
\item the value $W_{\theta^t}^f$ is approximated for each sample taken from $P_{\theta^t}$;
\item the time derivative $\frac{\d \theta^t}{\d t}$ is approximated
by a $\deltat$ time increment\new{: instead of the continuous-time IGO flow
\eqref{eq:IGOflow} we use its Euler approximation scheme
$\theta^{t+\deltat}\approx \theta^t+\deltat \frac{\d \theta^t}{\d t}$, so
that
the time $t$ of the flow is discretized with a step size
$\deltat$, which thus becomes the learning rate of the algorithm. (See Corollary~\ref{cor:KLspeed} for an interpretation of
$\deltat$ as a number of bits of information introduced in the
distribution $P_{\theta^t}$ at each step.})
\end{itemize}

We also assume that the Fisher information matrix $I(\theta)$ and
$\frac{\partial \ln P_\theta(x)}{\partial \theta}$ can be computed (see
discussion below if $I(\theta)$ is unknown).

\newcommand{\rk}{\mathrm{rk}}

At each step, we draw $N$ samples $x_1,\ldots,x_N$ under $P_{\theta^t}$.
To approximate $W_{\theta^t}^f$, we rank the samples
according to the value of $f$.
Define $\rk (x_i)= \#\{j\,| \, f(x_j)< f(x_i)\}$ and let the estimated
weight of sample $x_i$ be
\begin{equation}\label{eq:wi}
   \wi=\frac1N\,w\left(\frac{\rk (x_i) + 1/2 }{N}\right),
\end{equation}
using the non-increasing selection scheme function $w$
introduced in Definition~\ref{def:quantiles} above.
(This is assuming there are no ties in our sample; in case several sample
points have the same value of $f$, we define $\wi$ by averaging the
above
over all possible rankings of the
ties\footnote{A mathematically neater but less intuitive version would be
\begin{equation}
\label{eq:tie-break}
\wi=\frac1{\rk^\leq
(x_i)-\rk^<(x_i)}\int_{u=\rk^<(x_i)/N}^{u=\rk^\leq(x_i)/N} w(u)\d u
\end{equation}
with $\rk^<(x_i)=\#\{j \,| \, f(x_j)< f(x_i)\}$ and $\rk^\leq(x_i)=\#\{j
\,|\,
f(x_j)\leq f(x_i)\}$.
}.)
%

Then we can approximate the IGO flow as follows.

\begin{defi}[IGO algorithms]
\label{def:IGOalgo}
The \emph{IGO algorithm} associated with parametrization $\theta$, sample size $N$ and step size $\deltat$
is the following update rule for the parameter $\theta^t$.
At each step, $N$ sample points $x_1,\ldots,x_N$ are drawn according to
the distribution $P_{\theta^t}$. The parameter is updated according to
\begin{align}
\theta^{t+\deltat}
&=
\theta^t+\deltat \sum_{i=1}^{N} \wi \, \left.\widetilde\nabla_\theta \ln P_\theta(x_i)\right|_{\theta=\theta^t}
\label{eq:intrinsicIGOupdate}
\\&=
\theta^t+\deltat \, I^{-1}(\theta^t)\,
\sum_{i=1}^{N} \wi \, \left.\frac{\partial \ln P_\theta(x_i)}{\partial
\theta}\right|_{\theta=\theta^t}
\label{eq:IGOupdate}
\end{align}
where $\wi$ is the weight \eqref{eq:wi} obtained from the ranked values of the
objective function $f$.
\end{defi}

Equivalently one can fix the weights
$w_i=\frac1N\,w\left(\frac{i-1/2}{N}\right)$ once and for all and rewrite
the update as
\begin{equation}\label{eq:IGOupdatebis}
\theta^{t+\deltat}=\theta^t+\deltat \, I^{-1}(\theta^t)\,
\sum_{i=1}^{N} w_i \, \left.\frac{\partial \ln P_\theta(x_{i:N})}{\partial
\theta}\right|_{\theta=\theta^t}
\end{equation}
where $x_{i:N}$ denotes the sampled point ranked \new{$i^\text{th}$} according
to $f$, i.e. $f(x_{1:N}) < \ldots < f(x_{N:N})$ (assuming again
there are no ties). Note that $\{x_{i:N}\} = \{x_i\}$ and $\{w_i\} = \{\wi\}$.

As will be discussed in Section~\ref{sec:CMA-PBIL-IGO}, this update
applied to multivariate normal distributions or Bernoulli measures allows
us
to neatly recover versions of some well-established algorithms, in
particular CMA-ES and PBIL. Actually, in the Gaussian context 
updates of the form \eqref{eq:IGOupdate} have already been introduced
\citep{Glasmachers2010,Akimoto:ppsn2010}, though not formally derived from
a continuous-time flow with inverse quantiles.
 
When $N\to\infty$, the IGO algorithm using samples approximates
the continuous-time IGO gradient flow, see Theorem~\ref{thm:consistency}
below. Indeed, the IGO algorithm, with
$N=\infty$, is simply the
Euler approximation scheme for the ordinary differential equation
defining the IGO flow \eqref{eq:IGOflow}. The latter result thus provides a sound mathematical
basis for currently used rank-based updates.

\subsubsection{IGO Flow versus IGO Algorithms} 
The IGO \emph{flow}
\eqref{eq:IGOflow} is a well-defined continuous-time set of
trajectories in the space
of probability distributions $P_\theta$, depending only on the objective
function $f$ and the chosen family of distributions. It
does not depend on the
chosen parametrization for $\theta$ (Proposition~\ref{prop:thetainv}).

On the other hand, there are several IGO \emph{algorithms} associated
with this flow. Each IGO algorithm approximates the IGO
flow in a slightly different way. An IGO algorithm depends on three further choices:
a sample size $N$, a time discretization step size $\deltat$, and a
choice of parametrization for $\theta$ in which to
implement~\eqref{eq:IGOupdate}. 

If $\deltat$ is small enough, and $N$ large enough, the influence of the
parametrization $\theta$ disappears and all IGO algorithms are
approximations of the ``ideal'' IGO flow trajectory. 
However, the larger $\deltat$, the poorer the approximation gets.

So for large
$\deltat$, different IGO algorithms for the same IGO flow may exhibit
different behaviors. We will see an instance of this phenomenon for
Gaussian distributions: both CMA-ES and the maximum likelihood
update (EMNA) can be seen as IGO algorithms, but the latter with
$\deltat=1$ is known to exhibit premature loss of diversity
(Section~\ref{sec:gaussian}).

Still, if $\deltat$ is sufficiently small, 
two IGO algorithms for the same IGO flow will differ less from
each other than from a non-IGO algorithm: at each step the difference is
only
$O(\deltat^2)$ (Appendix~\ref{sec:invariance}). On the other hand, for instance, the difference between an IGO algorithm and
the vanilla gradient ascent is, generally, not smaller than $O(\deltat)$ at each step, i.e., it can be roughly as big as the
steps themselves.

\subsubsection{Unknown Fisher Matrix} The \label{unknownfisher} algorithm presented so far assumes
that the Fisher matrix $I(\theta)$ is known as a function of $\theta$.
This is the case for Gaussian \new{or}\del{distributions in CMA-ES and for} Bernoulli
distributions. However, for restricted Boltzmann machines as considered
below, no analytical form is known. Yet, provided the quantity
$\frac{\partial}{\partial \theta} \ln P_\theta(x)$ can be computed or
approximated, it is possible to approximate the integral
\[
I_{ij}(\theta)=\int_x \frac{\partial \ln P_\theta(x)}{\partial
\theta_i}\frac{\partial \ln P_\theta(x)}{\partial \theta_j}\,
P_\theta(\d x)
\]
using $P_\theta$-Monte Carlo samples for $x$. These samples may or may not be the
same as those used in the IGO update \eqref{eq:IGOupdate}: in particular,
it is possible to use as many Monte Carlo samples as necessary to
approximate $I_{ij}$, at no additional cost in terms of the
number of calls to the black-box function $f$ to optimize.

Note that each Monte Carlo sample $x$ will contribute
$\frac{\partial \ln P_\theta(x)}{\partial
\theta_i}\frac{\partial \ln P_\theta(x)}{\partial \theta_j}$ to the
Fisher matrix approximation. This is a rank-$1$ non-negative
matrix\footnote{The alternative, equivalent formula
$I_{ij}(\theta)=-\int_x \frac{\partial^2
\ln P_\theta(x)}{\partial\theta_i \, \partial \theta_j} \,
P_\theta(\d x)$ for the Fisher matrix will not necessarily yield
non-negative matrices through Monte Carlo sampling.}. So, for the
approximated Fisher matrix to be invertible, the number of
(distinct) samples $x$ needs to be at least equal to, and ideally
much larger than, the number of
components of the parameter $\theta$: $N_\text{Fisher}\geq \dim \Theta$.

For exponential families of distributions, the IGO update has a
particular form \eqref{eq:expIGO} which simplifies this matter somewhat
(Section~\ref{sec:exp}).
In Appendix~\ref{sec:RBM} this is examplified using restricted
Boltzmann machines.

\section{First Properties of IGO}
\label{sec:properties}
\new{In this section we derive some basic properties of IGO and present the IGO flow for exponential families. }
\subsection{Consistency of Sampling}

The first property to check is that
when $N\to\infty$, the update rule using $N$ samples converges to the IGO
update rule. This is \emph{not} a straightforward application of the law
of large numbers, because the estimated weights $\wi$ depend
(non-continuously) on the \emph{whole} sample $x_1,\ldots,x_N$, and not only on
$x_i$.

\begin{thm}[Consistency]
\label{thm:consistency}
When $N\to\infty$, the $N$-sample IGO update rule \eqref{eq:IGOupdate}:
\[
\theta^{t+\deltat}=\theta^t+\deltat \, I^{-1}(\theta^t)\,
\sum_{i=1}^{N} \wi \, \left.\frac{\partial \ln
P_\theta(x_i)}{\partial
\theta}\right|_{\theta=\theta^t}
\]
converges with probability $1$ to the update rule
\eqref{eq:IGOexactupdate}:
\[
\theta^{t+\deltat}=\theta^t+\deltat \, \widetilde\nabla_{\!\theta} \int
W_{\theta^t}^f(x)\, P_{\theta}(\d x)\enspace.
\]
\end{thm}

The proof is given in Appendix~\ref{sec:proofs}, under mild regularity assumptions. In
particular we do not require that $w$ is continuous.  Unfortunately, the
proof does not provide an explicit sample size above which the IGO
algorithm would be guaranteed to stay close to the IGO flow with high
probability; presumably such a size would be larger than the typical
sample sizes used in practice.

This theorem may clarify previous claims
\citep{Wierstra2008,Akimoto:ppsn2010} where rank-based updates similar to
\eqref{eq:IGOexactupdate}, such as in NES or CMA-ES, were derived from
optimizing the expected value $-\E_{P_\theta} f$. The rank-based weights
\wi\ were then introduced \emph{after} the derivation as a 
useful heuristic to improve stability.
Theorem~\ref{thm:consistency} shows that, for large $N$, CMA-ES and NES
actually follow the gradient flow of the quantity
$\E_{P_\theta}W_\thetat^f$: the update can be rigorously derived from
optimizing the expected value of the inverse-quantile-rewriting $W_\thetat^f$.

\subsection{Monotonicity: Quantile Improvement}

Gradient descents come with a guarantee that the fitness value decreases
over time. Here, since we work with probability distributions on $X$, we
need to define a ``fitness'' of the distribution $P_{\theta^t}$. An obvious
choice is the expectation $\E_{P_{\theta^t}}f$, but it is not
invariant under $f$-transformation and moreover may be sensitive to
extreme values.

It turns out that the monotonicity properties of the IGO gradient flow
depend on the choice of the selection scheme $w$. For instance, if
$w(u)=\mathbbm{1}_{u\leq 1/2}$, then the median of $f$ under
$P_{\theta^t}$ improves over
time.

\begin{prop}[Quantile improvement]
\label{prop:qimprove}
Consider the IGO flow~\eqref{eq:IGOflow}, with
the weight $w(u)=\mathbbm{1}_{u\leq q}$ where $0<
q< 1$ is fixed.
Then the value of the $q$-quantile of $f$ improves over time: if
$t_1\leq t_2$ then $Q^q_{P_{\theta^{t_2}}}(f) \leq
Q^q_{P_{\theta^{t_1}}}(f)$. 
Here the $q$-quantile value $Q^q_P(f)$ of $f$ under a
probability distribution $P$ is defined as the largest number $m$ such that
$\Pr_{x\sim P} (f(x)\!\geq\! m)\geq
1-q$.

Assume moreover that the objective function $f$ has no plateau, i.e.\ for any
$v\in \R$ and any $\theta\in \Theta$ we have $\Pr_{x\sim
P_\theta}(f(x)\!=\!v)=0$. Then  for $t_1<t_2$ either
$\theta^{t_1}=\theta^{t_2}$ or
$Q^q_{P_{\theta^{t_2}}}(f) <
Q^q_{P_{\theta^{t_1}}}(f)$.
\end{prop}

The proof is given in Appendix~\ref{sec:proofs}, together with the necessary
regularity assumptions. Note that on a discrete search space, the
objective function has only plateaus, and the $q$-quantile will
evolve by successive jumps even as $\theta$ evolves continuously.

This property is proved here only for the IGO gradient
flow~\eqref{eq:IGOflow} with
$N=\infty$ and $\deltat\to 0$. For an IGO algorithm with finite $N$, the dynamics is random
and one cannot expect monotonicity. Still,
Theorem~\ref{thm:consistency} ensures that, with high probability,
trajectories of a large enough finite population stay close to the
infinite-population limit trajectory.

In \citet{quantile_igo} this result was extended to\del{ actual} finite time
steps instead of infinitesimal $\deltat$, using the IGO-ML framework from
Section~\ref{sec:IGOML} below.

\subsection{The IGO Flow for Exponential Families}
\label{sec:exp}

The expressions for the
IGO update
simplify somewhat if the family $P_\theta$ happens to be an exponential family of
probability distributions (see also
\citealt{MalagoGECCO2008InformationGeometryPerspectiveOfEDAs,
malago2011geometry} for optimization using the natural gradient for
exponential families). This covers, for instance, Gaussian or
Bernoulli distributions.

Suppose that
$P_\theta$ can be written as \[
P_\theta(\d x)=\frac1{Z(\theta)}\exp\left(\sum \theta_i
T_i(x)\right)\,H(\d x)
\] where $T_1,\ldots,T_k$
is a finite family of functions
on $X$, $H(\d x)$ is an arbitrary reference measure on $X$, and
$Z(\theta)$ is the normalization constant. It is
well-known \cite[(2.33)]{Amari2000book} that
\begin{equation}
\label{eq:gradexp}
\frac{\partial\ln P_\theta(x)}{\partial \theta_i}=T_i(x)-\E_{P_\theta}
T_i
\end{equation}
so that \cite[(3.59)]{Amari2000book} 
\begin{equation}
\label{eq:fishexp}
I_{ij}(\theta)=\Cov_{P_\theta}(T_i,T_j)\enspace.
\end{equation}

By plugging this into the definition of the IGO
flow~\eqref{eq:IGOexplicit} we find:

\begin{prop}
\label{prop:expIGO}
Let $P_\theta$ be an exponential family parametrized by the natural
parameters $\theta$ as above.
Then the IGO flow is given by
\begin{equation}
\label{eq:expIGO}
\frac{\d \theta}{\d t}=\Cov_{P_\theta}(T,T)^{-1}
\Cov_{P_\theta}(T,W_\theta^f)
\end{equation}
where $\Cov_{P_\theta}(T,W_\theta^f)$ denotes the vector
$(\Cov_{P_\theta}(T_i,W_\theta^f))_i$, and $\Cov_{P_\theta}(T,T)$ the
matrix $(\Cov_{P_\theta}(T_i,T_j))_{ij}$.
\end{prop}

Note that the right-hand side does not involve derivatives w.r.t.\
$\theta$ any more.
This result makes it easy to simulate the IGO flow using, e.g., a Gibbs
sampler for $P_\theta$: both covariances in~\eqref{eq:expIGO} may be
approximated by sampling, so that neither the Fisher matrix nor the
gradient term need to be known in advance, and no derivatives are
involved. 


The values of the variables $\bar T_i=\E T_i$, namely the expected value of $T_i$
under the current distribution, can often be used as an alternative
parametrization for an exponential family (e.g.\ for a
one-dimensional Gaussian, these are the mean $\mu$ and the second moment
$\mu^2+\sigma^2$). The IGO flow \eqref{eq:IGOsemiexplicit} may be
rewritten using these variables,
using the relation
$\widetilde \nabla_{\theta_i}=\frac{\partial}{\rule{0cm}{1.8ex}\partial
\bar T_i}$ for the
natural gradient of exponential families (Proposition~\ref{prop:dualgrad}
in Appendix~\ref{sec:proofs}). One finds: 

\begin{prop}
With the same setting as in Proposition~\ref{prop:expIGO},
the expectation variables $\bar T_i=\E_{P_\theta} T_i$ 
satisfy the following evolution equation under the IGO flow
\begin{equation}
\label{eq:IGOEPflow}
\frac{\d \bar T_i}{\d t}=\Cov(T_i, W_\theta^f)
=\E (T_i \,W_\theta^f)-\bar T_i \,\E W_\theta^f\enspace.
\end{equation}
\end{prop}

The proof is given in Appendix~\ref{sec:proofs}, in the proof of
Theorem~\ref{thm:IGOCEM}. We shall further exploit this fact in
Section~\ref{sec:IGOML}.

\subsubsection{Exponential Families with Latent Variables.}
Similar formulas hold when the distribution $P_\theta(x)$ is the marginal
of an exponential distribution $P_\theta(x,h)$ over a ``hidden'' or
``latent'' variable
$h$, such as the restricted Boltzmann machines of Appendix~\ref{sec:RBM}.

Namely, with
$P_\theta(x)=\frac1{Z(\theta)} \sum_h \exp(\sum_i \theta_i
T_i(x,h))\,H(\d x,\d h)$ we have
\begin{equation}
\label{eq:gradhexp}
\frac{\partial \ln P_\theta(x)}{\partial \theta_i}
=U_i(x)-\E_{P_\theta} U_i
\end{equation}
where
\begin{equation*}
U_i(x)=\E_{P_\theta} (T_i(x,h)|x)
\end{equation*}
is the expectation of
$T_i(x,h)$ knowing $x$.
Then
the Fisher matrix is
\begin{equation}
\label{eq:fishhexp}
I_{ij}(\theta)=\Cov_{P_\theta}(U_i,U_j)
\end{equation}
and consequently,
the IGO flow takes the form
\begin{equation}
\label{eq:hexpIGO}
\frac{\d \theta}{\d t}=\Cov_{P_\theta}(U,U)^{-1}
\Cov_{P_\theta}(U,W_\theta^f)\enspace.
\end{equation}

\subsection{Further Mathematical Properties of IGO}

IGO enjoys a number of other mathematical properties that are expanded
upon in Appendix~\ref{sec:moremath}.

\begin{itemize}
\item By its very construction, the IGO flow is invariant under a number
of transformations of the original problem. First, replacing the
objective function $f$ with a strictly increasing function of $f$ does
not change the IGO flow (Proposition~\ref{prop:finv} in
Appendix~\ref{sec:invariance}).

Second, changing the parameterization
$\theta$ used for the family $P_\theta$ (e.g., letting $\theta$ be a
variance or a standard deviation) results
in unchanged trajectories for the distributions $P_\theta$ under the IGO
flow (Proposition~\ref{prop:thetainv} in Appendix~\ref{sec:invariance}). 

Finally, the IGO flow is insensitive to transformations of the original problem by
a change of variable in the search space $X$ itself, provided this
transformation can be reflected in the family of distributions $P_\theta$
(Proposition~\ref{prop:Xinv} in Appendix~\ref{sec:invariance}); this
covers, for instance, optimizing $f(Ax)$ instead of $f(x)$ in $\R^d$
using Gaussian distributions,
where $A$ is any invertible matrix.

These latter two invariances are specifically due to the natural
gradient and are not satisfied by a vanilla
gradient descent. $f$-invariance and $X$-invariance are directly
inherited from the IGO flow by IGO algorithms, but this is only approximately true of
$\theta$-invariance, as discussed in Appendix~\ref{sec:invariance}.

\item The speed of the IGO flow is bounded by the variance of the weight
function $w$ on $[0;1]$ (Appendix~\ref{sec:speed},
Proposition~\ref{prop:speed}). This implies that the parameter $\deltat$
in the IGO flow is not a meaningless variable but is related to the
maximal number of bits introduced in $P_\theta$ at each step
(Appendix~\ref{sec:speed}, Corollary~\ref{cor:KLspeed}).

\item IGO algorithms still make sense when the objective function $f$ is
noisy, that is, when each call to $f$ returns a non-deterministic value.
This can be accounted for without changing the framework: we prove in
Proposition~\ref{prop:noisyIGO} (Appendix~\ref{sec:noisy}) that IGO for
noisy $f$ is equivalent to IGO for a non-noisy $\tilde f$ defined on a
larger space $X\times \Omega$ with distributions $\tilde P_\theta$ that
are uniform over
$\Omega$. Consequently, theorems such as consistency of sampling
immediately transfer to the noisy case.

\item The IGO flow can be computed explicitly in the simple case of linear
functions on $\R^d$ using Gaussian distributions, and its
convergence can be proven for linear functions on  $\{0,1\}^d$ with Bernoulli distributions 
(Appendix~\ref{sec:examples}).

\end{itemize}

\subsection{Implementation Remarks}
\label{sec:impl}
\new{We conclude this section by a few practical remarks when implementing IGO algorithms.}

\subsubsection{Influence of the Selection Scheme $w$}
The selection scheme $w$ directly affects
the update rule~\eqref{eq:IGOupdatebis}.

A natural choice is $w(u)=\mathbbm{1}_{u\leq q}$. This results, as 
shown in Proposition~\ref{prop:qimprove}, in an
improvement of the $q$-quantile over the course of optimization.  Taking $q=1/2$
springs to mind \new{(giving positive weights to the better half of the samples)};
however, this is often not selective enough, and both
theory and experiments confirm that for the Gaussian case,
efficient
optimization requires $q<1/2$ (see Section~\ref{sec:gaussian}).
According to \citet{HGBeyer01},
on the sphere function $f(x) = \sum_i x_i^2$, the optimal $q$ is about $0.27$ if sample size
$N$ is not larger than the search
space dimension $d$, and even smaller otherwise \citep{jebaliaaugerppsn2010}.

Second,
replacing $w$ with $w+c$ for
some constant $c$ clearly has no influence on the IGO continuous-time
flow~\eqref{eq:IGOexactupdate}, since the gradient will cancel out the
constant. However, this is not the case for the update
rule~\eqref{eq:IGOupdatebis} with a finite sample of size $N$. 

Indeed, adding a constant
$c$ to $w$ adds a quantity $c\frac1N \sum \widetilde\nabla_\theta\ln
P_\theta(x_i)$ to the update. In expectation, this quantity vanishes
because
the $P_\theta$-expected value of
$\widetilde\nabla_\theta\ln P_\theta$ is $0$ (because $\int (\widetilde\nabla_\theta\ln
P_\theta)\,P_\theta=\int \widetilde\nabla P_\theta=\widetilde \nabla
1=0$). So adding a constant
to $w$ does not
change the expected value of the update, but it may change, e.g., its
variance. The empirical average of
$\widetilde\nabla_\theta\ln
P_\theta(x_i)$ in the sample will be $O(1/\sqrt{N})$. So translating the
weights results in a $O(1/\sqrt{N})$ change in the update. See also
Section~4 in \citet{Sun:2009:ENE:1569901.1569976}.

Thus, one may be tempted to introduce a well chosen value of $c$ so as to
reduce the variance of the update. However, determining an optimal value
for $c$ is difficult: the optimal value minimizing the variance actually
depends on possible correlations between $\widetilde\nabla_\theta\ln
P_\theta$ and the function $f$.  The only general result is that one
should shift $w$ such that $0$ lies within its range. Assuming
independence, or dependence with enough symmetry, the optimal shift is
when the weights average to $0$.

\subsubsection{Complexity}
The complexity of the IGO algorithm depends much on the computational
cost model. In optimization, it is fairly common to assume
that the objective function $f$ is very costly compared to any other
calculations performed by the algorithm \citep{more1981testing,dolan2002benchmarking}. 
Then the cost of IGO in terms of
number of $f$-calls is $N$ per iteration, and the cost of using
inverse quantiles and computing the natural gradient is negligible.

Setting the cost of $f$ aside, the complexity of the IGO algorithm depends
mainly on
the computation of the (inverse) Fisher matrix. Assume an analytical
expression for this matrix is known.  Then, with $p=\dim \Theta$ the
number of parameters, the cost of storage of the Fisher matrix is
$O(p^2)$ per iteration, and its inversion typically costs $O(p^3)$ per
iteration. However, depending on the situation and on possible algebraic
simplifications, strategies exist to reduce this cost (e.g.,
\citealt{LeRoux2007topmoumoute} in a learning context). For instance, for
CMA-ES the cost is $O(Np)$ \citep{suttorp2009efficient}. More generally,
parametrization by expectation parameters (see above), when available, may
reduce the cost to $O(Np)$ as well.

If no analytical form of the Fisher matrix is known and Monte Carlo
estimation is required, then complexity depends on the particular situation
at hand and is related to the best sampling strategies available for a
particular family of distributions. For Boltzmann machines, for instance,
a host of such strategies are available \citep{Ackley1985,Salakhutdinov2008, Salakhutdinov2009a,Desjardins2010}. Still, in such a
situation, IGO may be competitive if the objective function $f$ is costly.

\subsubsection{Recycling Old Samples} To compute the ranks of samples in
\eqref{eq:wi}, it might be advisable to re-use samples from previous
iterations, so that a smaller number of samples is necessary, see e.g.\ \citet{Sun:2009:ENE:1569901.1569976}. For $N=1$, this
is indispensable. 
In order to preserve sampling consistency
(Theorem~\ref{thm:consistency}) the old samples need to be reweighted
using the ratio of their likelihood under the current versus old
distribution, as in importance sampling.

\new{
In evolutionary computation,
elitist selection (also called plus-selection) is a common
approach where the all-time best samples are taken into account in each
iteration.  Elitist selection can be modelled in the IGO framework by using
the current all-time best samples in addition to samples from $P_\theta$. 
Specifically, in the ($\mu+\lambda$)-selection scheme, we set
$N=\mu+\lambda$ and let $x_1,\dots, x_\mu$ be the
current all-time $\mu$ best points. 
Then we sample $\lambda$ new points, $x_{\mu+1},\dots, x_N$, from
the current distribution $P_\theta$ and apply \eqref{eq:intrinsicIGOupdate} with $w(q) =
(N/\mu)\1_{q\le\mu/N}$.}

\subsubsection{Initialization}
As with other distribution-based optimization algorithms,
it is usually a good idea to
initialize in such a way as to cover a wide portion of
the search space, i.e.\ $\theta^0$ should be chosen so that
$P_{\theta^0}$ has large diversity. 
For IGO algorithms this is particularly effective, since, as explained
above, the natural gradient provides minimal change of diversity
(greedily at each step) for a given change in the objective function.

\section{IGO, Maximum Likelihood, and the Cross-Entropy Method}
\label{sec:IGOML}
\new{In this section we generalize the IGO update for settings where the natural gradient may not exist. 
This generalization reveals a unique IGO algorithm for finite step-sizes \deltat\ and a natural link to the cross-entropy method. }

\subsection{IGO as a Smooth-time Maximum Likelihood Estimate}
The IGO flow turns out to be
the only way to maximize a \emph{weighted} log-likelihood, where
points of the current distribution are slightly reweighted according to
$f$-preferences.

This relies on the following interpretation of the natural gradient as a
weighted maximum likelihood update with infinitesimal learning rate.
This result singles out, in yet another way, the \emph{natural} gradient among all
possible gradients.
The proof is given in Appendix~\ref{sec:proofs}.

\begin{thm}[Natural gradient as ML with infinitesimal weights]
\label{thm:IGOML}
Let $\eps>0$ and $\theta_0\in \Theta$. Let $W(x)$ be a function of $x$
and let $\theta$ be the solution of
\begin{equation}
\theta=\argmax_\theta\Bigg\{
(1-\eps) {\underbrace{{\int \ln P_\theta(x) \, P_{\theta_0}(\d
x)}}_\text{\back
  $= \mathrm{const} - \mathrm{KL}(P_{\theta_0}\|P_\theta)$, maximal for $\theta=\theta_0$\back\back}}
+
\eps \int \ln P_\theta(x) \, \overbrace{W(x)}^\text{\back\back\back~~~ preference weight biasing $P_{\theta_0}$\back\back} \,P_{\theta_0}(\d x)
\Bigg\}
\enspace.
\end{equation}
Then, when $\eps\to 0$ we have
\begin{equation}
\theta=\theta_0+\eps 
\int \gradt \ln
P_\theta(x) \,\, W(x) \,P_{\theta_0} (\d x) \, + O(\eps^2)
\enspace.
\end{equation}
Likewise for discrete samples: with $x_1,\ldots,x_N\in X$, let $\theta$ be the solution of
\begin{equation}
\theta=\argmax_\theta\left\{
(1-\eps) \int \ln P_\theta(x) \, P_{\theta_0}(\d x)
+\eps \sum_i W(x_i) \,\ln P_\theta(x_i)
\right\}\enspace.
\end{equation}
Then when $\eps\to 0$ we have
\begin{equation}
\theta=\theta_0+\eps \sum_i W(x_i)\,\, \gradt \ln
P_\theta(x_i) +O(\eps^2)
\enspace.
\end{equation}
\end{thm}

So if $W(x)=W_{\theta_0}^f(x)$ is the weight of the points according to
quantilized $f$-preferences, the weighted maximum log-likelihood
necessarily is the IGO flow~\eqref{eq:IGOsemiexplicit} using the
natural gradient---or the IGO update~\eqref{eq:IGOupdate} when using
samples. 

Thus the IGO flow is the unique flow that, continuously
in time, slightly changes the distribution to maximize the log-likelihood
of points with good values of $f$.
(In addition, IGO continuously updates
the weight
$W_{\theta^t}^f(x)$ depending on $f$ and on the current distribution, so
that we keep optimizing.)

This theorem suggests a way to approximate the IGO flow by enforcing this
interpretation for a given non-infinitesimal step size $\deltat$, as
follows.

\begin{defi}[IGO-ML algorithm]
\label{def:IGOML}
The \emph{IGO-ML algorithm} with step size $\deltat$ updates the value of
the parameter $\theta^t$ according to
\begin{equation}
\label{eq:IGOML}
\theta^{t+\deltat}=\argmax_\theta
\Bigg\{
{\textstyle(1-\deltat\sum\limits_i \wi)}\int \ln P_{\theta}(x)\,P_{\theta^t}(\d x)
\,+\,\deltat
\sum_i \wi \ln P_{\theta}(x_i)
\Bigg\}
\end{equation}
where $x_1,\ldots,x_N$ are sample points drawn according to the distribution
$P_{\theta^t}$, and
$\wi$ is the weight \eqref{eq:wi} obtained from the ranked values of the
objective function $f$.
\end{defi}

The IGO-ML algorithm is obviously independent of the parametrization
$\theta$: indeed it only depends on $P_\theta$ itself. Furthermore, the IGO-ML
update \eqref{eq:IGOML} does not even require a smooth parametrization of the
distribution anymore (though in this case, a small $\deltat$ will
likely result in stalling: $\theta^{t+\deltat}=\theta^t$ if the set of
possible values for $\theta$ is discrete).

Like the cross-entropy method below, the IGO-ML algorithm 
can be applied only when the argmax can be computed.

It turns out that for exponential families, IGO-ML is just the IGO
algorithm in a particular parametrization (see Theorem~\ref{thm:IGOCEM}).

\subsection{The Cross-Entropy Method}
Taking $\deltat=1$ in \eqref{eq:IGOML} above corresponds to a full maximum likelihood
update; when using the truncation selection scheme $w$, this is the
\emph{cross-entropy method} (CEM).
The cross-entropy method can be defined in an optimization setting as
follows \citep{CEMtutorial}. Like IGO, it depends on a family of
probability distributions $P_\theta$ parametrized by $\theta\in \Theta$,
and a number of samples $N$ at each iteration. Let also $N_e=\lceil q N\rceil$
($0< q< 1$) be a number of \emph{elite} samples.

At each step, the cross-entropy method for optimization samples $N$
points $x_1,\ldots,x_N$ from the current distribution $P_{\theta^t}$.
Let $\wi$ be $1/N_e$ if $x_i$ belongs to the $N_e$ samples with the best
value of the objective function $f$, and $\wi=0$ otherwise.
Then the \emph{cross-entropy method} or
\emph{maximum likelihoood} update (CEM/ML) for optimization is
\citep[Algorithm 3.1]{CEMtutorial}
\begin{equation}
\label{eq:CEM}
\theta^{t+1}=\argmax_\theta \sum \wi \ln P_\theta (x_i)
\end{equation}
(assuming the argmax is tractable). This corresponds to $\deltat=1$ in
\eqref{eq:IGOML}.

A commonly used version of CEM with a smoother update depends on a step size parameter $0<\alpha\leq 1$
and is given \citep{CEMtutorial} by
\begin{equation}
\label{eq:CEMsmooth}
\theta^{t+1}=(1-\alpha)\theta^t+\alpha \argmax_\theta \sum
\wi \ln P_\theta (x_i).
\end{equation}
The standard CEM/ML update is $\alpha=1$.
For $\alpha=1$, the standard cross-entropy method is independent of the
parametrization $\theta$, whereas for $\alpha<1$ this is not the case.

Note the difference between the IGO-ML algorithm \eqref{eq:IGOML}
and the smoothed CEM update \eqref{eq:CEMsmooth} with step size $\alpha=\deltat$:
the smoothed CEM update performs a weighted average of the parameter value
\emph{after} taking the maximum likelihood estimate, whereas IGO-ML
uses a weighted average of current and previous likelihoods, \emph{then} takes a maximum
likelihood estimate.
In general, these two rules can greatly differ, as they do for Gaussian
distributions (Section~\ref{sec:gaussian}).

This swapping of averaging makes 
IGO-ML parametrization-independent
whereas the smoothed CEM update is not.

Yet, for exponential families of probability distributions, there exists
one particular parametrization
$\theta$ in which the IGO algorithm 
and the smoothed CEM update coincide. We now proceed to this construction.

\subsection{IGO for Expectation Parameters and Maximum Likelihood}
The particular form of IGO for exponential families has an interesting
consequence if the parametrization chosen for the exponential family is
the set of \emph{expectation parameters}.
Let $
P_\theta(x)=\frac1{Z(\theta)}\exp\left(\sum \theta_j
T_j(x)\right)\,H(\d x)
$
be an exponential family as above. The \emph{expectation parameters} are
$\bar T_j=\bar T_j(\theta)=\E_{P_\theta} T_j$,
(denoted $\eta_j$ in
\citealt[Eq.\ 3.56]{Amari2000book}). The notation $\bar T$ will denote the
collection $(\bar T_j)$. We shall use the notation $P_{\bar T}$ to
denote the probability distribution $P$ parametrized by the expectation
parameters.

It is well-known that, in this parametrization, the maximum likelihood
estimate for a sample of points $x_1,\ldots,x_N$ is just the empirical
average of the expectation parameters over that sample:
\begin{equation}
\label{eq:TML}
\argmax_{\bar T} \frac1N \sum_{i=1}^N \ln P_{\bar T} (x_i)=\frac1N
\sum_{i=1}^N T(x_i)
\enspace.
\end{equation}

In the discussion above, one main difference between IGO and smoothed CEM was
whether we took averages before or after taking the maximum
log-likelihood estimate. For the expectation parameters
$\bar T_i$, we see that these operations commute. (One can say that these
expectation parameters ``linearize maximum likelihood estimates''.) After some work we get the following result.

\begin{thm}[IGO, CEM and maximum likelihood]
\label{thm:IGOCEM}
Let \[
P_\theta(\d x)=\frac1{Z(\theta)}\exp\left(\sum \theta_j
T_j(x)\right)\,H(\d x)
\]
be an exponential family of probability
distributions, where the $T_j$ are functions of $x$ and $H$ is some
reference measure. Let us parametrize this family by the expected values 
$\bar{T_j}=\E T_j$.

Let us assume the chosen weights $\wi$ sum to $1$.
For a sample $x_1,\ldots,x_N$, let
\[
T_j^\ast=\sum_i \wi \,T_j(x_i).
\]
Then the IGO update \eqref{eq:IGOupdate} in this parametrization reads
\begin{equation}
\label{eq:IGOEP}
\bar T_j^{t+\deltat}=(1-\deltat)\, \bar T_j^{t}+\deltat\, T_j^\ast.
\end{equation}
Moreover these three algorithms coincide:
\begin{itemize}
\item The IGO-ML algorithm \eqref{eq:IGOML}.
\item The IGO algorithm \eqref{eq:IGOupdate} written in the
parametrization $\bar{T_j}$ \eqref{eq:IGOEP}.
\item The smoothed CEM algorithm \eqref{eq:CEMsmooth} written in the parametrization
$\bar{T_j}$, with $\alpha=\deltat$.
\end{itemize}
\end{thm}

\begin{cor}
\label{cor:CEMIGO}
For exponential families, the standard CEM/ML update \eqref{eq:CEM} coincides with the IGO algorithm
in parametrization $\bar T_j$ with $\deltat=1$.
\end{cor}

Beware that the expectation parameters $\bar T_j$ are not always the most
obvious parameters \cite[Section 3.5]{Amari2000book}. For example, for $1$-dimensional Gaussian
distributions, the expectation parameters are the mean $\mu$ and the
second moment $\mu^2+\sigma^2$, not the mean and variance. 
When expressed back in terms of mean and
variance, the update \eqref{eq:IGOEP} boils down to
$\mu \gets (1-\deltat)\mu+\deltat \mu^*$ and 
$\sigma^2\gets (1-\deltat)\sigma^2+\deltat (\sigma^*)^2+
\deltat(1-\deltat)(\mu^*-\mu)^2$, where $\mu^\ast$ and $\sigma^*$ denote
the mean and standard deviation of the samples $x_i$.

On the other hand, when using smoothed
CEM with mean and variance as parameters, the new variance is 
$(1-\deltat)\sigma^2+\deltat
(\sigma^*)^2$, which can be significantly smaller for $\deltat\in (0,1)$.
This proves, in passing, that the smoothed CEM update
in other parametrizations is generally \emph{not} an IGO algorithm
(because it can differ at first order in $\deltat$).

The case of Gaussian distributions is further exemplified in
Section~\ref{sec:gaussian} below: in particular, smoothed CEM in the
$(\mu,\sigma)$ parametrization almost invariably exhibits
a reduction of variance, often leading to premature convergence.

For these reasons we think that the IGO-ML algorithm is the sensible way
to define an interpolated ML estimate for $\deltat<1$ in a
parametrization-independent way (see however the analysis of a critical
$\deltat$ in Section~\ref{sec:gaussian}).\niko{The problem: for
$\deltat=1$ this algorithm doesn't work on a linear function, hence it is
not likely to be the method of choice for any value close to one. }  In
Appendix~\ref{sec:discussion} we further discuss IGO and CEM and sum up
the differences and relative advantages.

Taking $\deltat=1$ is a bold approximation choice:
the ``ideal'' continuous-time IGO flow itself, after time $1$, does not coincide with
the maximum likelihood update of the best points in the sample. Since
the maximum likelihood algorithm is known to converge prematurely in some
instances
(Section~\ref{sec:gaussian}), using the parametrization by expectation
parameters with large $\deltat$ may not be desirable.

The considerable simplification of the IGO update in these coordinates
reflects the duality of coordinates $\bar T_i$ and $\theta_i$. More
precisely, the natural gradient ascent w.r.t.\ the parameters
$\bar T_i$ is given by the vanilla gradient w.r.t.\ the
parameters $\theta_i$: 
\[
\widetilde{\nabla}_{\bar
T_i}=\frac{\partial}{\partial \theta_i}
\]
(Proposition~\ref{prop:dualgrad} in Appendix~\ref{sec:proofs}).

\section{CMA-ES, NES, EDAs and PBIL from the IGO Framework}\label{sec:CMA-PBIL-IGO}

In this section we investigate the IGO algorithms for Bernoulli measures
and for multivariate normal distributions, and show the correspondence to
well-known algorithms. Restricted Boltzmann machines are given as a
third, novel example. In addition, we discuss the influence of the parametrization of the distributions. 

\subsection{PBIL and cGA as IGO Algorithm\new{s} for Bernoulli Measures}
\label{sec:PBIL}

Let us consider on $X = \{ 0,1\}^{d}$ a family of Bernoulli measures
$P_\theta(x)= p_{\theta_{1}}(x_{1}) \times \ldots \times
p_{\theta_{d}}(x_{d}) $ with $p_{\theta_{i}}(x_{i})=\theta_{i}^{x_{i}}(1
- \theta_{i})^{1-x_{i}}$, with each $\theta_i\in[0;1]$. As this family is a product of probability
measures $p_{\theta_{i}}(x_{i})$, the different components of a random
vector $y$ following $P_{\theta}$ are independent and all off-diagonal
terms of the Fisher information matrix are zero. Diagonal terms are given
by $\frac{1}{\theta_{i}(1-\theta_{i})}$. Therefore the inverse of the
Fisher matrix is a diagonal matrix with diagonal entries equal to $
\theta_{i}(1 - \theta_{i})$. In addition, the partial derivative of $\ln
P_{\theta}(x)$ w.r.t.\ $\theta_{i}$ is computed in a straightforward manner resulting in
$$
\frac{\partial \ln P_{\theta}(x)}{\partial \theta_{i}} = \frac{x_{i}}{\theta_{i}} - \frac{1-x_{i}}{1 - \theta_{i}} \enspace .
$$ 

Let $x_{1}, \ldots, x_{N}$ be $N$ samples at step $t$ with distribution
$P_{\theta^{t}}$ and let $x_{1:N}, \ldots, x_{N:N}$ be the samples ranked
according to $f$ value.
The natural gradient update \eqref{eq:IGOupdatebis} with Bernoulli measures is then
\begin{equation}\label{eq:PBILIGO}
\theta^{t+\deltat}_{i} = \theta^{t}_{i} + \deltat \, \theta_{i}^{t} (1-\theta_{i}^{t})  \sum_{j=1}^{N} w_j \left(  \frac{[x_{j:N}]_{i}}{\theta_{i}^{t}} - \frac{1-[x_{j:N}]_{i}}{1 - \theta_{i}^{t}} \right)
\end{equation}
where $w_j = w({(j - 1/2)}/{N})/N $ and $[y]_{i}$ denotes the
$i^\text{th}$ coordinate of $y \in X$. The previous equation simplifies to
\begin{equation}
\label{eq:PBILIGOsimple}
\theta^{t+\deltat}_{i} = \theta^{t}_{i} + \deltat \, \sum_{j=1}^{N} w_j \left([x_{j:N}]_{i} - \theta_{i}^{t} \right) \enspace,
\end{equation}
or, denoting $\bar{w}$ the sum of the weights $\sum_{j=1}^{N} w_j$,
\begin{equation}
\theta^{t+\deltat}_{i} = (1 - \bar{w} \deltat)\, \theta^{t}_{i} + \deltat
\, \sum_{j=1}^{N} w_j \, [x_{j:N}]_{i} \enspace.
\end{equation}

\newcommand{\LR}{\ensuremath{{\text{\scshape lr}}}}
\newcommand{\MUTPROB}{\ensuremath{{\text{\scshape mut\_probability}}}}

The algorithm so obtained coincides with the so-called
\emph{population-based incremental learning} algorithm (PBIL,
\citealt{Baluja:ICML1995}) 
for $N=\texttt{NUMBER\_SAMPLES}$ and the appropriate (usually non-negative) weights $w_j$, 
as well as with the \emph{compact genetic algorithm} (cGA, \citealt{harik1999compact}) for $N=2$ and $w_1=-w_2$. Different variants of PBIL correspond to
different choices of the selection scheme $w$.
\new{In cGA, components for which both samples have the same value are unchanged, because $w_1+w_2=0$.}\del{when both samples have the same value, the parameter is
unchanged as $w_1+w_2=0$.} 
We have thus proved the following.

\begin{prop}
\label{prop:PBIL}
The IGO algorithm on $\{0,1\}^d$ using Bernoulli measures parametrized
by $\theta$ as above, coincides with the \emph{compact Genetic Algorithm} (cGA) 
when $N=2$ and $w_1=-w_2$.

Moreover, it coincides with \emph{Population-Based Incremental Learning} (PBIL) with the following correspondence of parameters.
The PBIL algorithm using the $\mu$ best solutions, see
\citet[Figure~4]{Baluja:ICML1995}, is recovered\,\footnote{Note that the pseudocode
for
the algorithm in \citet[Figure~4]{Baluja:ICML1995} is slightly erroneous since it
gives smaller weights to better
individuals. The error can be fixed by updating the probability in
reversed order, looping from {\tt\footnotesize
NUMBER\_OF\_VECTORS\_TO\_UPDATE\_FROM} to 1. This was confirmed
by S.~Baluja in personal communication. We consider here the corrected
version of the algorithm.} using $\deltat=\LR$, 
$w_j = (1-\LR)^{j-1}$ for $j=1,\ldots,\mu$, and $w_j=0$ for
$j=\mu+1,\ldots,N$.

If the selection scheme of IGO is chosen as $w_1 = 1$, $w_j=0$ for
$j=2,\ldots,N$, IGO recovers the PBIL/EGA algorithm with update rule
towards the best solution \cite[Figure~4]{Baluja:PBIL:1994}, with $\deltat = \LR$ (the learning rate
of PBIL) and $\MUTPROB=0$ (no random mutation of $\theta$).
\end{prop}

%

Interestingly, the parameters $\theta_i$ are the expectation parameters described
in Section~\ref{sec:IGOML}: indeed, the expectation of $x_i$ is
$\theta_i$. So the formulas above are particular
cases of \eqref{eq:IGOEP}. Thus, by Theorem~\ref{thm:IGOCEM}, PBIL
is both a smoothed CEM in these parameters and an IGO-ML algorithm.

\newcommand{\thetap}{\vphantom{\tilde{\tilde\theta}}\tilde{\theta}}

Let us now consider another, so-called ``logit'' representation, given by the logistic function
$P(x_i=1) = \frac{1}{1 + \exp(-\thetap_i)}$. This $\thetap$ is the
exponential parametrization of Section~\ref{sec:exp}. We find that
\begin{equation}
\frac{\partial \ln P_{\thetap}(x)}{\partial \thetap_{i}} = (x_{i} -1) +
\frac{\exp(-\thetap_{i})}{1+\exp(-\thetap_{i})}=x_i-\E x_i
\end{equation}
(cf.~Eq.\;\ref{eq:gradexp}) and that the diagonal elements of the Fisher information matrix are given
by $\exp(-\thetap_{i})/(1+\exp(-\thetap_{i}))^{2}=\Var x_i$
(as per Eq.\;\ref{eq:fishexp}).  So the natural gradient update
\eqref{eq:IGOupdatebis} with Bernoulli measures in parametrization
$\tilde \theta$ reads
\begin{equation}
\thetap^{t+\deltat}_{i} = \thetap^{t}_{i} + \deltat (1 + \exp(\thetap_{i}^{t})) \left( - \bar{w} +  (1+\exp(-\thetap_{i}^{t}))  \sum_{j=1}^{N} w_j [x_{j:N}]_{i} \right) \enspace.
\end{equation}

To better compare the update with the previous representation, note
that $\theta_{i} = \frac{1}{1 + \exp(-\thetap_i)} $ and thus we can
rewrite
\begin{equation}
\thetap^{t+\deltat}_{i} = \thetap^{t}_{i} +
\frac{\deltat}{\theta_{i}^{t}(1-\theta_{i}^{t})} 
\sum_{j=1}^{N} w_j \left(
[x_{j:N}]_{i}-\theta_i^t
\right)
\enspace.
\end{equation}
%

So the direction of the update is the same as before and is given by the
proportion of bits set to 0 or 1 in the best samples, compared to its expected
value
under the current distribution. The
magnitude of the update is different since the parameter $\thetap$ ranges
from $-\infty$ to $+\infty$ instead of from $0$ to $1$.
We did not find this algorithm in the literature.

These updates also illustrate the influence of setting the sum of weights to
$0$ or not (Section~\ref{sec:impl}). If, at some time, the first
bit is equal to 1 both for a majority of good points and for a majority of
bad points, then the original PBIL will increase the probability of
setting the first bit to $1$, which is counterintuitive. If the weights
$w_i$ are chosen to sum to $0$ this noise effect
disappears; otherwise, it disappears only on average.

\subsection{Multivariate Normal Distributions (Gaussians)}
\label{sec:gaussian}

Evolution strategies \citep{rechenberg1973, schwefel1995evolution,
beyer2002evolution} are black-box optimization algorithms for the
continuous search domain, $X\subseteq\R^\n$ (for simplicity we
assume $X=\R^\n$ in the following), which use multivariate normal
distributions to sample new solutions.
In the context of continuous black-box
optimization, \emph{Natural Evolution Strategies} (NES) introduced the
idea of using a natural gradient update of the distribution parameters
\citep{Wierstra2008, Sun:2009:ENE:1569901.1569976, Glasmachers2010,wierstra2014natural}.
Surprisingly, the well-known \emph{Covariance Matrix Adaption
Evolution Strategy} (CMA-ES, \citealt{hansen1996, hansen2001,
hansen2003reducing, hansen2004, jastrebski2006improving}) also turns out to
conduct a natural gradient update of distribution parameters
\citep{Akimoto:ppsn2010, Glasmachers2010}. 

Let $x\in\R^d$. As the most prominent example, we use mean vector $\m=\E
x$ and covariance matrix $\Cm=\E (x-\m)(x-\m)^\T=\E(xx^\T)-\m\m^\T$ to
parametrize a normal distribution via $\theta=(\m, C)$. 
The IGO update in \eqref{eq:IGOupdate} or \eqref{eq:IGOupdatebis} in this parametrization can now be entirely formulated without the (inverse) Fisher matrix, similarly
to \eqref{eq:IGOEP} or \eqref{eq:IGOEPflow}.
The complexity of the update is linear in the number of parameters (size of $\theta=(m,C)$,
where $(d^2-d)/2$ parameters are redundant).

Let us discuss known algorithms that implement updates of this kind.

\subsubsection{CMA-ES.} 
The rank-$\mu$-update CMA-ES implements the equations\footnote{The CMA-ES implements these
equations given the parameter setting $c_1=0$ and $c_\sigma=0$ (or
$d_\sigma=\infty$, see e.g.\ \citealt{hansen2009benchmarking}) that
disengages the effect of the rank-one update and of step size control and therefore of both so-called evolution paths.}
\mc{\mt}{m^t}
\mc{\mtt}{m^{t+1}}
\mc{\Ct}{C^t}
\mc{\Ctt}{C^{t+1}}
\begin{align}
   \label{eq:cmamean}
   \mtt &= \mt + \etam \sum_{i=1}^N \wi (\x_i - \mt)
   \\ 
    \label{eq:cmacov}
   \Ctt &= \Ct + \etac \sum_{i=1}^N \wi((\x_i-\mt) (\x_i - \mt)^\T - \Ct)
\end{align}
where $\wi$ are the weights based on ranked $f$-values, see \refeq{wi}
and \eqigo.

\begin{prop}
The IGO update \eqref{eq:IGOupdate} for Gaussian distributions in the
parametrization by mean and covariance matrix $(m,C)$, coincides with 
the CMA-ES update equations \eqref{eq:cmamean} and \eqref{eq:cmacov} with
$\etac=\etam$.
\end{prop}
This result is essentially due to 
\citet{Akimoto:ppsn2010} and
\citet{Glasmachers2010}, who showed that the CMA-ES
update with $\etac=\etam$ is a natural gradient update\footnote{ In these articles the result has been derived for
$\theta\gets\theta + \eta\,\gradt \E_{\Pt} f$, see \eqref{eq:IGOflowf},
leading to $f(x_i)$ in place of \wi. No assumptions on $f$ have been used
besides that it does not depend on $\theta$. Consequently, by replacing $f$ with \Wf,
where $\thetat$ is fixed, the derivation holds equally well for
$\theta\gets\theta + \eta\,\gradt\thegoal$.}.

However, 
in deviation from the IGO algorithm, the learning rates \etam\ and
\etac\ are assigned
different values if $N\ll\dim \Theta$ in CMA-ES\footnote{%
Specifically, let $\sum |\wi| = 1$, then the settings are $\etam=1$ and $\etac\approx 1/(\n^2\sum \wi^2)$ \citep{hansen2006eda}. }. 
Note that the Fisher
information matrix is block-diagonal in $m$ and $C$
\citep{Akimoto:ppsn2010}, so
that application of the different learning rates and of the inverse
Fisher matrix commute.
Moreover, CMA-ES uses a \emph{path cumulation} method to adjust the
step sizes, which is not covered by the IGO framework (see
Appendix~\ref{sec:discussion}).

\subsubsection{Convenient Reparametrizations Over Time.} For practical
purposes, at each step it is convenient to work in a representation of
$\theta$ in which the diagonal Fisher matrix $I(\theta^t)$ has a simple
form, e.g., diagonal with simple diagonal entries. It is generally not
possible to obtain such a representation for all $\theta$ simultaneously.
Still it is always possible to find a transformation achieving a diagonal
Fisher matrix at a single parameter $\theta^t$, in multiple ways (it
amounts to choosing a basis of parameter space which is
orthogonal in the Fisher metric). Such a representation is never unique
and not intrinsic,
yet it still provides a convenient way to write the algorithms.

For CMA-ES, one such representation can be found by sending the current
covariant matrix $C^t$ to the identity, e.g., by representing the mean and
covariance matrix by $((C^t)^{-1/2}m, (C^t)^{-1/2}C(C^t)^{-1/2})$ instead
of $(m,C)$. Then the Fisher matrix $I(\theta^t)$ at $(m^t,C^t)$ becomes
diagonal.\NDY{with entries $1$ for $m$, $1/4$ for the diagonal entries of
$C$ and $1/2$ for the non-diag entries of $C$ I believe, but I'd have to
check carefully.}
The next algorithm we discuss, xNES
\citep{Glasmachers2010},
exploits this possibility in a logarithmic representation of the covariance
matrix.

\subsubsection{Natural Evolution Strategies.} 
Natural evolution strategies (NES, \citealt{Wierstra2008,
Sun:2009:ENE:1569901.1569976}) implement \eqref{eq:cmamean} as well, while using
a Cholesky decomposition of $C$ as the parametrization for the update of the
variance parameters. The resulting update that replaces \eqref{eq:cmacov}
is neither particularly elegant nor numerically efficient. The more recent xNES \citep{Glasmachers2010} chooses an ``exponential'' parametrization that naturally depends on the current parameters. This leads to an elegant formulation where the additive update in exponential parametrization becomes a multiplicative update for $C$. With $C=AA^\T$, the matrix update reads
\begin{equation}\label{eq:glasmachers}
    A \gets A \times \exp\left(\frac{\etac}{2} \sum_{i=1}^N \wi\times(z_i z_i^\T - \I)\right)
\end{equation}
where $z_i = A^{-1}(\x_i-\m)$ and $\I$ is the identity matrix. From
\eqref{eq:glasmachers} the updated covariance matrix is
$C\gets A\times \exp\left(\etac \sum_{i=1}^N \wi\times(z_i z_i^\T - \I)\right) \times A^\T$.

Compared to \eqref{eq:cmacov}, the update has the advantage
that also negative weights, $\wi<0$, always lead to a feasible
covariance matrix. By default, xNES sets $\etam\not=\etac$ in the same circumstances
as in CMA-ES, but contrary to CMA-ES the past evolution path is not taken into account \citep{Glasmachers2010}.

When $\etac=\etam$, xNES is consistent with the IGO flow
\eqref{eq:IGOflow}, and implements an IGO algorithm
\eqref{eq:IGOupdate} slightly generalized in that it uses a
$\theta^t$-dependent parametrization, which represents the current
covariance matrix by $0$. Namely, we have:

\begin{prop}[exponential IGO update of Gaussians]
Let $(m^t,C^t)$ be the current mean and covariance matrix. Let
$C^t=AA^\T$. Let $\theta$
be the time-dependent parametrization of the space of Gaussian
distributions, which parametrizes the Gaussian distribution $(m,C)$ by
\begin{equation*}
\theta=(m,R),\qquad R=\ln (A^{-1} C (A^\T)^{-1})
\end{equation*}
where $\ln$ is the logarithm of positive matrices.
(Note that the current value $C^t$ of $C$ is represented as
$R=0$.)

Then the IGO update~\eqref{eq:IGOupdate} in the parametrization $\theta$ is as
follows: the mean $m$ is updated as in CMA-ES~\eqref{eq:cmamean}, and the
parameter $R$ is updated as
\begin{equation}
R \gets \deltat \sum_{i=1}^N \wi\times(A^{-1} (x_i-m)
(x_i-m)^\T (A^\T)^{-1} - \I)
\end{equation}
thus resulting in the same update as~\eqref{eq:glasmachers} (with
$\etac=\deltat$) for the covariance
matrix: $C\gets A \exp(R) A^\T$.
\end{prop}

\begin{proof}
Indeed, by basic differential geometry, if parametrization $\theta'={\phi}(\theta)$ is
used, the IGO update for $\theta'$ is $D{\phi}(\theta^t)$ applied to the IGO
update for $\theta$, where $D{\phi}$ is the differential of
${\phi}$. Here, given the update~\eqref{eq:cmacov} for $C$, to find
the update for $R$
we have to compute the differential of the map $C\mapsto \ln(A^{-1}C
(A^\T)^{-1})$ taken at $C=AA^\T$: for any matrix $M$ we have
$\ln(A^{-1}(AA^\T+\epsilon
M)(A^\T)^{-1})=\epsilon\,A^{-1}M(A^\T)^{-1}+O(\eps^2)$. So to find the
update for the variable $R$ we have to apply $A^{-1}\ldots (A^\T)^{-1}$
to the update~\eqref{eq:cmacov} for $C$.
\end{proof}




\subsubsection{Cross-Entropy Method and EMNA} 
\emph{Estimation of distribution algorithms} (EDA) and the \emph{cross-entropy method} (CEM, \citealt{rubinstein1999cross,rubinstein2004cross}) estimate a new distribution from a censored sample.
Generally, the new parameter value
can be written as 
\mc{\thetaLL}{\theta_\mathrm{maxLL}}
\mc{\Pw}{P_w}%
\begin{align}
    \label{eq:minEMNA}
    \thetaLL 
      &= \arg\max_\theta \sum_{i=1}^N \wi \ln \Pt(x_i)
    \\\nonumber
      &\longrightarrow_{N\to\infty} \arg\max_\theta \E_{\Ptt} \Wf \ln\Pt 
\end{align}
by Theorem~\ref{thm:consistency}.
Here, the weights $\hat{w}_i$ are equal to $1/\mu$ for the $\mu$ best
points (censored or elitist sample) and $0$ otherwise. 
This \thetaLL\ maximizes the weighted log-likelihood of $x_1,\dots, x_N$;
equivalently, it minimizes the cross-entropy and the Kullback--Leibler
divergence to the distribution of the best $\mu$ samples%
\footnote{Let \Pw\ denote the distribution of the weighted
samples: $\mathrm{Pr}(x=x_i) = \wi$ and $\sum_i\wi=1$. Then the
cross-entropy between \Pw\ and \Pt\ reads $\sum_i \Pw(x_i) \ln1/\Pt(x_i)$
and the KL divergence reads $\KL{\Pw}{\Pt} = \sum_i \Pw(x_i)
\ln1/\Pt(x_i) - \sum_i \Pw(x_i) \ln 1/\Pw(x_i)$. Minimization of both
terms in $\theta$ result in \thetaLL. }.

For Gaussian distributions, Equation \eqref{eq:minEMNA} can be explicitly written in the form 
\mc{\Chat}{\widehat{C}}
\begin{align}
\label{eq:EMNAmean}
\mtt &= m^\ast=\sum_{i=1}^N \wi x_i \\
\label{eq:EMNAcov}
\Ctt &= C^\ast=
    \sum_{i=1}^N \wi (x_i - m^\ast)(x_i - m^\ast)^\T
\end{align}
the empirical mean and variance of the elite sample.
%
%

Equations \refeq{EMNAmean} and \refeq{EMNAcov} also define the
simplest continuous domain EDA, the \emph{estimation of multivariate
normal algorithm} (EMNA$_\mathrm{global}$,
\citealt{larranaga2002estimation}). Interestingly,
\refeq{EMNAmean} and \refeq{EMNAcov} only differ from \eqref{eq:cmamean}
and \eqref{eq:cmacov} (with $\etam=\etac=1$) in that the new mean \mtt\ is used instead of \mt{} in the covariance matrix update \citep{hansen2006eda}.

The smoothed CEM \eqref{eq:CEMsmooth} in this parametrization thus writes
\begin{equation}m^{t+\deltat}=(1-\deltat) m^t+\deltat
\label{eq:smoothCEMmean}
m^\ast\end{equation}
\begin{equation}
\label{eq:smoothCEMcov}
C^{t+\deltat}=(1-\deltat) C^t+\deltat C^\ast \enspace.
\end{equation}
Note that \emph{this is
not an IGO algorithm} (i.e., there is no parametrization of the family of
Gaussian distributions in which the IGO algorithm coincides with
update Eq.\;\ref{eq:smoothCEMcov}): 
indeed, all IGO algorithms coincide at first order in $\deltat$
when $\deltat\to 0$ (because they recover the IGO flow), while this update for $C^{t+\deltat}$ does not
coincide with~\eqref{eq:cmacov} in this limit, due to the use of
$m^\ast$ instead of $m^t$.
This does not
contradict Theorem~\ref{thm:IGOCEM}: smoothed CEM is an IGO algorithm
only if smoothed CEM is written in the expectation parametrization, which $(m,C)$ is not.

\subsubsection{CMA-ES, Smoothed CEM, and IGO-ML}
Let us compare IGO-ML \eqref{eq:IGOML}, rank-$\mu$ CMA
\eqref{eq:cmamean}--\eqref{eq:cmacov}, and smoothed CEM
\eqref{eq:smoothCEMmean}--\eqref{eq:smoothCEMcov}
 in the parametrization with
mean and covariance matrix. 
These algorithms all update the
distribution mean in the same way, while the
update of the covariance matrix depends on the algorithm. With
learning rate $\deltat$, these
updates are computed to be
\begin{equation}
\label{eq:covgeneral}
\begin{aligned}
\mtt &= (1-\deltat)\, \mt + \deltat\, m^*
\\
\Ctt &= (1-\deltat)\, \Ct + \deltat\, C^\ast + \deltat (1-\deltat)^j\,
(m^*-\mt)(m^*-\mt)^\T
\enspace,
\end{aligned}
\end{equation}
for different values of $j$,
where $m^\ast$ and $C^\ast$ are the mean and covariance matrix computed
over the elite sample (with positive weights $\wi$ summing to one) as above. 
The rightmost term of \refeq{covgeneral} is reminiscent of the so-called rank-one
update in CMA-ES (not included in Eq.\;\ref{eq:cmacov}).

For $j=0$ we recover the rank-$\mu$ CMA-ES update \eqref{eq:cmacov}, for $j=1$ we recover IGO-ML, and for $j=\infty$ we recover smoothed CEM (the
rightmost term is absent). The case $j=2$ corresponds to an update 
that uses $m^{t+1}$  instead of $m^t$ in
\eqref{eq:cmacov} (with $\etam=\etac=\deltat$).
For $0<\deltat<1$, the larger $j$, the smaller \Ctt.
For $\deltat=1$, IGO-ML and
smoothed CEM/EMNA
realize \thetaLL\ from \refeq{minEMNA}--\eqref{eq:EMNAcov}.

For $\deltat\to0$, the update is independent of $j$ at first order in
$\deltat$ if $j<\infty$: this reflects compatibility with the IGO flow of
CMA-ES and of IGO-ML, but not of smoothed CEM.

\mc{\muw}{\widehat\mu}
In the default (full) CMA-ES (as opposed to rank-$\mu$ CMA-ES), the coefficient preceeding $(m^*-\mt)(m^*-\mt)^\T$ in 
\eqref{eq:covgeneral} reads approximately $3\deltat$, where the
additional $2\deltat$ originate from the so-called rank-one update and
are moreover modulated by a ``cumulated path'' up to a factor of about $\sqrt{d}$ \citep{hansen2014principled}.

\subsubsection{Critical \deltat}
Let us assume that $\mu<N$ weights are set to
$\wi = 1/\mu$ and the remaining weights to zero, so that the selection ratio is $q=\mu/N$.

Then there is a critical value of $\deltat$ depending on this 
ratio $q$, such
that above this critical \deltat\ the algorithms given by IGO-ML and
smoothed CEM are prone to premature
convergence. Indeed, let $f$ be a linear function on $\R^d$, and consider
the variance in the direction of the gradient of $f$. 
Assuming further $N\to\infty$ and $q\le1/2$, then the variance $C^\ast$
of the elite sample is smaller than the current variance $C^t$, by a
constant factor.
Depending on the precise update for $C^{t+1}$, if $\deltat$ is too
large, the variance $C^{t+1}$ is going to be smaller than $C^t$ by a
constant factor as well.
This implies that the algorithm is going to stall, i.e., the variance
will go to $0$ before the optimum is reached. (On the other hand,
the continuous-time IGO flow corresponding to $\deltat\to 0$ does not
stall, see Appendix~\ref{sec:examples}.)

We now study the critical $\deltat$ (in the limit $N\to\infty$) \new{below}\del{under}
which the algorithm does not stall (this is done similarly
to the analysis in Appendix~\ref{sec:examples} for linear functions).
For IGO-ML, 
($j=1$ in Eq.\;\ref{eq:covgeneral}, or equivalently for the smoothed CEM in the
expectation parameters $(m,C+mm^\T)$, see
Section~\ref{sec:IGOML}),
the variance increases if and only if $\deltat$ is smaller than the
critical value
$\deltat_\text{crit}=qb\sqrt{2\pi}e^{b^2/2}$
where $b$ is the percentile function of $q$, 
i.e.\ $b$ is such that $q=\int_b^\infty
e^{-x^2/2}/\sqrt{2\pi}$. This
value $\deltat_\text{crit}$ is plotted as a solid line in Fig.~\ref{fig:criticaldt}.  For $j=2$, $\deltat_\text{crit}$ is
smaller, related to the above by
$\deltat_\text{crit}\leftarrow\sqrt{1+\deltat_\text{crit}}-1$ and plotted
as a dashed line in Fig.~\ref{fig:criticaldt}. For CEM ($j=\infty$), the
critical \deltat\ is zero (reflecting the non-IGO behavior of CEM in this
parametrization). For CMA-ES ($j=0$), the
critical \deltat\ is infinite for $q<1/2$. When the selection ratio
$q$ is above $1/2$, for all algorithms the critical \deltat\ becomes zero. 

\begin{figure}
\centering\includegraphics[width=0.49\textwidth]{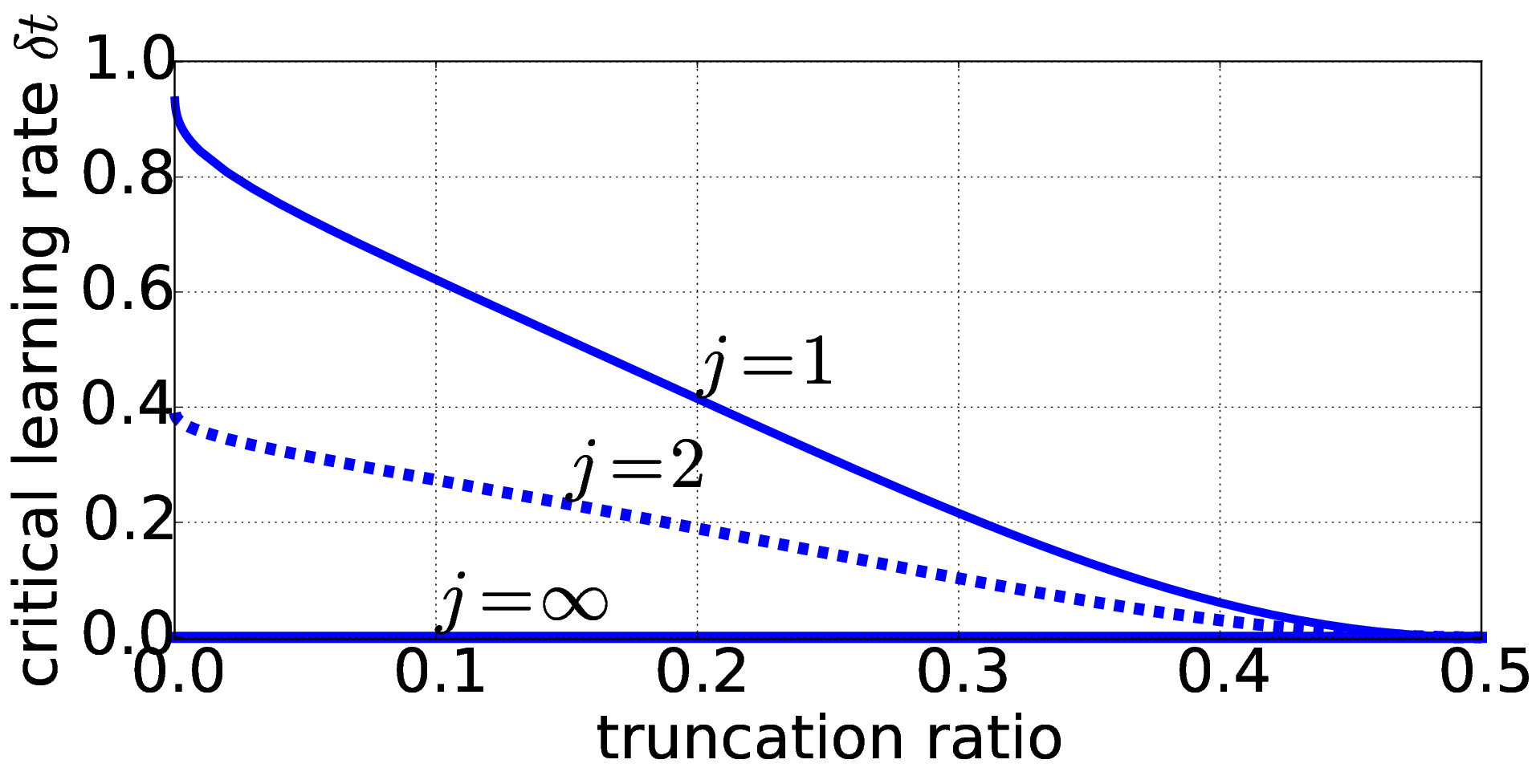}
\caption{\label{fig:criticaldt}Critical \deltat\ versus selection truncation
ratio $q$ for three values of $j$ in \refeq{covgeneral}. With \deltat\ above the critical \deltat, the variance decreases systematically when optimizing a linear function, indicating failure of the algorithm. For CMA-ES/NES \new{where $j=0$}, the critical \deltat\ for $q<0.5$ is infinite. }
\end{figure}

\del{We conclude that, despite the principled approach of ascending the
natural gradient, the choice of the selection function $w$, the choice of
\deltat, and possible choices in the update for $\deltat>0$\TODO[I don't
really understand what is meant here---Yann], 
need to be taken with care in
relation to the choice of parametrization.}

\subsubsection{Gaussian Distributions with Restricted Parametrization}

When considering a restricted parametrization of multivariate normal
distributions, IGO recovers other known algorithms. In particular for
sep-CMA-ES \citep{ros2008simple} and SNES \citep{schaul2011high}, the
update has been restricted to the diagonal of the covariance matrix.   

\del{\subsubsection{Elitist selection.} In evolution strategies,
elitist selection (also called plus-selection) is another common
approach in which the all-time best point(s) are taken into account in each
iteration.  Plus-selection can be modelled in the IGO framework by using
the current all-time best samples in addition to samples from $P_\theta$. 
Specifically, in the ($\mu+\lambda$)-selection scheme, we set
$N=\mu+\lambda$ and let $x_1,\dots, x_\mu$ be the
current all-time $\mu$ best points. 
Then we sample $\lambda$ new points, $x_{\mu+1},\dots, x_N$, from
the current distribution $P_\theta$ and just apply \eqref{eq:intrinsicIGOupdate} with $w(q) =
(N/\mu)\1_{q\le\mu/N} $.}

\subsection{IGO for Restricted Boltzmann Machines,
and Multimodal Optimization}

As a third example after Bernoulli and Gaussian distributions, we have
applied IGO to restricted Boltzmann machines (RBMs,
\citealt{Smolensky1986,Ackley1985}), which are families of 
distributions on the search space $X=\{0,1\}^d$ extending Bernoulli
distributions to allow for dependencies between the bits (see, e.g.,
\citealt{Berny2002} for Boltzmann machines used in an optimization
setting). The details are given in Appendix~\ref{sec:RBM}.
We chose this example for two reasons. 

First, it illustrates how to obtain a novel algorithm from a family of
probability distributions in the IGO framework, in a case when there is
no fully explicit expression for the Fisher information matrix.

Second, it illustrates the influence of the natural gradient on diversity
and its relevance to multimodal optimization. The RBM probability
distributions on $\{0,1\}^d$ are multimodal, contrary to Bernoulli
distributions. Thus, in principle, IGO on such distributions could reach several optima
simultaneously. This allows for a simple testing of the
entropy-maximizing property of the IGO flow
(Proposition~\ref{prop:maxent}). Accordingly, for a given improvement on
the objective function, we would expect IGO to favor preserving the
diversity of $P_\theta$.
Indeed on a simple objective function
with two optima, we find that the IGO flow converges to a distribution
putting weight on both optima, while a descent using the vanilla gradient
always only concentrates around one optimum. This might be relevant for
multimodal optimization, where simultaneous handling of multiple
hypotheses is seen as useful for optimization 
\citep{SareniK_multimodal98,Suganthan_multimodal2011}, or in situations
in which several valleys appear equally good at first but only one of
them contains the true optimum.

We refer to Appendix~\ref{sec:RBM} for a reminder about Boltzmann
machines, for how IGO rolls out in this case, and for the detailed
experimental setup for studying the influence of IGO in a multimodal
situation.

\section{Summary and Conclusion}

We sum up:

\begin{itemize}
\item The information-geometric optimization (IGO) framework derives from
invariance principles the uniquely defined IGO flow
(Definition~\ref{def:IGOflow}), and allows,
via discretization in time and space, to build optimization algorithms from any family of distributions on any
search space. In some instances (Gaussian distributions on $\R^d$ or
Bernoulli distributions on $\{0,1\}^d$) it
recovers versions of known algorithms (PBIL, CMA-ES, cGA, NES); in other
instances (restricted Boltzmann machine distributions) it produces new,
hopefully efficient optimization algorithms.

\item The use of a quantile-based, time-dependent transform of the
objective function, Equation~\refeq{f-replacement}, provides a rigorous
derivation of rank-based update rules, frequently used in current optimization
algorithms. Theorem~\ref{thm:consistency} uniquely identifies the
infinite-population limit of these update rules.

\item The IGO flow is singled out by its equivalent description as an
infinitesimal weighted maximum log-likelihood update (Theorem~\ref{thm:IGOML}). 
In a particular
parametrization and with a step size of $1$, IGO recovers the
cross-entropy method (Corollary~\ref{cor:CEMIGO}). This reveals a
connection between CEM and the natural gradient, and allows to
define a new, fully parametrization-invariant smoothed maximum likelihood
update, the IGO-ML.

\item Theoretical arguments suggest that the IGO flow minimizes the
change of diversity in the course of optimization. 
\new{As diversity usually decreases in the course of optimization, 
IGO algorithms tend to exhibit minimal diversity loss for the observed improvement.}
In particular,
starting with high diversity and using multimodal distributions may allow
simultaneous exploration of multiple optima of the objective function.
Preliminary experiments with restricted Boltzmann machines confirm this
effect in a simple situation.

\end{itemize}

Thus, the IGO framework provides sound theoretical
foundations to optimization algorithms based on
probability distributions. 
In particular, this viewpoint helps to bridge the gap between continuous
and discrete optimization.

The invariance properties, which reduce the number of arbitrary choices,
together with the relationship between natural gradient and diversity,
may contribute to a theoretical explanation of the good practical
performance of those currently used algorithms, such as CMA-ES, which can be interpreted as
instantiations of IGO.

We hope that invariance properties
will acquire in computer science the importance they have in mathematics,
where intrinsic thinking is the first step for abstract linear algebra or differential
geometry, and in modern physics, where the notions of invariance w.r.t.\
the coordinate system and so-called gauge invariance play a central role.

\acks{The authors would like to thank Michèle Sebag for the acronym and for
helpful comments. We also thank Youhei Akimoto for helpful
feedback and inspiration.
Y.\,O.\ would like to thank Cédric Villani and Bruno Sévennec for helpful
discussions on the Fisher metric. A.\,A. and N.\,H. would like to acknowledge the Dagstuhl Seminar No~10361 on the Theory of Evolutionary Computation 
(\url{http://www.dagstuhl.de/10361}) for crucially inspiring this paper, their work on natural
gradients and beyond. This work was partially supported by the ANR-2010-COSI-002
grant (SIMINOLE) of the French National Research Agency.}

\vfill\eject
\appendix

\section{Further Discussion and Perspectives}
\label{sec:discussion}
This appendix touches upon further aspects and perspectives of the IGO framework and its implementations. 

\subsection{A Single Framework for Optimization on Arbitrary Spaces}
A strength of the IGO viewpoint is to automatically provide a distinct
and arguably \new{in some sense} optimal 
optimization algorithm from
any family of probability distributions on any given
space, discrete or continuous.
IGO algorithms feature desired
invariance properties and thereby make fewer arbitrary choices.

In particular, IGO describes several well-known optimization
algorithms within a single framework. For Bernoulli distributions, to the best of our knowledge, PBIL or cGA have never been identified as a natural gradient
ascent in the literature\footnote{Thanks to Jonathan Shapiro for giving an early
argument confirming this property (personal communication).}.
For Gaussian distributions, algorithms of the same form \eqref{eq:IGOupdate}
had been developed previously \citep{hansen2001,Wierstra2008} and their
close relationship with a natural gradient ascent had been recognized
\citep{Akimoto:ppsn2010, Glasmachers2010}. These works, however, also strongly suggest 
that IGO algorithms may need to be complemented with further heuristics to 
achieve efficient optimization algorithms. In particular, learning rate
settings and diversity control (step-size) deviate from the original
framework. These deviations mostly stem from 
time discretization and 
the choice of a finite sample size, both necessary to derive an IGO \emph{algorithm} from the 
IGO 
\emph{flow}. Further work is needed to fully understand this from a
theoretical viewpoint.

The wide applicability of natural gradient approaches seems not to be
widely known in the optimization community (though see
\citealt{MalagoGECCO2008InformationGeometryPerspectiveOfEDAs}).

\subsection{About Invariance} 
The role of invariance is two-fold. First,
invariance breaks otherwise arbitrary choices in the algorithm design.
Second, and more importantly, invariance implies generalization of
behavior from single problem instances to entire problem classes (with
the caveat to choose the initial state of the algorithm accordingly), thereby making the outcome of optimization more predictable. 

Optimization problems faced in reality---and algorithms used in
practice---are often far too complex to be amenable to a rigorous
mathematical analysis. Consequently, the judgement of algorithms in
practice is to a large extent based in empirical observations, either on
artificial benchmarks or on the experience with real-world problems.
Invariance, as a guarantee of generalization, has an immediate impact on the
relevance of any such empirical observations and, in the same line of
reasoning, can be interpreted as a notion of robustness. 

\subsection{About Quantiles} The IGO flow in \eqref{eq:IGOflow} has, to
the best of our knowledge, never been defined before.
The introduction of the quantile-rewriting \refeq{f-replacement} of the
objective function provides the first rigorous derivation of quantile- or
rank- or comparison-based optimization from a gradient ascent in $\theta$-space.

NES and CMA-ES have been claimed to maximize $-\E_\Ptheta f$ via
natural gradient ascent \citep{Wierstra2008,Akimoto:ppsn2010}.  However,
we have proved that the NES and CMA-ES updates actually converge to the IGO flow,
not to the similar flow with the gradient of $\E_\Ptheta f$
(Theorem~\ref{thm:consistency}).  So we find that in reality these
algorithms maximize $\E_{\Pt} \Wf$, 
where \Wf\ is a decreasing transformation of the $f$-quantiles under the current sample
distribution.

Moreover, in practice, maximizing $-\E_\Ptheta f$ tends to be a rather
unstable procedure and has been discouraged, see for example \citet{whitley1989genitor} 
and \citet{Sun:2009:ENE:1569901.1569976}.

\subsection{About Choice of $P_\theta$: Learning a Model of Good Points}
The choice of the family of probability distributions $P_\theta$ plays a
double role.

First, it is analogous to choosing the variation operators (namely \textit{mutation} or \textit{recombination}) as
seen in evolutionary algorithms: indeed, $P_\theta$ encodes possible moves according to which
new sample points are explored.

Second, optimization algorithms using distributions can be interpreted as
learning a probabilistic model of where the points with good values lie
in the search space. With this point of view, $P_\theta$ describes
\emph{richness of this model}: for instance, restricted Boltzmann
machines with $h$ hidden units can describe distributions with up to
$2^h$ modes, whereas the Bernoulli distribution is unimodal.
This influences, for instance, the ability to explore several valleys and
optimize multimodal functions in a single run.

More generally, the IGO framework makes it tempting to use more
complex models of where good points lie, inspired,  e.g., from machine learning,
and adapt them for optimization. The restricted Boltzmann machines of
Appendix~\ref{sec:RBM} are a first step in this direction. The initial
idea behind these machines is that each hidden unit controls a block of
coordinates of the search space (a block of features), so that the
optimization algorithm hopefully builds a good model of which features
must be activated or de-activated together to obtain good values of $f$.
This is somewhat reminiscent of a crossover operator: if observation of
good points shows that a block of features go together, this information
is stored in the RBM structure and this block may be later activated as a
whole, thus effectively transferring blocks of features from one good
solution to another.  Inspired by models of deep learning
\citep{Bengio2012}, one
might be tempted to stack such models on top of each other, so that
optimization would operate on a more and more abstract representation of
the problem.
IGO and the natural gradient might help in exploiting the added expressivity that comes with
richer models: in our simple experiment, the vanilla gradient ignores the
additional expressivity of RBMs with respect to Bernoulli distributions
(Appendix~\ref{sec:RBM}).
The downside of a richer model is that either the sample size $N$ must be
increased or the learning rate $\deltat$ must be decreased to obtain a
stable algorithm (a situation analogous to overfitting in machine
learning: richer models will be quicker to concentrate too much
around a few observed good samples).\niko{My interpretation of the failure mode I
have observed in rich models is somewhat different though: I thought they
fail due to random drift of parameters. Contemplating over this, I would
not even be 100\% sure anymore that richer models \emph{always} need a
larger $N/\deltat$. }\NDY{True. That's actually another way in which
overfitting manifests itself in practice...}

\subsection{Natural Gradient and Parametrization Invariance}
Central to IGO is the use of the natural gradient, which follows from
$\theta$-invariance (parametrization invariance) and makes sense on any search space,
discrete or continuous. 

While the IGO flow is exactly $\theta$-invariant, for any practical
implementation of an IGO algorithm, a parametrization choice has to
be made. Different parametrizations lead to different algorithms and larger values of \deltat\ are likely to result in more differing algorithms. 
Still, since all IGO algorithms approximate the IGO flow,
two parametrizations in combination with IGO will
differ less than the same two parametrizations 
in combination with another algorithm (such as the
vanilla gradient or the smoothed CEM method)---at
least if the learning rate $\deltat$ is not too large. 
The chosen parametrization becomes more relevant as
the step size $\deltat$ increases.

On the other hand, natural evolution strategies have not emphasized
$\theta$-invariance:  the chosen parametrization
(Cholesky, exponential) has been considered as a defining feature. We
believe the term ``natural evolution strategy'' should rather be used
independently of the chosen parameterization, thereby referring to
the usage of the natural gradient as the main principle for the update of distribution parameters.

\subsection{IGO, Maximum Likelihood and Cross-Entropy} The cross-entropy
method (CEM, \citealt{CEMtutorial}) can be used to produce
optimization algorithms given a family of probability distributions on an
arbitrary space, by performing a jump to a maximum likelihood
estimate of the parameters.

We have seen (Corollary~\ref{cor:CEMIGO}) that the standard CEM is an IGO
algorithm \emph{in a particular parametrization}, with a learning rate $\deltat$
equal to $1$. However, it is well-known, both theoretically and
experimentally \citep{branke2007addressing,hansen2006eda,WAGNER:2004:INRIA-00070802:1}, that standard CEM
loses diversity too fast in many situations. The usual solution
\citep{CEMtutorial} is to
reduce the learning rate (smoothed CEM, Equation\;\ref{eq:CEMsmooth}), but this
breaks the reparametrization invariance of non-smoothed CEM.

On the other hand, the IGO flow can be seen as a \emph{maximum likelihood
update with infinitesimal learning rate} (Theorem~\ref{thm:IGOML}). This
interpretation allows to define a particular IGO algorithm, the IGO-ML
(Definition~\ref{def:IGOML}): it performs a maximum likelihood update
with an arbitrary learning rate, and keeps the reparametrization
invariance. It coincides with CEM when the learning rate is set to $1$,
but it differs from smoothed CEM by the exchange of the order of
argmax and averaging (compare Equations\;\ref{eq:IGOML} and
\ref{eq:CEMsmooth}),
and coincides with the IGO flow for small learning rates.
We argue that this new, fully invariant algorithm is the conceptually
better way to reduce the learning rate and
achieve smoothing in CEM.

Standard CEM with rate $1$ can lose diversity, yet is a particular case of an
IGO algorithm: this illustrates the fact that reasonable values of the learning
rate $\deltat$ depend on the parametrization. We have studied this
phenomenon in detail for various Gaussian IGO algorithms
(Section~\ref{sec:gaussian}).

Why would a smaller learning rate perform better than a large one in
an optimization setting? It might seem more efficient to jump directly to
the maximum likelihood estimate of currently known good points, instead
of performing an iterated gradient ascent towards this maximum.
Due to having only a limited number of samples, however, optimization faces a ``moving target'', contrary to a learning
setting in which the example distribution is often stationary. Currently
known good points heavily depend on the current distribution and
are likely not to indicate the \emph{position} at which the optimum lies,
but, rather, the broad \emph{direction} in which the optimum is to be found.
After each update, the next elite sample points are going to be located
somewhere new.  
\del{So the goal is certainly not to settle down around these
currently known points, as a maximum likelihood update does: by design,
CEM only tries to reflect status-quo\niko{For infinite sample size picking the best sample seems to be the optimal strategy, so CEM seems optimal. I can see that for a fixed $q$ the current distribution remains present, while "tries to reflect status-quo" seems still strange, as it is rather a bug (wrong parameter $q$) than a feature in this case. For infinite sample size we will also never see premature convergence, independently of $q$ and the natural gradient will not perform better. }, whereas IGO
tries to move somewhere.  When the target moves over time, a progressive
gradient ascent is more reasonable than an immediate jump to a temporary
optimum, and realizes a kind of time smoothing.\niko{I think that 
\emph{all} this is just a question of the number of samples. Jumping right to the best samples is perfectly fine, or even optimal, if and only if the sample size is large enough, that is, if and only if, the best samples \emph{are} located in some way "over the" optimum. I find it super-important to emphasize that the current distribution heavily biases what we observe (as has been done). The reflecting status-quo part seems however a bit misleading, as the problem appears in particular with a large step-size (and only with a large step-size) also for IGO-ML, which is the opposite of status-quo. I could try to reformulate...}

\NDY{Actually I may even like it better if this and the next paragraph
were simply removed: I do not think we make any clear point in the end!
The whole should-we-lose-more-diversity-or-preserve-diversity argument is
just an equilibrium for the best $\deltat$ and truncation quantile $q_0$
given the current sample size, whether the algorithm is IGO-ML or CEM...
I suggest to simply stop this whole section at ``the conceptually better way to
reduce de learning rate and achieve smoothing in CEM''.}%
\niko{For the moment I cut quite some, but left the part where the point is made that everything we see hinges on the current distribution (unless $N\to\infty$). I think this is a very important point to make and indeed a crucial reason why we need to iterate. }

This phenomenon is most clear when the number of sample points is small.
Then, a full maximum likelihood update risks losing a lot of diversity;
it may even produce a degenerate distribution if the number of sample
points is smaller than the number of parameters of the distribution. On
the other hand, for smaller $\deltat$, the IGO algorithms do, by design,
try to maintain diversity by moving as little as possible from the
current distribution $P_\theta$ in Kullback--Leibler divergence. A full
ML update disregards the current distribution and tries to move as close
as possible to the elite sample in Kullback--Leibler divergence
\citep{CEMtutorial}, thus realizing maximal diversity loss. This makes
sense in a non-iterated scenario such as batch learning but 
is in general unsuited for optimization.}

\subsection{Diversity and Multiple Optima}
The IGO framework emphasizes the relation of natural gradient and diversity:
we argued that IGO provides minimal diversity change for a given
objective function improvement. In particular, provided the initial
diversity is large, diversity is kept at a maximum after the upate. This theoretical
relationship can be observed
experimentally for restricted Boltzmann
machines (Appendix~\ref{sec:RBM}).

On the other hand, using the vanilla gradient does not lead to a balanced
distribution between the two optima in our experiments.
Using the vanilla gradient introduces hidden arbitrary choices between
those points (more exactly between moves in $\Theta$-space). This results
in unnecessary and unwelcome
loss of diversity, and might also be detrimental at later stages in
the optimization.
This may reflect the fact that the Euclidean metric on the
space of parameters, implicitly used in the vanilla gradient, becomes
less and less meaningful for gradient descent on complex distributions.

IGO and the natural gradient are certainly relevant to the
well-known problem of exploration-exploitation balance: as we have seen,
arguably the natural gradient realizes the largest improvement in the objective with the least possible change of diversity in the
distribution.

More generally, like other distribution-based optimization algorithms,
IGO tries to learn a model of where the good
points are. This is typical of machine learning, one of the
contexts for which the natural gradient was studied. The conceptual
relationship of IGO and IGO-like optimization algorithms with machine
learning is still to be explored and exploited.  

\del{Learning a model of where the good points are. Cf discussion the other day:
if exponential number of near-minima with only one true minimum, have a
model of where the near-minima are without
re-running. The example was something like
$f=0.001\sum x_i + (\sum_{i=0}^{N/30-1}x_i)(N/30-\sum_{i=0}^{N/30-1}
x_i)+(\sum_{i=N/30}^{2N/30-1}x_i)(N/30-\sum_{i=N/30}^{2N/30-1} x_i)+\ldots$
which has $2^{30}$ near-minima but only one true minimum. In the best of cases, an RBM with more than 30
hidden units may be able to learn that.

Diversity, balance exploration-exploitation: we have the following things which are equivalent: 
\begin{itemize}
\item minimal KL-divergence between $\theta^t$ and $\theta^{t+\delta t}$
\item minimal diversity change
\end{itemize} 
both for a given change in objective function. This implies minimal
diversity loss, given initially large diversity. }

\del{\subsubsection{Further Related work.}}
\del{somewhere probably not here, relate to work by Berny or Gallagher
taking the vanilla gradient of the KL divergence.\NDY{We concluded it's really
not the same and that explaining the lack of relationship would be rather
confusing, I think.}}

\bigskip

We now present some ideas which we believe would be worth exploring.

\subsection{Adaptive Learning Rate}
Comparing consecutive updates
to evaluate a learning rate or step size is an effective measure. For
example, in back-propagation, the update sign has been used to adapt the
learning rate of each single weight in an artificial neural network \citep{silva1990acceleration}.
In CMA-ES, a step size is adapted
depending on whether recent steps tended to move in a consistent
direction or to backtrack. This is measured by considering the changes
of the mean $m$ of the Gaussian distribution.

For a probability
distribution $P_\theta$ on an arbitrary search space, in general no
notion of mean may be defined. However, it is still possible to define
``backtracking'' in the evolution of $\theta$ as follows.

Consider two successive updates $\delta\theta^t=\theta^t-\theta^{t-\deltat}$ and
$\delta\theta^{t+\deltat}=\theta^{t+\deltat}-\theta^t$. Their
scalar product in the Fisher metric $I(\theta^t)$ is
\[
\langle \delta\theta^t, \delta\theta^{t+\deltat}\rangle=\sum_{ij}
I_{ij}(\theta^t) \, \delta\theta^t_i \, \delta\theta^{t+\deltat}_j 
\enspace.
\]
Dividing by the associated norms will yield the cosine
$\cos \alpha$ of
the angle between $\delta\theta^t$ and $\delta\theta^{t+\deltat}$.

If this cosine is positive, the learning rate $\deltat$ may be increased. If
the cosine is negative, the learning rate probably needs to be
decreased. Various schemes for the change of $\deltat$ can be devised; for instance,
inspired by step size updates commonly used in evolution
strategies,
one can multiply $\deltat$ by $\exp(\beta(\cos\alpha))$ or
$\exp(\beta\sign(\cos\alpha))$,\niko{more general: 
$\exp^\beta(\sign(\cos\alpha)\;|\cos\alpha|^\gamma)$}
where $\beta\approx\min(N/\dim\Theta, 1/2)$. 
As this cosine can be quite noisy, cumulation over several time steps might be advisable.

As before, this scheme is constructed to be robust
w.r.t.\ reparametrization of $\theta$, thanks to the use of the Fisher
metric. However, for large learning rates \deltat, in practice the parametrization might well become relevant.

\del{A consistent direction of the updates does not necessarily
mean that the algorithm is performing well: for instance, when CEM/EMNA
exhibits premature convergence (see above), the parameters consistently
move towards a zero covariance matrix and the cosines above are positive.
This indicates too small steps, as the desired target value for the cosine is zero. }

\subsection{Geodesic Parametrization} While the IGO flow is fully
invariant under $\theta$-reparametrization, an IGO algorithm does depend
on the choice of parametrization for $\theta$, even if for small
$\deltat$ the difference between two IGO algorithms is $O(\deltat^2)$, one order of
magnitude smaller than between IGO and vanilla gradient in general.

So one can wonder how to discretize time in the IGO flow in a fully
intrinsic way, not depending at all on a parametrization for $\theta$. A
first possibility is given by the IGO-ML algorithm
(Definition~\ref{def:IGOML})---this means, for exponential families, that
we can decide to single out the parametrization by expectation
parameters.

Another, more geometric solution is to use \emph{geodesics} on the
statistical manifold. This means we approximate the trajectories of the
IGO flow by successive geodesic segments of length $\deltat$ in the
Fisher metric, where the initial direction of each segment is given by
the direction of the IGO flow \eqref{eq:intrinsicIGOupdate}.
\del{

More precisely, if
\begin{equation*}
Y
=
\sum_{i=1}^{N} \wi \, \left.\widetilde\nabla_\theta \ln
P_\theta(x_i)\right|_{\theta=\theta^t}
=
I^{-1}(\theta^t)\,
\sum_{i=1}^{N} \wi \, \left.\frac{\partial \ln P_\theta(x_i)}{\partial
\theta}\right|_{\theta=\theta^t}
\end{equation*}
is the direction of the IGO update \eqref{eq:intrinsicIGOupdate} at $\theta^t$, one can define
\[
\theta^{t+\deltat}=\exp_{\theta^t} (\deltat.Y)
\]
where $\exp$ is the exponential map of the Riemannian manifold $\Theta$
equipped with the Fisher information metric.

}
This defines an approximation to the IGO flow that depends on the step
size $\deltat$ and sample size $N$, but \emph{not} on any choice of
parametrization. This approach is fully developed
in~\citet{BensadonGIGO}.

\del{
\NDY{The formal definition was a bit much and pretentious given we don't really have
anything to say about it...}

\begin{defi}[Geodesic IGO update]
\label{def:geoIGO}
The \emph{geodesic IGO update} with sample size $N$ and step size $\deltat$
is the following update rule for the parameter $\theta^t$.
At each step, $N$ sample points $x_1,\ldots,x_N$ are drawn according to
the distribution $P_{\theta^t}$. Let $Y$ be the vector in
$\theta$-space giving the direction of the IGO update at $\theta^t$ i.e.
\[
Y
=
\sum_{i=1}^{N} \wi \, \left.\widetilde\nabla_\theta \ln P_\theta(x_i)\right|_{\theta=\theta^t}
=
I^{-1}(\theta^t)\,
\sum_{i=1}^{N} \wi \, \left.\frac{\partial \ln P_\theta(x_i)}{\partial
\theta}\right|_{\theta=\theta^t}
\]
Then the new value $\theta^{t+\deltat}$ of the parameter $\theta$ is
obtained by following the geodesic (in the Fisher metric) whose initial
point is $\theta$ and whose initial velocity is the tangent vector
$Y$,
i.e.
\[
\theta^{t+\deltat}=\exp_{\theta^t} (\deltat.Y)
\]
where $\exp$ is the exponential map of the Riemannian manifold $\Theta$
equipped with the Fisher information metric.
\end{defi}

Then the geodesic IGO update does not depend on a choice of
parametrization for $\theta$.
}

\del{
Practical implementation will depend on being able to compute the geodesics
of the Fisher metric. The equation of geodesics may be computed explicitly in
some particular cases \citep{Burbea1986}, \cite[Chapter
5]{Amarietal1987}, such as Bernoulli distributions or Gaussian
distributions with a fixed mean or with a fixed covariance matrix.
Interestingly, for Gaussian distributions with a fixed mean, the geodesic
update resembles the one in xNES.

When no closed formula for geodesics is available, $\theta^{t+\deltat}$
can always be found by numerically integrating the geodesic equation
starting at
$\theta^t$ with initial speed $Y$. This is, of course,
an added computational cost, but it does not require any calls to the
objective function $f$.
}

\subsection{Finite Sample Size and Noisy IGO Flow} The IGO flow is an
ideal model of the IGO algorithms. But the IGO
flow is deterministic while IGO algorithms are stochastic, depending on a
finite number $N$ of random samples. This might result in important
differences in their behavior and one can wonder if there is a way to
reflect stochasticity directly in the definition of the IGO flow.

The IGO update~\eqref{eq:intrinsicIGOupdate} is a stochastic update
\[
\theta^{t+\deltat}
=
\theta^t+\deltat \sum_{i=1}^{N} \wi \, \left.\widetilde\nabla_\theta \ln P_\theta(x_i)\right|_{\theta=\theta^t}
\]
because the term $\sum_{i=1}^{N} \wi \, \left.\widetilde\nabla_\theta \ln
P_\theta(x_i)\right|_{\theta=\theta^t}$ involves a random sample. As
such, this term has an expectation and a variance.
So for a fixed $N$ and $\deltat$, this random update is a weak
approximation with step size $\deltat$ \cite[Chapter 9.7]{KloedenPlaten92} of a stochastic differential equation on
$\theta$, whose drift is the expectation of the IGO update (which tends
to the IGO flow when $N\to\infty$), and whose
noise term is $\sqrt{\deltat}$ times the square root of the covariance
matrix of the update applied to a normal random vector.

Such a stochastic differential equation, defining a \emph{noisy IGO
flow}, might be a better theoretical object with which to compare the
actual behavior of IGO algorithms, than the ideal noiseless IGO flow.

For instance, this strongly suggests that if we have $\deltat\to 0$ while
$N$ is kept fixed in an IGO algorithm, noise will disappear (compare
Remark~2 in \citealt{Akimoto2012ppsn}).

Second, for large $N$, one expects the variance of the IGO update to
scale like $1/N$, so that the noise term will scale like
$\sqrt{\deltat/N}$.
This formally suggests that, within reasonable bounds,
multiplying or dividing both $N$ and $\deltat$ by the same factor should
result in similar behavior of the algorithm, so that for instance it
should be reasonable to reset
$\deltat$ to $10\deltat/N$ and $N$ to $10$.
(Note that the cost in terms of $f$-calls of these two algorithms is
similar.)

This dependency is reflected in evolution strategies in several
ways, provided $N$ is smaller than the search space dimension.
First,\del{ asymptotic} theoretical results for large search space
dimension on the sphere function, $f(x)=\|x\|^2$, indicate that the
optimal step size \deltat\ for the mean vector is proportional to $N$,
provided the weighting function $w$ is either truncation selection with a fixed truncation ratio \citep{HGBeyer01} or optimal weights \citep{arnold2006weighted}. Second, the learning rate \deltat\ of the covariance matrix in CMA-ES is chosen proportional to $\left(\sum_{i=1}^{N} \wi\right)^2 / \sum_{i=1}^{N} \wi^2 $ which is again proportional to $N$ \citep{hansen2004}. For small enough $N$, the progress per $f$-call is then in both cases rather independent of the choice of $N$. 

These results suggest that in implementations, $N$ can be chosen rather freely, whereas $\deltat$ will be set to $\mathrm{const}\cdot N$. The constant is chosen small enough to ensure stability, 
but as large as possible to maximize speed;
still, $\deltat\le1$ is another constraint for (very) large $N$.

\subsection{Influence of the Fisher Geometry of the Statistical Manifold}
The global Riemannian geometry of the statistical manifold $P_\theta$
might have a bearing on the behavior of stochastic algorithms exploring
this manifold. For instance, the Fisher metric identifies the set of
$1$-dimensional normal distributions $\mathcal{N}(m,\sigma^2)$ with the
two-dimensional hyperbolic plane. The latter has negative curvature. The
sign of curvature has a strong influence on the behavior of random walks
in a Riemannian manifold: in particular, in negative curvature,
successive random errors tend to not compensate as much as in the
Euclidean case (because geodesics diverge more quickly); this might be
relevant to the settings of a stochastic optimization algorithm,
suggesting to use larger sample size or smaller steps when curvature is
negative; \citet{BensadonGIGO} provides first observations in this
direction, associated with negative curvature in the space of Gaussians. 
This is speculative and remains to be explored.

\del{
\begin{itemize}
\item understanding diagonal FIM, intrinsic bias of a parametrization
(=spontaneous drift of the stochastic IGO update on a constant function,
cf. \texttt{diagfisher.tex}), etc.
\item Resampling for Fisher matrix evaluation
\item More RBM experiments with more modes and hidden units, on more
complex functions.
\end{itemize}
}

\section{Implementing an IGO Algorithm with a New Family of Probability
Distributions: Restricted Boltzmann Machines}
\label{sec:RBM}

In this \new{appendix}\del{section} we show how to apply IGO to 
restricted Boltzmann machines (RBMs,
\citealt{Smolensky1986,Ackley1985}), a family of probability distributions
on the discrete hypercube that extends Bernoulli distributions and can
represent correlations between variables.

This first illustrates how to set up an IGO algorithm from a family of
probability distributions for which no fully explicit expression for the
Fisher information matrix is available.

Second, we test the relationship between IGO and diversity of $p_\theta$,
in view of the entropy-maximizing property of the IGO flow
(Proposition~\ref{prop:maxent}). We consider a simple function with two
optima, and compare the behavior of the IGO flow with that of a vanilla
gradient flow. The latter always converges to a single optimum, even
though RBMs allow for multimodal distributions. On the other hand, IGO
seems to always converge to both optima at once, taking advantage of RBM
multimodality.

Finally we interpret this observations by pointing out a non-obvious
breach of symmetry between $0$ and $1$ on $\{0,1\}^d$ for the vanilla
gradient of RBMs; the natural gradient automatically compensates for this.

\subsection{Restricted Boltzmann Machines.} RBMs first define a joint
distribution on $\mathbf{x}\in \{0,1\}^d$ together with a \emph{hidden} or
\emph{latent} variable
$\mathbf{h}\in\{0,1\}^{d_h}$ \citep{Ghahramani2004}. Summation over $\mathbf{h}$ provides the distribution over
$\mathbf{x}$.
The probability associated with an observation $\mathbf{x}=(x_i)\in
\{0,1\}^d$ and latent
variable $\mathbf{h}=(h_j)\in \{0,1\}^{d_h}$ is
\begin{equation}
P_\theta
(\mathbf{x},\mathbf{h})=\frac{e^{-E(\mathbf{x},\mathbf{h})}}{\sum_{\mathbf{x'},\mathbf{h'}}e^{-E(\mathbf{x'},\mathbf{h'})}}
\,,
\qquad P_\theta(\mathbf{x})=\sum_{\mathbf{h}} P_\theta
(\mathbf{x},\mathbf{h})
,\label{eq:rbm-proba}
\end{equation}
where the \emph{energy function $E$} is
\begin{equation}
E(\mathbf{x},\mathbf{h})=-\sum_{i}a_{i}x_{i}-\sum_{j}b_{j}h_{j}-\sum_{i,j}w_{ij}x_{i}h_{j}\;\label{eq:rbm-energy}
\end{equation}
(compare Section \ref{sec:exp}).
The distribution is parametrized by
$\theta = (\mathbf{a}, \mathbf{b}, \mathbf{w})$.
The ``biases'' $\mathbf{a}$ and $\mathbf{b}$ can be subsumed into the
weights $\mathbf{w}$ by introducing variables $x_0$ and $h_0$ always
equal to one, which we implicitly do from now on.

The hidden variable $\mathbf{h}$ introduces correlations between the
components of $\mathbf{x}$ (see Figure~\ref{fig:rbm}). For instance, with
$d_h=1$, the distribution on $\mathbf{x}$ is the sum of two
Bernoulli distributions with $\mathbf{h}$ acting as a ``switch''.

\begin{figure}
\begin{centering}
\includegraphics[scale=.7]{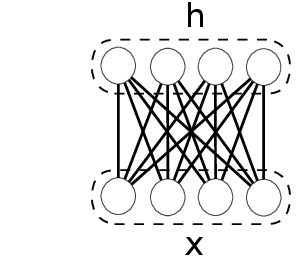}
\par\end{centering}
\caption{The RBM architecture with the observed ($\mathbf{x}$) and latent
($\mathbf{h}$)
variables. In our experiments, a single hidden unit was used. }
\label{fig:rbm}
\end{figure}

\subsection{IGO for RBMs.} RBMs constitute an exponential family with
latent variables (with the statistics $T(\mathbf{x},\mathbf{h})$ being
all the $x_ih_j$). So the IGO flow \eqref{eq:IGOexplicit} is given by
\eqref{eq:hexpIGO} where the first contribution is
the Fisher matrix
\eqref{eq:fishhexp} and the second contribution is the log-probability
derivatives weighted by the objective function values. However, these expressions are not explicit due to the
expectations under $P_\theta$. Still, these expectations can be replaced
with Monte Carlo sampling; this is what we will do.

Actually, when applying IGO to such distributions to optimize an
objective function $f(\mathbf{x})$, we have two choices. The first is to
see the objective function $f(\mathbf{x})$ as a function of
$(\mathbf{x},\mathbf{h})$ where $\mathbf{h}$ is a dummy variable; then we
can use the IGO algorithm to optimize over $(\mathbf{x},\mathbf{h})$
using the distributions $P_\theta(\mathbf{x},\mathbf{h})$. A second
possibility is to marginalize $P_\theta(\mathbf{x},\mathbf{h})$ over the
hidden units $\mathbf{h}$ as in \eqref{eq:rbm-proba}, to define the
distribution $P_\theta(\mathbf{x})$; then we can use the IGO algorithm to
optimize $f$ over $\mathbf{x}$ using $P_\theta(\mathbf{x})$.

These two approaches yield slightly different algorithms. The Fisher
matrix for the distributions $P_\theta(\mathbf{x},\mathbf{h})$ is given by
\eqref{eq:fishexp} (exponential families) whereas the one for the distributions
$P_\theta(\mathbf{x})$ is given by \eqref{eq:fishhexp} (exponential
families with latent variables). For instance, 
with $I_{w_{ij}w_{i'j'}}$ denoting the entry of the Fisher matrix
corresponding to the components $w_{ij}$ and $w_{i'j'}$ of the parameter
$\theta$, from \eqref{eq:fishexp} we get
\begin{equation}
\label{eq:fishrbm}
I_{w_{ij}w_{i'j'}}(\theta)= \Cov_{P_\theta}(x_ih_j,x_{i'}h_{j'})=\E_{P_\theta} [x_i h_j x_{i'} h_{j'}] -
\E_{P_\theta}[ x_i h_j ] \,\E_{P_\theta}[ x_{i'} h_{j'} ] \\
\end{equation}
whereas from \eqref{eq:fishhexp} we get the same expression in which each
$h_j$ is replaced with its expectation $\bar h_j$ knowing $\mathbf{x}$ namely $\bar
h_j=\E_{P_\theta}[h_j|\mathbf{x}]=\left(
1+e^{-b_j-\sum_i x_i w_{ij}}
\right)^{-1}$ and likewise for $h_{j'}$.

Both versions were tested on a small instance and found to
be viable.  However the version using $(\mathbf{x},\mathbf{h})$ is
numerically more stable and requires fewer samples to get a reliable (in
particular, invertible) estimate of the Fisher
matrix, both in practice and theory
\footnote{Indeed, if $I_1(\theta)$ is the Fisher matrix
at $\theta$ in the first approach and $I_2(\theta)$ in the second
approach, we always have $I_1(\theta)\geq I_2(\theta)$ in the sense of
positive-definite matrices. This is because probability distributions on
the pair $(\mathbf{x},\mathbf{h})$ carry more information than their
projections on $\mathbf{x}$ only, and so the Kullback--Leibler distances
will always be larger.  In particular, there exist values of $\theta$ for
which the Fisher matrix $I_2$ is not invertible whereas $I_1$ is.
}.
For this reason we selected the first approach: we optimize $f$
as a function of $(\mathbf{x},\mathbf{h})$ using IGO for the probability
distributions $P_\theta(\mathbf{x},\mathbf{h})$. 

A practical IGO update is thus obtained from \eqref{eq:expIGO} by
replacing the expectations with Monte Carlo samples
$(\mathbf{x},\mathbf{h})$ from $P_{\theta}$. A host of strategies are
available to sample from RBMs \citep{Ackley1985,Salakhutdinov2008,
Salakhutdinov2009a,Desjardins2010}; we simply used Gibbs sampling
\citep{Hinton2002}. So the IGO update reads
\begin{equation}
\theta^{t+\deltat}=\theta^t+\deltat \,
\widehat\Cov(T,T)^{-1} \, \widehat\Cov(T,\Wf)
\end{equation}
where $T$ denotes the vector of all statistics
$T_{ij}(\mathbf{x},\mathbf{h})=x_ih_j$ corresponding to the component
$w_{ij}$ of the parameter $\theta$, and where $\widehat\Cov$ denotes the
empirical covariance over a set of Monte Carlo samples
$(\mathbf{x},\mathbf{h})$ taken from
$P_{\theta^t}(\mathbf{x},\mathbf{h})$. Thus $\hat I=\widehat\Cov(T,T)$ is
the estimated Fisher matrix, a
matrix of size $\dim(\theta)^2$ as in \eqref{eq:fishrbm}. 
$\widehat\Cov(T,\Wf)$ is a vector of size $\dim(\theta)$ giving the correlation, in
the Monte Carlo sample, between the statistics $T_{ij}$ and the ranked
values $\Wf$ of the objective function of the points
$(\mathbf{x},\mathbf{h})$ in the sample: this vector is the sum of
weighted log-probability derivatives in the IGO update
\eqref{eq:IGOupdate}, thanks to \eqref{eq:gradexp}.

Different sample sizes $N_\mathrm{Fish}$ and $N$ can be used to evaluate
$\widehat\Cov(T,T)$ and $\widehat\Cov(T,\Wf)$. The sample size for
$\widehat\Cov(T,\Wf)$
is just the IGO parameter $N$, the number of points on which the
objective function has to be evaluated; it is typically kept small
especially if $f$ is costly. On the other hand it is important to obtain
a reliable (in particular, invertible) estimate of the Fisher matrix:
invertibility requires a number of samples at least, and ideally much
larger than, $\dim(\theta)$, because each point in the sample contributes a
rank-1 term to the empirical covariance matrix $\widehat\Cov(T,T)$.
Increasing $N_\mathrm{Fish}$ does not require additional $f$-calls.

\subsection{An Experiment with Two Optima: IGO, Diversity, and Multimodal
Optimization.} We tested the resulting IGO trajectories on a simple objective function
with two optima on $\{0,1\}^d$, namely, the \emph{two-min function based
at $\mathbf{y}$} defined as
\begin{equation}
\label{eq:twomin}
f_{\mathbf{y}}(\mathbf{x}) = \min \left(  \sum_i \left| x_i - y_i \right|, \sum_i \left| (1-x_i) - y_i ) \right| \right)
\end{equation}
which, as a function of $\mathbf{x}$, has two optima, one at
$\mathbf{x}=\mathbf{y}$ and the other at its binary complement
$\mathbf{x}=\bar{\mathbf{y}}$. The value of the base point $\mathbf{y}$ was
randomized for each independent run.

We ran both the IGO algorithm as described above, and a version using
the vanilla gradient instead of the natural gradient (that is, omitting
the Fisher matrix in the IGO update).

The dimension was $d=40$ and we used an RBM with only one latent variable
($d_h=1$). Therefore $\dim(\theta)=81$.
We used a large sample size of $N=10,000$ for Monte Carlo sampling, so as
to be close to the theoretical IGO flow behavior. We also tested a
smaller, more realistic sample size of $N=10$  (still keeping
$N_\mathrm{Fish}=10,000$), with similar but noisier results.
The selection scheme (Section \ref{sec:igo}) was
$w(q)=\mathbbm{1}_{q\leq1/5}$ (cf.\ \citealt{Rechenberg:94})
so that the best $20\%$ points in the
sample are given weight $1$ for the update.

The RBM was initialized so that at startup, the distribution
$P_{\theta^0}$ is close to uniform on $(\mathbf{x},\mathbf{h})$, in line
with Proposition~\ref{prop:maxent}. Explicitly we set $w_{ij}\gets
\mathcal{N}(0,\frac{1}{d.d_h})$ and then $b_j \gets - \sum_i
\frac{w_{ij}}{2}$ and  $a_i \leftarrow - \sum_j \frac{w_{ij}}{2} +
\mathcal{N}(0,\frac{0.01}{d^2})$ which ensure a close-to-uniform initial
distribution.

Full experimental details, including detailed setup and additional
results, can be found in \new{a previous}\del{the first} version of this article \citep[Section
5]{IGOpreprintv2}. (In
particular, IGO runs are frozen when the estimated Fisher matrix
becomes singular or unreliable.)  The code
used for these experiments can be found at
\url{http://www.ludovicarnold.com/projects:igocode} .

\bigskip

The results show that, most of the time, with IGO
the distribution $P_{\theta^t}$ converges to a distribution giving
positive mass to both optima; on the other hand, over $300$ independent
runs, the same algorithm using the vanilla gradient only converges to one
optimum at the expense at the other, so that the vanilla gradient
\emph{never} exploited the possibility offered by the RBM to create a
bimodal probability distribution on $\{0,1\}^d$.
Figure~\ref{fig:large-pop-sr} shows ten random runs (out of $300$ in our
experiments) of the
two algorithms\del{, and plots the distance between the
$P_{\theta^t}$-sample and the two optima of the objective
function —(reviewer)Explain better what is depicted in fig. 3. When you write "the distance between the $P_{\theta^t}$-sample and the two optima", do you mean
* the distance of the closest individual for each of the optimum, or
* the average distance of all the samples individuals which are closer to the respective optimum, or
* ...?}: for each of the two optima we plot its
distance to the nearest of the points drawn from $P_{\theta^t}$, as a function of time $t$.\footnote{Note that the value of $\deltat$ is not directly
comparable between the natural and vanilla gradients. Theory suggests that at startup the IGO trajectory
with $\deltat$ is most comparable to the vanilla gradient trajectory with
$4\deltat$, because from \eqref{eq:fishrbm} most of the diagonal terms of the Fisher
matrix are equal to $1/4$ and most off-diagonal terms are $0$ at startup. The experiments confirm that this
yields roughly comparable convergence speeds.}

\begin{figure}[t]
\begin{centering}
\includegraphics[width=1.\columnwidth]{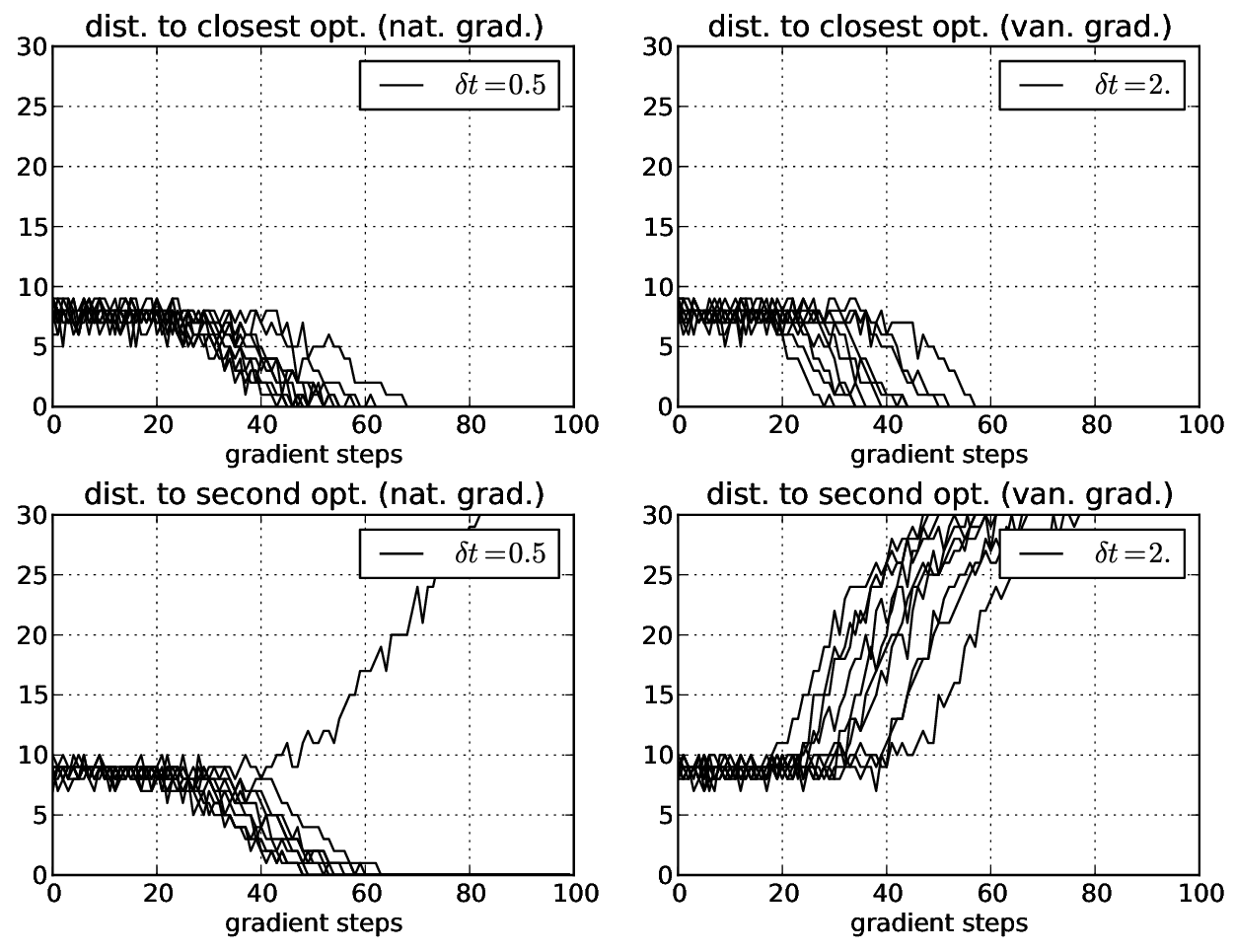}
\par\end{centering}
\caption{Distance to the two optima 
during 10 IGO optimization runs and 10 vanilla gradient runs. For each
optimum, we plot its distance to the nearest point among the samples from
$P_{\theta^t}$ found at each step.%
}
\label{fig:large-pop-sr}
\end{figure}

This is even clearer when looking at the hidden variable $\mathbf{h}$.
With $d_h=1$ the two possible values $h=0$ and $h=1$ can create a sum of
two Bernoulli distributions on $\mathbf{x}$.
Figure~\ref{fig:large-pop-hstats} plots the average value of $h$ in the
sample for IGO and for the vanilla gradient. As can be seen, over IGO
optimization the balance between the two modes is preserved, whereas
using vanilla gradient optimization the zero mode is lost.

\begin{figure}[t]
\centering \includegraphics[width=.9\columnwidth]{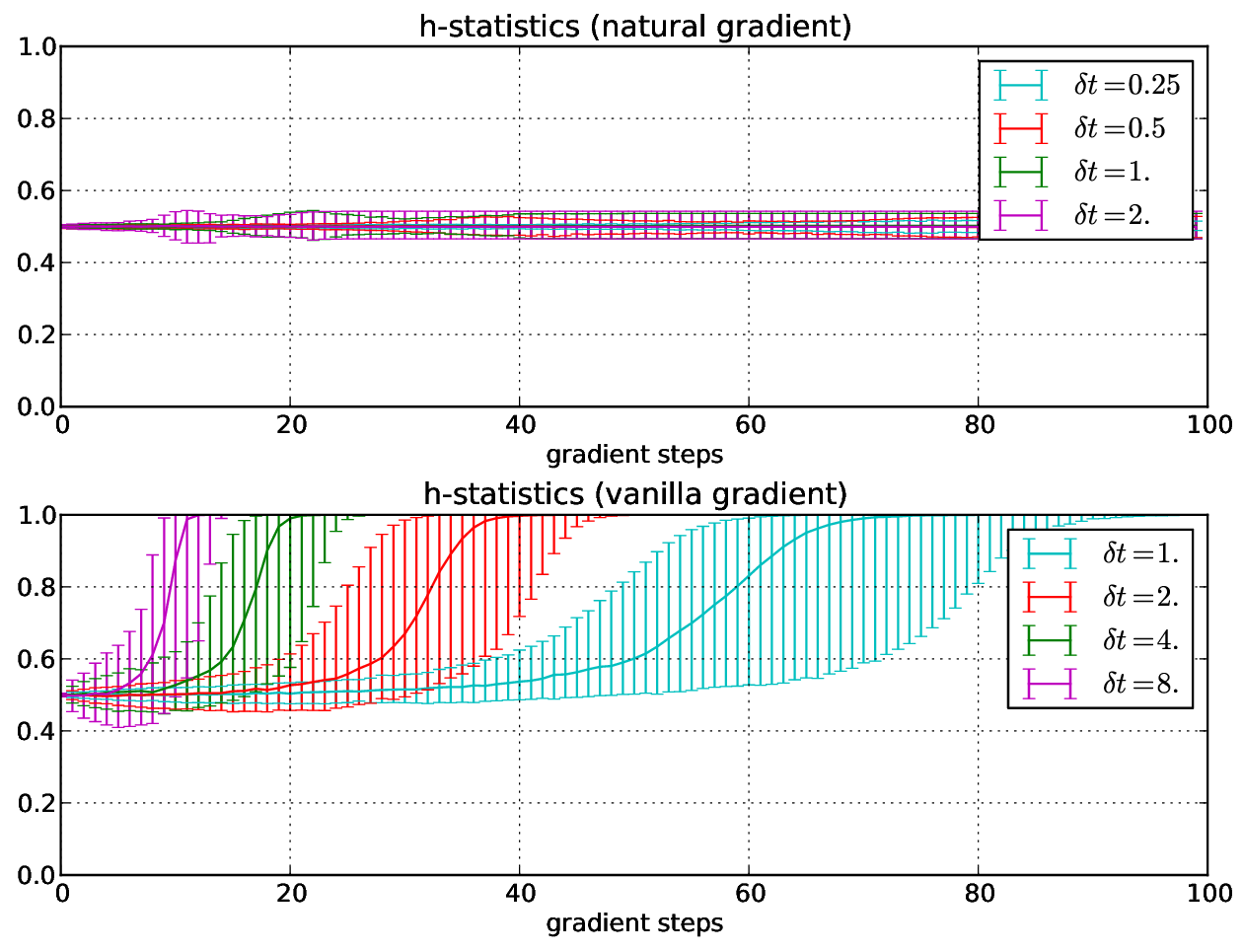}
\caption{Median over 300 runs of the average of $h$ in the sample, for
IGO and vanilla gradient optimization.
Error bars indicate
the 16th and 84th quantile over the runs.}
\label{fig:large-pop-hstats}
\end{figure}

We interpret this as an illustration of Proposition~\ref{prop:maxent}:
for a given improvement on the objective function, IGO will favor
preserving the diversity (entropy) of $P_\theta$. Using more richly
multimodal distributions (e.g., RBMs with $d_h>1$), this might be useful
for multimodal optimization or for optimization situations in which there
are several almost equally deep valleys only one of which contains the
true optimum.

\subsection{Breach of Symmetry by the Vanilla Gradient.} This
experiment  reveals that, curiously,
the vanilla gradient loses multimodality by always setting the hidden
variable $h$ to $1$, not to $0$ (Figure~\ref{fig:large-pop-hstats}).  So the vanilla gradient for RBMs seems to favor
$h=1$.

Of course, exchanging the values $0$ and $1$ for the hidden variables in
a restricted Boltzmann machine still gives a distribution of another
Boltzmann machine. Changing $h_j$ into $1-h_j$ is equivalent to
resetting $a_i\leftarrow a_i+w_{ij}$, $b_j\leftarrow -b_j$, and
$w_{ij}\leftarrow -w_{ij}$.

IGO
and the natural gradient are impervious to such a change
by Proposition~\ref{prop:Xinv}.
But the vanilla gradient implicitly relies on the Euclidean norm on parameter
space, as explained in Section~\ref{sec:grad}. For this norm, the distance between the RBM distributions
$(a_i,b_j,w_{ij})$ and $(a'_i,b'_j,w'_{ij})$ is $\sum_i
\abs{a_i-a'_i}^2 + \sum_j \abs{b_j-b'_j}^2+\sum_{ij}
\abs{w_{ij}-w'_{ij}}^2$. However, the change of variables $a_i\leftarrow
a_i+w_{ij},\; b_j\leftarrow -b_j,\; w_{ij}\leftarrow -w_{ij}$ does \emph{not} preserve this
Euclidean metric. Thus, exchanging $0$ and $1$ for the hidden variables
will result in two different vanilla gradient ascents. The observed asymmetry
on $h$ is a consequence of this implicit asymmetry. 


Of course it is possible to use 
parametrizations for which the vanilla
gradient will be more symmetric: for instance, using $-1/1$ instead of $0/1$ for the
variables, or rewriting the energy as
\begin{equation}
E(\mathbf{x},\mathbf{h})=-\tsum_{i}A_{i}(x_{i}-\tfrac12)-\tsum_{j}B_{j}(h_{j}-\tfrac12)-\tsum_{i,j}W_{ij}(x_{i}-\tfrac12)(h_{j}-\tfrac12)
\end{equation}
with ``bias-free'' parameters $A_i,B_j,W_{ij}$ related to the usual
parametrization by $w_{ij}=W_{ij}$, $a_i=A_i-\frac12 \sum_j
w_{ij}$, and $b_j=B_j-\frac12\sum_i w_{ij}$.
The vanilla gradient might perform better in this parametrization.

However, we adopted the approach of
using a family of probability
distributions found in the literature, with the parametrization
commonly found in
the literature.
We then used the vanilla gradient and the natural gradient
on these distributions---and indeed the vanilla gradient, or an approximation thereof, is routinely
applied to RBMs in the literature to optimize the log-likelihood of
data \citep{Hinton2002,Hinton2006,Bengio2007a}. It was not obvious a
priori (at least for us) that the vanilla gradient ascent favors $h=1$.
Since the first version of this article was written, this phenomenon
has been recognized for Boltzmann machines \citep{CenteringTrick2012}.

This directly illustrates the specific influence of the chosen gradient
(the two implementations only differ by the inclusion of the Fisher
matrix): the natural gradient offers a systematic way to recover symmetry
from a non-symmetric gradient update.

Symmetry alone does not explain why IGO reaches the
two optima simultaneously: a symmetry-preserving stochastic
algorithm could as well end up on either single optimum with 50\%
probability in each run. The diversity-preserving property of IGO
(Proposition~\ref{prop:maxent}) offers a
reasonable interpretation of why this does not happen.

\section{Further Mathematical Properties of the IGO Flow}
\label{sec:moremath}
\new{This appendix provides further mathematical properties of the IGO flow in general and in specific scenarios.}
\subsection{Invariance Properties}
\label{sec:invariance}

Here we formally state the
invariance properties of the IGO flow under various reparametrizations.
Since these results follow from the very construction of the algorithm,
the proofs are omitted.

\begin{prop}[$f$-invariance]
\label{prop:finv}
Let $\phi:\R\to\R$ be a strictly increasing function. 
Then the trajectories of the IGO flow when optimizing the functions
$f$ and $\phi(f)$ are the same.

The same is true for the discretized algorithm with population size $N$ and
step size $\deltat>0$.
\end{prop}

\begin{prop}[$\theta$-invariance]
\label{prop:thetainv}
Let $\theta'=\phi(\theta)$ be a smooth bijective function of $\theta$ and let
$P'_{\theta'}=P_{\phi^{-1}(\theta')}$.
Let $\theta^t$ be the trajectory of the IGO flow when optimizing a
function $f$ using the distributions $P_\theta$, initialized at $\theta^0$.
Then
the IGO flow trajectory $(\theta')^t$ obtained from the optimization of
the function $f$ using the distributions
$P'_{\theta'}$, initialized at
$(\theta')^0=\phi(\theta^0)$, is the same, namely
$(\theta')^t=\phi(\theta^t)$.
\end{prop}

For the algorithm with finite $N$ and $\deltat>0$, invariance under
reparametrization of $\theta$ is only true approximately, in the limit when
$\deltat\to 0$. As mentioned above, the IGO update \eqref{eq:IGOupdate}, with
$N=\infty$, is simply the
Euler approximation scheme for the ordinary differential equation
\eqref{eq:IGOflow} defining the IGO flow.
At each step, the Euler scheme
is known to make an error $O(\deltat^2)$ with respect to the true flow.
This error actually depends on the parametrization of $\theta$.

So the IGO updates for different parametrizations coincide
at first order in $\deltat$, and may, in general, differ by $O(\deltat^2)$.
For instance the difference between the CMA-ES and xNES updates
is indeed $O(\deltat^2)$, see 
Section~\ref{sec:gaussian}.

For comparison, using the vanilla gradient results in a divergence
of $O(\deltat)$ at each step between different parametrizations, so
this divergence could be of the same magnitude as the steps
themselves.

In that sense, one can say that IGO algorithms are ``more
parametrization-invariant'' than other algorithms. 
This stems from their origin
 as a discretization of the IGO flow.

However, if the map $\phi$ is affine then this phenomenon
disappears: parametrizations that differ by an affine map on $\theta$
yield the same IGO algorithm.

\medskip

The next proposition states that, for example, if one uses a family of
distributions on $\R^d$ which is invariant under affine transformations,
then IGO algorithms optimize equally well a function and its image under
any affine transformation (up to an obvious change in the initialization).
This proposition generalizes the well-known corresponding property of CMA-ES
\citep[Proposition 9]{hansen2014principled}.

This invariance under $X$-transformations only holds provided the
$X$-transformation preserves the ``shape'' of the family of probability
distributions $P_\theta$, as follows.

Let us define, as usual, the image (\emph{pushforward}) of a probability distribution $P$ by a
transformation $\phi:X\to X$ as the probability distribution $P'$ such that
$P'(Y)=P(\phi^{-1}(Y))$ for any subset $Y\subset X$ \cite[Chapter
7]{Schilling_measure}. In the continuous
domain, the density of the new distribution $P'$ is obtained by the usual
change of variable formula involving the Jacobian of $\phi$ \cite[Chapter
15]{Schilling_measure}.

We say that a transformation $\phi:X\to X$ \emph{globally preserves}
a family of probability distributions $(P_\theta)$, if the image of any
$P_\theta$ by $\phi$ is equal to some distribution $P_{\theta'}$ in
the same family, and if moreover the correspondence
$\theta\mapsto\theta'$ is locally a diffeomorphism.

\begin{prop}[$X$-invariance]
\label{prop:Xinv}
Let $\phi:X\to X$ be a one-to-one transformation of the search space 
which globally preserves the family of measures
$P_\theta$. Let $\theta^t$ be the IGO flow trajectory for the
optimization of function $f$, initialized at
$P_{\theta^0}$. Let $(\theta')^t$
be the IGO flow trajectory for optimization of $f\circ\phi^{-1}$,
initialized at the image of $P_{\theta^0}$ by $\phi$. Then
$P_{(\theta')^t}$ is the image of $P_{\theta^t}$ by $\phi$.

For the discretized algorithm with population size $N$ and
step size $\deltat>0$, the same is true up to an error of $O(\deltat^2)$
per iteration. This error disappears if the action of the map $\phi$ on $\Theta$ by
pushforward is affine.
\end{prop}

The latter case of affine transforms is well exemplified by CMA-ES: here,
using the variance and mean as the parametrization of Gaussians, the new
mean and variance after an affine transform of the search space are an
affine function of the old mean and variance; specifically, for the affine transformation $A:x\mapsto Ax+b$ we have $(m, C)\mapsto (Am+b, ACA^\T)$.
Another example, on the discrete search space $X=\{0,1\}^d$, is
the exchange of $0$ and $1$: for reasonable choices of the family
$P_\theta$, the IGO flow and IGO algorithms will be invariant under such
a change in the way the data is presented.

\subsection{Speed of the IGO Flow}
\label{sec:speed}

\begin{prop}
\label{prop:speed}
The speed of the IGO flow, i.e.\ the norm of $\frac{\d \theta^t}{\d t}$
in the Fisher metric, is at most $\sqrt{\int_0^1 w^2-(\int_0^1 w)^2}$ where $w$ is the
selection scheme.
\end{prop}
The proof is given in Appendix~\ref{sec:proofs}.

A bounded speed means that the IGO flow will not explode in finite time,
or go out-of-domain if the Fisher metric on the statistical manifold
$\Theta$ is complete (for instance, the IGO flow on Gaussian
distributions will not yield non-positive or degenerate covariance
matrices). Due to the approximation terms $O(\deltat^2)$, this may not be
true of IGO algorithms.

\medskip

This speed can be monitored in practice in at least two ways. The first is
just to compute the Fisher norm of the increment
$\theta^{t+\deltat}-\theta^t$ using the Fisher matrix; for small
$\deltat$ this is close to $\deltat \norm{\frac{\d\theta}{\d t}}$ with
$\norm{\cdot}$ the Fisher metric. The
second is as follows: since the Fisher metric
coincides with the Kullback--Leibler divergence up to a factor $1/2$, we
have $\KL{P_{\theta^{t+\deltat}}}{P_{\theta^t}}\approx \frac12 \deltat^2
\norm{\frac{\d\theta}{\d t}}^2$ at least for small $\deltat$. Since it is relatively
easy to estimate $\KL{P_{\theta^{t+\deltat}}}{P_{\theta^t}}$ by comparing
the new and old log-likelihoods of points in a Monte Carlo sample, one can obtain
an estimate of $\norm{\frac{\d\theta}{\d t}}$.

\begin{cor}
\label{cor:KLspeed}
Consider an IGO algorithm with selection scheme $w$,
step size $\deltat$ and sample size $N$. Then, with probability \new{$1$},
\[
\frac{\KL{P_{\theta^{t+\deltat}}}{P_{\theta^t}}}{\deltat^2}\leq \frac12 
\Var_{[0,1]} w+o(1)_{(\deltat,N)\to(0,\infty)}
\enspace.
\]
\end{cor}

For instance, with $w(q)=\1_{q\leq q_0}$ and 
neglecting the error terms, 
an IGO algorithm
introduces at most $\frac12 \deltat^2\, q_0(1-q_0)$ bits of information (in base $e$) per
iteration into the probability distribution $P_\theta$. The proof is
given in Appendix~\ref{sec:proofs}.

Thus, the time discretization parameter $\deltat$ is not just an
arbitrary variable: it has an intrinsic interpretation related to a
number of bits introduced at each step of the algorithm.
This kind of
relationship suggests, more generally, to use the Kullback--Leibler
divergence as an external and objective way to measure learning rates in
those optimization algorithms which use probability distributions.

The result above is only an upper bound. Maximal speed can be achieved
only if all ``good'' points point in the same direction. If the various
good points in the sample suggest moves in inconsistent directions, then
the IGO update will be much smaller. While non-consistent moves are
generally to be expected if $N<\dim\Theta$, it may also be a sign that
the signal is noisy, or that the family of distributions $P_\theta$ is
not well suited to the problem at hand and should be enriched.

As an example, using a family of Gaussian distributions with unkown mean
and fixed identity variance on $\R^d$, one checks that for the
optimization of a linear function on $\R^d$, with the weight
$w(u)=\1_{u<1/2}$, the IGO flow moves at constant speed
$1/\sqrt{2\pi}\approx 0.4$, whatever the dimension $d$. On a rapidly varying sinusoidal function, the moving
speed will be much slower because there are ``good'' and ``bad'' points in all directions.

This may suggest ways to design the selection scheme $w$ to achieve
maximal speed in some instances. Indeed, looking at the proof of the
proposition, which involves a Cauchy--Schwarz inequality, one can see
that the maximal speed is achieved only if there is a linear relationship
between the weights $W_\theta^f(x)$ and the gradient $\widetilde \nabla_\theta \ln
P_\theta(x)$. For instance, for the optimization of a linear function on
$\R^d$ using Gaussian measures of known variance, the maximal speed will
be achieved when the selection scheme $w(u)$ is the inverse of the
Gaussian cumulative distribution function. (In particular, $w(u)$ tends
to $+\infty$ when $u\to 0$ and to $-\infty$ when $u\to 1$.) This is
in accordance with known results: the expected value of the
$i$-th order statistic of $N$ standard Gaussian variates is the optimal
$\wi$ value for Gaussians on the sphere function, $f(x) = \sum_i x_i^2$, where $d\to\infty$ \citep{HGBeyer01,arnold2006weighted}.
For $N\to\infty$, this order statistic converges
to the inverse Gaussian cumulative distribution function. 

\subsection{Noisy Objective Functions}
\label{sec:noisy}

Suppose that the objective function $f$ is non-deterministic: each time
we ask for the value of $f$ at a point $x\in X$, we get a random result.
In this setting we may write the random value $f(x)$  as
$f(x)=\tilde f(x,\omega)$ where $\omega$ is an unseen random
\del{parameter}\new{seed}, and
$\tilde f$ is a deterministic function of $x$ and $\omega$. Without
loss of generality, up to a change of variables we can assume that
$\omega$ is uniformly distributed in $[0,1]$.

We can still use the IGO algorithm without modification in this context.
One might wonder which properties (consistency of sampling, etc.) still
apply when $f$ is not deterministic. Actually, IGO algorithms for noisy
functions fit very nicely into the IGO framework: the following
proposition allows to transfer any property of IGO to the case of noisy
functions.

\begin{prop}[Noisy IGO]\label{prop:noisyIGO}
Let $f$ be a random function of $x\in X$, namely, $f(x)=\tilde
f(x,\omega)$ where $\omega$ is a random variable uniformly distributed in
$[0,1]$, and $\tilde f$ is a deterministic function of $x$ and
$\omega$. Then the two following algorithms coincide:
\begin{itemize}

\item The IGO algorithm \eqref{eq:intrinsicIGOupdate}, using a family of
distributions $P_\theta$ on space $X$, applied to the noisy function $f$,
and where the samples are ranked according to the random observed value of
$f$ (here we assume that, for each sample, the noise $\omega$ is
independent from everything else);

\item The IGO algorithm on space $X\times [0,1]$, using the
family of distributions $\tilde P_\theta=P_\theta\otimes U_{[0,1]}$, applied to
the deterministic function $\tilde f$. Here $U_{[0,1]}$ denotes the
uniform law on $[0,1]$.
\end{itemize}
\end{prop}

The (easy) proof is given in the Appendix~\ref{sec:proofs}.

This proposition states that noisy optimization 
can be modeled as ordinary distribution-based optimization with 
a component of the distribution being independent of the control 
parameter $\theta$. Conversely,\del{is the same as ordinary
optimization using a family of distributions which cannot operate any
selection or convergence over the parameter $\omega$. More generally,} any
component of the search space in which a distribution-based optimization 
algorithm cannot perform selection or specialization will effectively act
as a random noise on the objective function.

As a consequence of this result, all properties of IGO can be transferred to
the noisy case. Consider, for instance, consistency of sampling
(Theorem~\ref{thm:consistency}).
The $N$-sample IGO update rule for the noisy case is identical to the
non-noisy case \eqref{eq:IGOupdate}:
\[
\theta^{t+\deltat}=\theta^t+\deltat \, I^{-1}(\theta^t)\,
\sum_{i=1}^{N} \wi \, \left.\frac{\partial \ln
P_\theta(x_i)}{\partial
\theta}\right|_{\theta=\theta^t}
\]
where each weight $\wi$ computed from \eqref{eq:wi}  now
incorporates noise from the objective function because the rank of $x_{i}$ is computed on the random function, or equivalently on the deterministic function $\tilde f$: $\rk(x_{i})= \# \{j, \tilde f(x_{j},\omega_{j}) < \tilde f(x_{i},\omega_{i}) \} $. 

Consistency of sampling (Theorem~\ref{thm:consistency}) thus takes the
following form through Proposition~\ref{prop:noisyIGO}:
When $N\to\infty$, the $N$-sample IGO update rule on the noisy function $f$ converges with probability $1$ to the update rule
\begin{align}\nonumber
\theta^{t+\deltat}& =\theta^t+\deltat \, \widetilde\nabla_{\!\theta} \int_{0}^{1} \int
W_{\theta^t}^{\tilde f}(x,\omega)\, P_{\theta}(\d x) \, \d \omega \\\label{eq:IGOlimitnoisy}
& = \theta^t+\deltat \, \widetilde\nabla_{\!\theta} \int
\bar W_{\theta^t}^{f}(x)\, P_{\theta}(\d x)
\end{align}
where $\bar W^{f}_{\theta}(x) = \E_{\omega}W_{\theta}^{\tilde f}(x,\omega)$.
Thus when $N\to\infty$ the update becomes deterministic as the effects of noise
get averaged, as could be expected. 

Thus, applying the IGO algorithm to noisy objective functions as in Proposition~\ref{prop:noisyIGO}
requires defining the IGO flow for noisy objective functions as
\begin{equation}
\label{eq:noisyIGOflow}
\frac{\d\theta^t}{\d t}=\widetilde\nabla_{\!\theta} \int
\bar W_{\theta^t}^{f}(x)\, P_{\theta}(\d x)
\end{equation}
where $\bar W^{f}_{\theta}(x)=\E_{\omega}W_{\theta}^{\tilde f}(x,\omega)$
is the \emph{average} weight of $x$ over the values $f(x)$. 
Thus, the effects of noise remain visible even for 
$N\to\infty$, as $\bar W$ is in general flatter than $W$.

Note that there is another possible way to define the IGO flow for noisy
objective functions. The flow \eqref{eq:noisyIGOflow} amounts to taking a
point $x$, computing a (random) value $f(x)$, computing the level
$q_\theta^<(x)=\Pr\nolimits_{x'\sim P_\theta}(f(x')<f(x))$ of this random
value $f(x)$, applying the selection scheme $w$, and then averaging. In
this approach the value $q_\theta^<(x)$ is a random variable depending on
the value of $f(x)$.  Alternatively, we could define the\del{ quantile} value
$q_\theta(x)$ by applying Definition~\ref{def:quantiles} unchanged,
namely, $q_\theta^<(x) = \Pr\nolimits_{x'\sim P_\theta}(f(x')<f(x))$ in
which $f(x)$ itself is treated as a random variable under the probability
$\Pr$, so that $q_\theta^<(x)$ is deterministic and takes into account
all possible values for $f(x)$. Then we could apply the selection
scheme $w$ to this value $q_\theta^<(x)$ as in
Definition~\ref{def:quantiles} and define the IGO flow accordingly. The
value of $q_\theta^<(x)$ in this second version is the expected value of
$q_\theta^<(x)$ in the first version.

When the selection scheme $w$ is affine, these two approaches coincide;
however this is not the case in general, as the second version averages the
$q$-values while the first version averages the weights $w(q)$.  The second
version would be the $N\to\infty$ limit of a slightly more complex
algorithm using several evaluations of $f$ for each sample $x_i$ in order
to compute noise-free average ranks and quantiles. The first version has
the advantage of leaving the algorithm unchanged and of inheriting
properties from the deterministic case via
Proposition~\ref{prop:noisyIGO}.

\subsection{The IGO Flow for Linear Functions on
$\{0,1\}^d$ and $\R^d$}
\label{sec:examples}

In this section we take a closer look at the IGO differential equation solutions of
\eqref{eq:IGOflow} for some simple examples of \new{objective}\del{fitness} functions, for
which it is possible to obtain exact information about these IGO
trajectories.

We start with the discrete
search space $X = \{ 0, 1\}^{d}$ and linear functions (to be minimized)
defined as $f(x) = c - \sum_{i=1}^{d} \alpha_{i} x_{i}$ with $\alpha_{i}
> 0$. (So maximization of the classical onemax function $f_{\rm
onemax}(x)=\sum_{i=1}^{d} x_{i}$ is covered by setting $\alpha_{i}=1$.)
The differential equation of the IGO flow \eqref{eq:IGOflow} for the Bernoulli measures
$P_{\theta}(x)=p_{\theta_{1}}(x_{1}) \ldots p_{\theta_{d}}(x_{d})$
defined on $X$ is the $\deltat\to 0$ limit of the IGO-PBIL update
\eqref{eq:PBILIGOsimple}:
\begin{equation}\label{eq:igopbil}
\frac{\d \theta_{i}^{t}}{\d t} = \int W_{\theta^{t}}^{f}(x) (x_{i} - \theta_{i}^{t}) P_{\theta^{t}}(\d x) \eqd g_{i}(\theta^{t}) \enspace.
\end{equation}
Although finding the analytical solution of the differential equation
\eqref{eq:igopbil} for any initial condition seems a bit intricate, we
show that  the solution of \eqref{eq:igopbil} converges to $(1,\ldots,1)$
starting from any initial $\theta \in (0,1]^d$. Note that starting
with $\theta_i=0$ for some $i$ prevents IGO (and PBIL) from sampling any
$1$ for component $i$, so that the components of $\theta$ that are equal
to $0$ at startup will always stay so; in that case the IGO flow effectively works
as in a smaller-dimensional space.

To prove convergence to $(1,\ldots,1)$, we establish the following result:
\begin{lem}\label{lem:igo-pbil}
Assume that the selection scheme $w: [0,1] \to \R$ is bounded, that $w$ is a nonincreasing function and that there  exists $q_0, q_1 \in (0,1)$ such that $w(q_0) \not= w(q_1)$.
On $f(x) = c - \sum_{i=1}^{d} \alpha_{i} x_{i}$,
the solution of \eqref{eq:igopbil} satisfies $
\sum_{i=1}^{d} \alpha_{i} \frac{\d \theta_{i}^{t}}{\d t} \geq 0 $; moreover $ \sum \alpha_{i} \frac{\d \theta_{i}^{t}}{\d t} = 0 $ if and only $\theta \in \{0,1\}^d$.
\end{lem}
\begin{proof}
We compute $\sum_{i=1}^{d} \alpha_{i} g_{i}(\theta^{t})$ and find that 
\begin{align*}
 \sum_{i=1}^{d} \alpha_{i} \frac{\d \theta_{i}^{t}}{\d t} & = \int
 W_{\theta^{t}}^{f}(x) \left(\sum_{i=1}^{d} \alpha_{i} x_{i} -
 \sum_{i=1}^{d} \alpha_{i}  \theta_{i}^{t}\right) P_{\theta^{t}}(\d x) \\
 & = \int W_{\theta^{t}}^{f}(x) (f(\theta^{t}) - f(x)) P_{\theta^{t}}(\d x) \\
 & = \E[ W_{\theta^{t}}^{f}(x) ] \,\E[ f(x) ] - \E[ W_{\theta^{t}}^{f}(x) f(x) ]
\end{align*}
where the expectations are taken under $P_{\theta^t}$. Using
Lemma~\ref{lem-igopbil-intermediate} \new{below}, we have that $\E[ W_{\theta^{t}}^{f}(x) ] \,\E[ f(x) ] - \E[ W_{\theta^{t}}^{f}(x) f(x) ] \geq 0$ and in addition $\E[ W_{\theta^{t}}^{f}(x) ] \,\E[ f(x) ] - \E[ W_{\theta^{t}}^{f}(x) f(x) ] = 0$ if and only if $\theta^t \in \{0,1\}^d$.
\del{In addition, $W_{\theta^{t}}^{f}(x)$ is a non-increasing bounded function
in the variable $f(x)$, and so
 $- W_{\theta^{t}}^{f}(x)$ and $f(x)$ are positively correlated (see \cite[Chapter~1]{Thorisson} for a proof of this result), i.e.
$$
 \E[ - W_{\theta^{t}}^{f}(x) f(x) ] \geq  \E[ - W_{\theta^{t}}^{f}(x)] \E[ f(x) ]
$$
with equality if and only if \del{$\theta^{t}=(0,\ldots,0)$ or
$\theta^{t}=(1,\ldots,1)$}{$\theta^t \in \{0,1\}^d$}. \TODO[I don't
think this "if and only if" is really obvious in the end... we need to
give an argument here. ]Thus $\sum_{i=1}^{d} \alpha_{i} \frac{\d \theta_{i}^{t}}{\d t} \geq 0$ {and $\sum_{i=1}^{d} \alpha_{i} \frac{\d \theta_{i}^{t}}{\d t} =0$ if and only if $\theta \in \{0,1\}^d$}. }
\end{proof}

\begin{lem}\label{lem-igopbil-intermediate}
Under the assumptions of Lemma~\ref{lem:igo-pbil}
\begin{equation}
\E[ - W_{\theta^{t}}^{f}(x) f(x) ] \geq  \E[ - W_{\theta^{t}}^{f}(x)] \E[ f(x) ]
\end{equation}
with equality if and only if  $\theta^t \in \{0,1\}^d$. Here the
expectations are taken under $x\sim P_{\theta^t}$.
\end{lem}
\begin{proof}
We want to prove that $\Cov[- W_{\theta^{t}}^{f}(x), f(x)] \geq 0 $ with equality if and only if $\theta^t \in \{0,1\}^d$.
Let us define the random variable $Z=f(x)$ when $x \sim P_{\theta^t}$. Remark that $W_{\theta^t}^f(x) = \g(Z)$ where
\begin{equation}\label{eq:gfunction}
\g(z) = \frac{1}{F^{\leq}_{\theta^t}(z) - F^{<}_{\theta^t}(z) } \int_{F^{<}_{\theta^t}(z)}^{F^{\leq}_{\theta^t}(z)}  w(q) dq \enspace,
\end{equation}
where $F^{\leq}_{\theta^t}(z)= \Pr_{x'\sim P_{\theta^t}}( f(x') \leq z) $ and $F^{<}_{\theta^t}(z)= \Pr_{x'\sim P_{\theta^t}}( f(x') < z) $. That is $\g$ is the average of $w$ on the interval $(F^{<}_{\theta^t}(z), F^{\leq}_{\theta^t}(z))$. Since $w$ is nonincreasing, the function $\g$ is also nonincreasing. We have the following equality
$$
\Cov[- W_{\theta^{t}}^{f}(x), f(x)] 
=
\Cov[ - \g(Z), Z] 
$$ 
where $-\g$ is nondecreasing.
Following \cite[Chapter~1]{Thorisson}, we write 
$$
\Cov[ - \g(Z), Z ] = \frac{1}{2} \Cov[-\g(Z) + \g(Z'),Z - Z']
$$
where $Z$ and $Z'$ are independent following the distribution $P_{\theta^t}$. Given that both the mean of $-\g(Z) + \g(Z')$ and $Z-Z'$ are zero
$$
\Cov[ - \g(Z), Z ] =  \frac12  E[(-\g(Z) + \g(Z')) (Z - Z' ) ] = \frac12 E[|-\g(Z) + \g(Z')| |Z - Z'  | ] \geq 0 \enspace,
$$
where the last equality holds because $-\g$ is nondecreasing. This
proves the main statement of the lemma.

We will now show that if $\thetat \notin \{0,1\}^d$, then $ E[|-\g(Z) + \g(Z')| |Z - Z'  | ] > 0$ and thus consequently $\Cov[- W_{\theta^{t}}^{f}(x), f(x)] > 0 $. We simply need to show that $Z$ can take with strictly positive probabilities $p_1$ and $p_2$ two distinct values $z_1$ and $z_2$ such that $\g(z_1)$ and $\g(z_2)$ are distinct. We will then have
$$
E[|-\g(Z) + \g(Z')| |Z - Z'  | ] \geq \left( | - \g(z_1)  + \g(z_2)| | z_1 - z_2 | \right) p_1 p_2 > 0 \enspace.
$$
Let us assume that one single $\theta_i$ belongs to $(0,1)$ and all the
others are either $0$ or $1$. We can assume without loss of generality that $\theta_1 \in
(0,1)$. Then $f(x)$ takes (only) two distinct values with positive
probability; let $z_1$ be the one associated with $x_1=1$ and $z_2$
with $x_1=0$. Then $z_1 < z_2$ thanks to the definition of $f$ and
because $\alpha_i>0$. Moreover, unravelling the definitions we find $\g(z_1) = \frac{1}{\theta_1}\int_0^{\theta_1}w(q) dq$ and $\g(z_2) = \frac{1}{1-\theta_1} \int_{\theta_1}^1 w(q) dq $.

The assumption states that there exists $q_0, q_1 \in (0,1)$ such that
$w(q_0) \not= w(q_1)$. Assume without loss of generality that $q_0<q_1$
and $w(q_0)>w(q_1)$. Either $w(q_1) \geq w(\theta_1)$ or $w(q_1) <
w(\theta_1)$. In the first case, then $\theta_1>q_0$ and
\[
\frac{1}{\theta_1} \int_0^{\theta_1} w(q) dq = \frac{1}{\theta_1} \left(
\int_0^{q_0} w(q) dq + \int_{q_0}^{\theta_1} w(q) dq \right) >
\frac{1}{\theta_1}  \left( w(q_0) q_0 + w(\theta_1)(\theta_1-q_0) \right)>
w(\theta_1)\enspace,
\]
Meanwhile, $\g(z_2) = \frac{1}{1-\theta_1} \int_{\theta_1}^1 w(q) dq \leq w(\theta_1)$ and thus
$$
\g(z_1) > \g(z_2) \enspace.
$$
If  $w(q_1) < w(\theta_1)$, then $q_1>\theta_1$ and
\begin{multline}
\g(z_2) = \frac{1}{1-\theta_1} \int_{\theta_1}^1 w(q) dq = \frac{1}{1-\theta_1} \left( \int_{\theta_1}^{q_1} w(q) dq + \int_{q_1}^1 w(q) dq \right) \\ \leq \frac{1}{1-\theta_1} (w(\theta_1)(q_1 - \theta_1) + w(q_1)(1 - q_1)) < w(\theta_1) 
\end{multline}
and thus $\g(z_2) < w(\theta_1) \leq \g(z_1)$ as well.

In the case where we have not only two but $l$ distinct fitnesses that can be sampled with strictly positive probabilities, say $z_1 < \ldots < z_l$, we can show similarly that $\g(z_1) > \g(z_l) $.
\end{proof}

We are now ready to state the convergence of the trajectories of IGO-PBIL (and cGA) update in the limit of $\deltat\to 0$.

\begin{prop}\label{stability-igopbil}
Assume that the selection scheme $w: [0,1] \to \R$ is bounded, that $w$ is a nonincreasing function and that there  exists $q_0, q_1 \in (0,1)$ such that $w(q_0)\not=w(q_1)$.
On the linear functions $f(x) = c - \sum_{i=1}^{d} \alpha_{i} x_{i}$, the
critical point $\theta=(1,\ldots,1)$ of the
IGO-PBIL differential equation \eqref{eq:igopbil} is stable and for any initial condition $\theta\in (0;1]^d$, the continuous-time trajectory solving \eqref{eq:igopbil} converges to $(1,\ldots,1)$.
\end{prop}

\begin{proof}
Remark first that if we start at $\theta^0 \in [\epsilon,1]^d$, then
the solution of \eqref{eq:igopbil} stays in $[\epsilon,1]^d$. Indeed,
the trajectory cannot go out of $[0,1]^d$ and, in
addition, for $\theta \in [\epsilon, 1)^d$ we have  $g_i(\theta) \geq 0 $ so
that the trajectory of \eqref{eq:igopbil} cannot go out of
$[\epsilon,1]^d$. The proof that $g_i(\theta) \geq 0$ for $\theta \in
[\epsilon, 1)^d$ is similar to the proof that $\Cov[-
W_{\theta^{t}}^{f}(x), f(x)] \geq 0 $ in
Lemma~\ref{lem-igopbil-intermediate}: namely, we can write $g_i(\theta) = \int
\E[ W_{\theta^{t}}^{f}(x) | x_i ] (x_i - \theta_i) p_{\theta_i}(x_i) dx_i
= \theta_i (1 - \theta_i) [h(1) - h(0)]$ where $h(x_i) =   \E[
W_{\theta^{t}}^{f}(x) | x_i ] = \E[\g(f(x)) | x_i]$ where $\g$ is defined
in \eqref{eq:gfunction}. Then $h(1) \geq h(0)$ comes from the increase of
$\g$ when going to better function values.

Consider now on $[\epsilon,1]^d$, the non-negative function
$V(\theta)= \sum_{i=1}^{d} \alpha_{i} - \sum_{i=1}^{d} \alpha_{i}
\theta_{i}$. Then  $V^{*}(\theta) \deq \nabla V
(\theta) \cdot g(\theta) \leq 0 $, and $V^{*}(\theta)$ equals zero
only for $\theta \in \{0,1\}^d$ according to
Lemma~\ref{lem:igo-pbil}.
 Hence
for $\theta \in [\epsilon,1]^d$, the function $V$ is a Lyapunov function
\citep{Khalil1996book,
AgarwalDonal}, which
is minimal at $(1,\ldots,1)$ and  such that $V^*(\theta) < 0$ except
at $(1,\ldots,1)$.  Therefore the trajectory of \eqref{eq:igopbil} will
converge to $(1,\ldots,1)$ as $t$ goes to infinity (the proof is similar to that of \citep[Theorem 4.1]{Khalil1996book}). Given that this holds for any $\epsilon > 0$, we can conclude that the trajectory of  \eqref{eq:igopbil} converges to $(1,\ldots,1)$ starting from any $\theta \in (0,1]^d$.
\end{proof}


\newcommand{\NN}{\mathcal{N}}
\newcommand{\Id}{I_{d}}

We now consider on $\R^{d}$ the family of multivariate normal
distributions $P_{\theta} = \NN(\m,\sigma^{2} \Id)$ with covariance
matrix equal to $\sigma^{2} \Id$. The parameter $\theta$ thus has $d+1$
components $\theta=(\m,\sigma) \in \R^{d} \times \R$. The natural
gradient update using this family was derived in \citet{Glasmachers2010};
from this we can derive the IGO differential equation which reads:
%
\newcommand{\sig}{\sigma}
\newcommand{\sigtilde}{\tilde{\sigma}}
\begin{align}
\frac{\d \m^{t}}{\d t} & = \int_{\R^{d}} W_{\theta^{t}}^{f}(x)(x - \m^{t})P_{\NN(\m^{t},(\sigma^{t})^{2} \Id)}(x) \d x \\
\frac{\d \sigtilde^{t}}{\d t} & =  \int_{\R^{d}} \frac{1}{2d} \left\{
\sum_{i=1}^{d} \left( \frac{x_{i} - \m^{t}_{i}}{\sig^{t}} \right)^{2}  -
1 \right\} W_{\theta^{t}}^{f}(x) P_{\NN(\m^{t},(\sigma^{t})^{2} \Id)}(x)
\d x
\end{align} 
where $\sig^{t}$ and $\sigtilde^{t}$ are linked via $\sig^{t} =
\exp(\sigtilde^{t})$ or $\sigtilde^{t} = \ln(\sig^{t})$. Denoting $\NN$ a
random vector following a centered multivariate normal distribution with
identity covariance matrix we can rewrite the gradient flow as
\begin{align}\label{eq:esssone}
\frac{\d \m^{t}}{\d t} & = \sig^{t} \E \left[ W_{\theta^{t}}^{f}(\m^{t} + \sig^{t} \NN) \NN \right] \\\label{eq:essstwo}
\frac{\d \sigtilde^{t}}{\d t} & =  \E \left[ \frac12 \left(\frac{\|\NN \|^{2}}{d}  - 1 \right) W_{\theta^{t}}^{f}(\m^{t} + \sig^{t} \NN) \right] \enspace.
\end{align} 
Let us analyze the solution of the previous system on linear functions.
Without loss of generality (thanks to invariance) we can consider the
linear function $f(x) = x_{1}$. We have
$$
W_{\theta^{t}}^{f}(x) = w(\Pr( \m^{t}_{1} + \sigma^{t} Z_{1} < x_{1}))
$$
where $Z_{1}$ follows a standard one-dimensional normal distribution and thus
\newcommand{\CC}{\mathcal{F}}
\begin{align}
W_{\theta^{t}}^{f}(\m^{t} + \sig^{t} \NN) & = w(\Pr\nolimits_{Z_{1} \sim \NN(0,1)}(  Z_{1} < \NN_{1})) \\
& = w(\CC(\NN_{1}))
\end{align}
with $\CC$ the cumulative distribution of a standard normal distribution,
and $\NN_{1}$ the first component of $\NN$. The differential equation thus simplifies into
\begin{align}
\frac{\d \m^{t}}{\d t} & = \sig^{t} \left( \begin{array}{c}
 \E \left[ w(\CC(\NN_{1})) \NN_{1} \right]  \\
0  \\
\vdots \\
0  \end{array} \right) \\
\frac{\d \sigtilde^{t}}{\d t} & = \frac{1}{2d}\E \left[ (|\NN_{1} |^{2}  - 1 ) w(\CC(\NN_{1}))\right] \enspace.
\end{align} 

Consider, for example, the truncation selection function, i.e.\ $w(q)= 1_{q \leq q_{0}}$ where $q_{0} \in (0,1)$. 
(Within the IGO algorithm, this results in so-called
\emph{intermediate recombination} of the $q_0\times N$ best samples.) 
We find that

\begin{align}
\frac{\d \m^{t}_{1}}{\d t} & = \sig^{t} \E[\NN_{1} 1_{\{ \NN_{1} \leq \CC^{-1}(q_{0}) \}} ] \eqd \sig^{t} \beta \\
\frac{\d \sigtilde^{t}}{\d t} & = \frac{1}{2d} \left( \int_{0}^{q_{0}} \CC^{-1}(u)^{2} \d u - q_{0} \right) \eqd \alpha \enspace.
\end{align}

The solution of the IGO flow for the linear function $f(x)=x_{1}$ is thus given by
\begin{align}
\m^{t}_{1} & = \m^{0}_{1} + \frac{\sigma^{0}\beta}{\alpha} \exp(\alpha t) \\
\sig^{t} & = \sig^{0} \exp(\alpha t) \enspace .
\end{align}
The coefficient $\beta$ is negative for any $q_0<1$.
The coefficient $\alpha$ is  positive if and only if $q_{0} <
1/2$ by a simple calculus argument\footnote{Indeed $ \alpha =  \frac{1}{2 d \sqrt{2\pi}}  \int_{- \infty}^{\CC^{-1}(q_{0})} (x^{2} -1) \exp(-x^{2}/2) \d x = \frac{1}{2 d \sqrt{2\pi}}  g(\CC^{-1}(q_{0})) $ where $ g(y)= \int_{- \infty}^{y} (x^{2} -1) \exp(-x^{2}/2) \d x$.  
Using $g(0)=0$ and $\lim_{y \to \pm \infty} g(y)=0$, and studying the
sign of $g'(y)$, we find that $g$ is positive for $y<0$ and negative for
$y>0$. Since
$\CC^{-1}(q_{0})<0$  if and only if $q_{0} < 1/2$, we find that $\alpha =
\frac{1}{2 d \sqrt{2\pi}}  g(\CC^{-1}(q_{0})) $ is positive if
and only if $q_{0} < 1/2$.  }; 
this corresponds to  selecting less than half of the
sampled points in an ES. In this case the step size $\sigma^{t}$ grows
exponentially fast to infinity and the mean vector moves along the
gradient direction towards minus $\infty$ at the same
rate. If more than half of the points are selected, $q_{0} \geq 1/2$, the
step size will decrease to zero exponentially fast and the mean vector
will get stuck (compare also \citealt{grahl2005behaviour,hansen2006ecj,posik2008preventing}).

For an analysis of the solutions of the system of differential equations \eqref{eq:esssone} and \eqref{eq:essstwo} on more complex functions, namely convex-quadratic functions and twice continuously differentiable functions, we refer to \citet{Akimoto2012ppsn}.

\section{Proofs}
\label{sec:proofs}
\new{This final appendix provides longer proofs of propositions and theorems of the paper.} 

\subsection{Proof of Proposition~\ref{prop:maxent}}

We begin with a calculus lemma which will also be used for the proof of
Theorem~\ref{thm:IGOML}. The proof is omitted and amounts to writing the
maximum of a quadratic function obtained by the second-order Taylor
expansion of $f$.

\begin{lem}
\label{lem:gradargmax}
Let $f$ be real-valued function on a finite-dimensional vector space $E$
equipped with a positive definite quadratic form $\norm{\cdot}^2$. Assume
$f$ is smooth and has at most quadratic growth at infinity. Then, for any $x\in E$, we have
\[
\nabla f(x)=\lim_{\eps\to 0}\frac1\eps \argmax_h \left\{ f(x+h)-\frac{1}{2\eps}
\norm{h}^2\right\}
\]
where $\nabla$ is the gradient associated with the norm $\norm{\cdot}$.
Equivalently,
\[
\argmax_y \left\{ f(y)-\frac{1}{2\eps}
\norm{y-x}^2\right\}=x+\eps\nabla f(x)+O(\eps^2)
\]
when $\eps\to0$.
\end{lem}

Proposition~\ref{prop:maxent} follows by taking the Fisher information
metric at $\theta^0$ for the metric $\norm{\cdot}^2$, and using the
relation $\KL{P}{Q}=-\Ent(P)+\log\#X$
if $Q$ is the uniform distribution on a finite space $X$.

\subsection{Proof of Theorem~\ref{thm:consistency} (Convergence of
Empirical Means and Quantiles)}

Let us give a more precise statement including the necessary regularity
conditions.

\begin{prop}
Let $\theta\in\Theta$. Assume that the derivative $\frac{\partial\ln
P_\theta(x)}{\partial \theta}$ exists for $P_\theta$-almost all $x\in X$
and that $\E_{P_\theta} \abs{\frac{\partial\ln
P_\theta(x)}{\partial \theta}}^2<+\infty$. Assume that the function
$w$ is non-decreasing and bounded.

Let $(x_i)_{i\in \N}$ be a sequence of independent samples of $P_\theta$.
Then with probability $1$, as $N\to\infty$ we have
\[
\frac1N \sum_{i=1}^N \widehat W^f(x_i) \frac{\partial \ln
P_\theta(x_i)}{\partial \theta} \to \int W_\theta^f(x) \, \frac{\partial
\ln
P_\theta(x)}{\partial \theta}\, P_\theta(\d x)
\]
where
\[
\widehat W^f(x_i)=w\left(\frac{\rk_N(x_i)+1/2}{N}\right)
\]
with $\rk_N(x_i)=\#\{1\leq j\leq N,\, f(x_j)<f(x_i)\}$. (When there are
$f$-ties in the sample, $\widehat W^f(x_i)$ is defined as the average of
$w((r+1/2)/N)$ over the possible rankings $r$ of $x_i$.)
\end{prop}

\begin{proof}
Let $g:X\to \R$ be any function with $\E_{P_\theta} g^2<\infty$. We
will show that $\frac1N \sum \widehat W^f(x_i) g(x_i)
\to \int W_\theta^f(x)
g(x)\, P_\theta(\d x)$. Applying this with $g$ equal to the components of
$\frac{\partial \ln P_\theta(x)}{\partial \theta}$ will yield the result.

Let us decompose
\begin{align*}
\frac1N \sum \widehat W^f(x_i) g(x_i)=\frac1N \sum W_\theta^f(x_i) g(x_i)+\frac1N
\sum (\widehat W^f(x_i)-W_\theta^f(x_i)) g(x_i).
\end{align*}
Each summand in the first term involves only one sample $x_i$ (contrary
to $\widehat W^f(x_i)$ which depends on the whole sample). So by the
strong law of large numbers, almost surely $\frac1N \sum W_\theta^f(x_i)
g(x_i)$ converges to $\int W_\theta^f(x)
g(x)\, P_\theta(\d x)$. So we have to show that the second term converges
to $0$ almost surely.

By the Cauchy--Schwarz inequality,
we have
\[
\abs{\frac1N
\sum (\widehat W^f(x_i)-W_\theta^f(x_i)) g(x_i)}^2
\leq 
\left(\frac1N \sum (\widehat W^f(x_i)-W_\theta^f(x_i))^2 \right)
\left(\frac1N \sum g(x_i)^2 \right)
\]
By the strong law of large numbers, the second term $\frac1N \sum
g(x_i)^2$ converges to $\E_{P_\theta} g^2$ almost surely. So we have to
prove that the first term $\frac1N \sum (\widehat W^f(x_i)-W_\theta^f(x_i))^2$ converges to $0$
almost surely.

Since $w$ is bounded by assumption, we can write
\begin{align*}
(\widehat W^f(x_i)-W_\theta^f(x_i))^2
& \leq 2B\abs{\widehat W^f(x_i)-W_\theta^f(x_i)}
\\&=2B\abs{\widehat W^f(x_i)-W_\theta^f(x_i)}_+ +2B\abs{\widehat W^f(x_i)-W_\theta^f(x_i)}_-
\end{align*}
where $B$ is the bound on $\abs{w}$. We will bound each of these terms.

Let us abbreviate $q_i^<=\Pr_{x'\sim P_\theta}(f(x')<f(x_i))$,
$q_i^\leq=\Pr_{x'\sim P_\theta}(f(x')\leq f(x_i))$,
$r_i^<=\#\{j\!\leq\! N, \,f(x_j)<f(x_i)\}$, $r_i^\leq=
\#\{j\!\leq\! N , \, f(x_j)\leq f(x_i)\}$.

By definition of $\widehat W^f$ we have
\[
\widehat W^f(x_i)=\frac{1}{r_i^\leq -r_i^<}\sum_{k=r_i^<}^{r_i^\leq -1}
w((k+1/2)/N)
\]
and moreover $W_\theta^f(x_i)=w(q_i^<)$ if $q_i^<=q_i^\leq$ or
$W_\theta^f(x_i)=\frac{1}{q_i^\leq - q_i^<}\int_{q_i^<}^{q_i^\leq} w$ otherwise.

The Glivenko--Cantelli theorem \cite[Theorem 20.6]{Billingsley_PM}
implies that $\sup_i \abs{q_i^\leq -
r_i^\leq/N}$ tends to $0$ almost surely, and likewise for $\sup_i
\abs{q_i^< -
r_i^</N}$. So let $N$ be large enough so that these errors are bounded
by $\eps$.

Since $w$ is non-increasing, we have $w(q_i^<)\leq w(r_i^</N-\eps)$.
In the case $q_i^<\neq q_i^\leq$, we decompose the interval
$[q_i^<;q_i^\leq]$ into $(r_i^\leq -r_i^<)$ subintervals. The average of $w$
over each such subinterval is compared to a term in the sum defining
$w^N(x_i)$: since $w$ is non-increasing,
the average of $w$
over the $k^\text{th}$ subinterval is at most $w((r_i^<+k)/N-\eps)$. So we
get
\[
W_\theta^f(x_i)\leq \frac{1}{r_i^\leq -r_i^<}\sum_{k=r_i^<}^{r_i^\leq-1}
w(k/N-\eps)
\]
so that
\[
W_\theta^f(x_i)-\widehat W^f(x_i)\leq \frac{1}{r_i^\leq -r_i^<}
\sum_{k=r_i^<}^{r_i^\leq-1}
(w(k/N-\eps) -w((k+1/2)/N)).
\]

Let us sum over $i$, remembering that there are $(r_i^\leq -r_i^<)$ values
of $j$ for which $f(x_j)=f(x_i)$. Taking the positive part, we get
\[
\frac1N \sum_{i=1}^N \abs{W_\theta^f(x_i)-\widehat W^f(x_i)}_+
\leq \frac1N \sum_{k=0}^{N-1} (w(k/N-\eps) -w((k+1/2)/N)).
\]

Since $w$ is non-increasing we have
\[
\frac1N \sum_{k=0}^{N-1} w(k/N-\eps)\leq \int_{-\eps-1/N}^{1-\eps-1/N} w
\]
and
\[
\frac1N \sum_{k=0}^{N-1} w((k+1/2)/N)\geq \int_{1/2N}^{1+1/2N} w
\]
(we implicitly extend the range of $w$ so that $w(q)=w(0)$ for $q<0$\new{
and likewise for $q>1$}).
So we have
\[
\frac1N \sum_{i=1}^N \abs{W_\theta^f(x_i)-\widehat W^f(x_i)}_+
\leq \int_{-\eps-1/N}^{1/2N} w - \int_{1-\eps-1/N}^{1+1/2N}w
\leq (2\eps+3/N)B
\]
where $B$ is the bound on $\abs{w}$.

Reasoning symmetrically with $w(k/N+\eps)$ and the inequalities reversed,
we get a similar bound for $\frac1N \sum \abs{W_\theta^f(x_i)-\widehat W^f(x_i)}_-$.
This ends the proof.
\end{proof}

\subsection{Proof of Proposition~\ref{prop:qimprove} (Quantile
Improvement)}

Let us use the weight
$w(u)=\mathbbm{1}_{u\leq q}$.
Let $m$ be the value of the $q$-quantile of $f$ under $P_{\theta^t}$. We want to show that
the value of the $q$-quantile of $f$ under $P_{\theta^{t+\deltat}}$ is less than $m$,
unless the gradient vanishes and the IGO flow is stationary.

Let $p_-=\Pr_{x\sim P_{\theta^t}} (f(x)<m)$, $p_m=\Pr_{x\sim P_{\theta^t}}
(f(x)=m)$ and $p_+=\Pr_{x\sim P_{\theta^t}} (f(x)>m)$. By definition of the
quantile value we have $p_-+p_m\geq q$ and $p_++p_m\geq 1-q$.  Let us assume
that we are in the more complicated case $p_m\neq 0$ (for the case
$p_m=0$, simply remove the corresponding terms).

We have $W_{\theta_t}^f(x)=1$ if $f(x)<m$, $W_{\theta_t}^f(x)=0$ if $f(x)>m$
and $W_{\theta_t}^f(x)=\frac1{p_m} \int_{p_-}^{p_-+p_m}w(u)\d u=\frac{q-p_-}{p_m}$ if $f(x)=m$.

Using the same notation as above, let $g_t(\theta)=\int
W_{\theta^t}^f(x)\, P_\theta(\d x)$. Decomposing
this integral on the three sets $f(x)<m$, $f(x)=m$ and $f(x)>m$, we get
that 
$g_t(\theta)=\Pr_{x\sim P_\theta}(f(x)<m)+\Pr_{x\sim
P_\theta}(f(x)=m) \frac{q-p_-}{p_m}$.
In particular,
$g_t(\theta^t)= q$.

Since we follow a gradient ascent of $g_t$, for $\deltat$ 
small enough we
have $g_t(\theta^{t+\deltat})> g_t(\theta^t)$ unless the gradient
vanishes. If the gradient vanishes we have $\theta^{t+\deltat}=\theta^t$
and the quantiles are the same. Otherwise
we get $g_t(\theta^{t+\deltat})>g_t(\theta^t)=q$.

Since $\frac{q-p_-}{p_m}\leq \frac{(p_-+p_m)-p_-}{p_m} = 1$,
we have $g_t(\theta)\leq \Pr_{x\sim P_\theta}(f(x)<m)+\Pr_{x\sim
P_\theta}(f(x)=m)=\Pr_{x\sim P_\theta}(f(x)\leq m)$.

So $\Pr_{x\sim
P_{\theta^{t+\deltat}}}(f(x)\leq m)\geq g_t(\theta^{t+\deltat}) >q$. This
implies, by definition, that the
$q$-quantile value of $P_{\theta^{t+\deltat}}$ is at most $m$. Moreover,
if the objective function has no plateau then $\Pr_{x\sim
P_{\theta^{t+\deltat}}}(f(x)=m)=0$ and so $\Pr_{x\sim
P_{\theta^{t+\deltat}}}(f(x)< m)>q$ which implies that the $q$-quantile
of $P_{\theta^{t+\deltat}}$ is stricly less than $m$.

\subsection{Proof of Theorem~\ref{thm:IGOML} (Natural Gradient as
ML with Infinitesimal Weights)}

The proof of Theorem~\ref{thm:IGOML} will use
Lemma~\ref{lem:gradargmax}. Let $W$ be a
function of $x$, and fix some $\theta_0$ in $\Theta$.

We need some regularity assumptions: we assume that no two points $\theta\in\Theta$ define the
same probability distribution and that the map $P_\theta\mapsto\theta$
is continuous. We also assume that the map $\theta\mapsto P_\theta$ is
smooth enough, so that $\int \ln P_\theta(x) \, W(x) \,P_{\theta_0}(\d
x)$ is a smooth function of $\theta$. (These are restrictions on
$\theta$-regularity: this does not mean that $W$ has to
be continuous as a function of $x$.)

The two
statements of Theorem~\ref{thm:IGOML} using a sum and an integral have similar proofs, so we only
include the first.
For $\eps>0$, let
$\theta$ be the solution of
\[
\theta=\argmax\Bigg\{
(1-\eps) \int \ln P_\theta(x) \, P_{\theta_0}(\d x)
+\eps \int \ln P_\theta(x) \, W(x) \,P_{\theta_0}(\d x)
\Bigg\}.
\]

Then we have
\begin{align*}
\theta &=
\argmax\Bigg\{
\int \ln P_\theta(x) \, P_{\theta_0}(\d
x)
+\eps \int \ln P_\theta(x) \, (W(x)-1) \,P_{\theta_0}(\d x)
\Bigg\}
\\&=
\argmax\Bigg\{
\int \ln P_\theta(x) \, P_{\theta_0}(\d
x)
-\int \ln P_{\theta_0}(x) \, P_{\theta_0}(\d
x)
+\eps \int \ln P_\theta(x) \, (W(x)-1) \,P_{\theta_0}(\d x)
\Bigg\}
\intertext{(because the added term does not depend on $\theta$)}
&=\argmax\Bigg\{
-\KL{P_{\theta_0}}{P_\theta}+\eps \int \ln P_\theta(x) \, (W(x)-1)
\,P_{\theta_0}(\d x)
\Bigg\}\
\\
&=\argmax\Bigg\{
-\frac{1}{\eps} \KL{P_{\theta_0}}{P_\theta}+\int \ln P_\theta(x) \, (W(x)-1)
\,P_{\theta_0}(\d x)
\Bigg\}
\end{align*}

When $\eps\to 0$, the first term exceeds the second one if
$\KL{P_{\theta_0}}{P_\theta}$ is too large (because $W$ is bounded), and
so $\KL{P_{\theta_0}}{P_\theta}$ tends to $0$. So we can assume that
$\theta$ is close to $\theta_0$.

When $\theta=\theta_0+\delta\theta$ is close to $\theta_0$, we have
\[
\KL{P_{\theta_0}}{P_\theta}=\frac12 \sum I_{ij}(\theta_0)\,
\delta\theta_i\,\delta\theta_j
+O(\delta\theta^3)
\]
with $I_{ij}(\theta_0)$ the Fisher matrix at $\theta_0$. (This actually holds both for
$\KL{P_{\theta_0}}{P_\theta}$ and $\KL{P_{\theta}}{P_{\theta_0}}$.)

Thus, we can apply Lemma~\ref{lem:gradargmax} using the Fisher metric $\sum
I_{ij}(\theta_0)\,
\delta\theta_i\,\delta\theta_j$, and working on a small neighborhood of
$\theta_0$ in $\theta$-space (which can be identified with $\R^{\dim
\Theta}$). The lemma states that the argmax above is attained at
\[
\theta=\theta_0+\eps \widetilde\nabla_\theta \int \ln P_\theta(x) \,
(W(x)-1)
\,P_{\theta_0}(\d x)
\]
up to $O(\eps^2)$, with $\widetilde\nabla$ the natural gradient.

Finally, the gradient cancels the constant $-1$ because $
\int (\widetilde \nabla \ln P_\theta)\,P_{\theta_0}=0$ at
$\theta=\theta_0$. This proves Theorem~\ref{thm:IGOML}.

\subsection{Proof of Theorem~\ref{thm:IGOCEM} (IGO, CEM and IGO-ML)}

Let
$P_\theta$ be a family of probability distributions of the form
\[
P_\theta(\d x)=\frac1{Z(\theta)}\exp\left(\sum \theta_i
T_i(x)\right)\,H(\d x)
\] where $T_1,\ldots,T_k$
is a finite family of functions
on $X$ and $H(\d x)$ is some reference measure on $X$.
We assume that the family of functions $(T_i)_i$ together with the
constant function $T_0(x)=1$, are linearly independent. This prevents
redundant parametrizations where two values of $\theta$ describe the same
distribution; this also ensures, by elementary linear algebra, that the Fisher matrix $\Cov(T_i,T_j)$ is invertible.

The IGO update \eqref{eq:IGOupdate} in the parametrization $\bar T_i$ is
a sum of terms of the form \[
\widetilde\nabla_{\bar T_i} \ln P(x).
\]
So we will 
compute the natural gradient $\widetilde\nabla_{\bar T_i}$ in those
coordinates.

Let us start with a 
proposition giving an expression for the Fisher scalar product
between two tangent vectors $\delta P$ and $\delta'P$ of a statistical
manifold of exponential distributions. It is one way to express the duality between the coordinates
$\theta_i$ and $\bar T_i$ (compare \cite[(3.30) and Section
3.5]{Amari2000book}).

\begin{prop}
\label{prop:expdual}
Let $\delta\theta_i$ and $\delta'\theta_i$ be two small variations of the parameters $\theta_i$.
Let $\delta P(x)$ and $\delta' P(x)$ be the resulting variations of the
probability distribution $P$, and $\delta \bar T_i$ and $\delta' \bar
T_i$ the resulting variations of
$\bar T_i$. Then the scalar product, in Fisher
information metric, between the tangent vectors $\delta P$ and $\delta'
P$, is
\[
\langle \delta P,\delta' P\rangle=\sum_i \delta\theta_i \,\delta'\bar T_i=\sum_i
\delta'\theta_i \,\delta\bar T_i.
\]
\end{prop}


\begin{proof}

First, let us prove that $\frac{\partial \bar T_i}{\partial
\theta_j}=I_{ij}$, the Fisher matrix for $\theta$. Indeed, $\frac{\partial \bar T_i}{\partial
\theta_j}=\int_x T_i(x)\frac{\partial P(x)}{\partial \theta_j}=\int_x
T_i(x)P(x) \frac{\partial \ln P(x)}{\partial \theta_j}=\int_x
T_i(x)P(x)(T_j(x)-\bar T_j)$ by \eqref{eq:gradexp}; this is equal to
$\Cov(T_i,T_j)$ which is $I_{ij}$ by \eqref{eq:fishexp}. \note[Anne]{the
$\sum_x$ suggests it only holds for discrete distribution\NDY{fixed}. Also I do not see why $\Cov(T_i,T_j) =  \sum_x
T_i(x)P(x)(T_j(x)-\bar T_j)$, why you do not need $T_i$ to be centered?
\NDY{When computing a covariance, it is enough to center one of the two
variables only (any constant added to the other will cancel out):
$\E((X+c)(Y-\E Y))=\E(X(Y-\E Y))$.}\note[Anne]{OK thanks - not so sure why I didn't get it immediately by myself :-(.}}

Then, by definition of the Fisher metric we have $\langle \delta
P,\delta' P\rangle=\sum_{i,j} I_{ij} \,\delta\theta_i \,\delta'\theta_j$
but $\sum_j I_{ij} \delta'\theta_j$ is equal to $\delta' \bar T_i$ by the
above, because $\delta'\bar T_i=\sum_j 
\frac{\partial \bar T_i}{\partial
\theta_j}\delta'\theta_j$. Thus we find $\langle \delta
P,\delta' P\rangle=\sum_i \delta\theta_i \delta'\bar T_i$ as needed.\del{
By definition of
the Fisher metric:
\begin{align*}
\langle \delta P,\delta' P\rangle
&=\sum_{i,j} I_{ij} \,\delta\theta_i \,\delta'\theta_j
\\&=\sum_{i,j} \delta\theta_i \,\delta'\theta_j
\int_x \frac{\partial \ln P(x)}{\partial \theta_i}
\,\frac{\partial \ln P(x)}{\partial \theta_j}\,P(x)
\\&=
\int_x \sum_{i} \frac{\partial \ln P(x)}{\partial \theta_i}\delta\theta_i
\,\sum_j\frac{\partial \ln P(x)}{\partial \theta_j}\delta'\theta_j\,P(x)
\\&=\int_x
\sum_{i} 
 \frac{\partial \ln P(x)}{\partial \theta_i}\,\delta\theta_i 
\;\delta'(\ln P(x))\,P(x)
\\&=
\int_x \sum_i \frac{\partial \ln P(x)}{\partial \theta_i}\,\delta\theta_i\;\delta
'P(x)
\\&=\int_x \sum_i (T_i(x)-\bar T_i)\delta\theta_i\;\delta'
P(x)
\qquad\text{by \eqref{eq:gradexp}}
\\&=
\sum_i \delta\theta_i \left(\int_x T_i(x) \,\delta' P(x)\right)
-\sum_i \delta\theta_i \bar T_i \int_x \delta' P(x)
\\&=
\sum_i \delta\theta_i \,\delta'\bar T_i
\end{align*}
because $\int_x \delta' P(x)=0$ since the total mass of $P$ is $1$, and
$\int_x T_i(x) \,\delta' P(x)=\delta'
\bar T_i$ by definition of $\bar T_i$.  }
\end{proof}

\begin{prop}
\label{prop:dualgrad}
Let $f$ be a function on the statistical manifold of an exponential
family as above. Then the components of the natural gradient w.r.t.\ the
expectation parameters are given by the vanilla gradient w.r.t.\ the
natural parameters:
\begin{equation}
\label{eq:natgradduality2}
\widetilde\nabla_{\bar T_i} f=\frac{\partial f}{\partial \theta_i}
\end{equation}
and conversely
\begin{equation}
\label{eq:natgradduality}
\widetilde\nabla_{\theta_i} f=\frac{\partial f}{\partial \bar T_i}
\enspace.
\end{equation}
\end{prop}

(Beware this does \emph{not} mean that the gradient ascent in any of those
parametrizations is the vanilla gradient ascent.)

We could not find a reference for this result, though we think it is known
(as a consequence of \citealt[Eq.\ 3.32]{Amari2000book}).

\begin{proof}
%
We saw above that $\frac{\partial \bar T}{\partial \theta}=I$. Since
$\widetilde\nabla_\theta f=I^{-1}\frac{\partial f}{\partial \theta}$,
this proves \eqref{eq:natgradduality} by substituting 
$\frac{\partial f}{\partial \theta}=\left(\frac{\partial \bar T}{\partial
\theta}\right)^{\!\T}\frac{\partial f}{\partial \bar T}$.

For the first statement \eqref{eq:natgradduality2} (the one needed for Theorem~\ref{thm:IGOCEM}) we have
to derive the Fisher matrix for the variables $\bar T$. It follows from
Proposition~\ref{prop:expdual} that the Fisher matrix in these variables
is $I^{-1}$, by considering the Fisher metric
$\sum\delta\theta_i.\delta'\bar T_i$ and substituting $I^{-1}\delta\bar T$
for $\delta \theta$. Then \eqref{eq:natgradduality2}
is proved along the same lines as \eqref{eq:natgradduality}.\del{
By definition, the natural gradient $\widetilde\nabla f$ of a function
$f$ is the unique tangent vector $\delta P$  such 
that, for any other tangent vector $\delta' P$, we have
\[
\delta' f=\langle \delta P, \delta' P\rangle
\]
with $\langle \cdot,\cdot\rangle$ the scalar product associated with the
Fisher metric and $\delta'f $ the first-order variation of $f$ associated
with $\delta'P$. We want to compute this natural gradient in coordinates
$\bar T_i$, so we are interested in the variations $\delta \bar T_i$
associated with $\delta P$.

By Proposition~\ref{prop:expdual}, the scalar 
product above is
\[
\langle \delta P, \delta' P\rangle=\sum \delta \bar T_i\,
\delta'\theta_i
\]
where $\delta \bar T_i$ is the variation of $\bar T_i$ associated with
$\delta P$, and $\delta'\theta_i$ the variation of $\theta_i$ associated
with $\delta' P$.

On the other hand we have $\delta'f=\sum_i \frac{\partial
f}{\partial\theta_i}\,\delta'\theta_i$. So we must have
\[
\sum_i \delta \bar T_i \,\delta' \theta_i= \sum_i \frac{\partial
f}{\partial\theta_i}\,\delta'\theta_i
\]
for any $\delta'P$, which leads to
\[
\delta \bar T_i=\frac{\partial f}{\partial \theta_i}
\]
as needed. The converse relation is proved \emph{mutatis mutandis}.}
\end{proof}

Back to the proof of Theorem~\ref{thm:IGOCEM}. We can now compute the
desired terms:
\[
\widetilde\nabla_{\bar T_i}\ln P(x)=\frac{\partial \ln P(x)}{\partial
\theta_i}=T_i(x)-\bar T_i
\]
by \eqref{eq:gradexp}.
This proves the first statement \eqref{eq:IGOEP} in Theorem~\ref{thm:IGOCEM} about the
form of the IGO update in these parameters.

The other statements follow easily from this together with the
additional fact \eqref{eq:TML} that, for any set of (positive or
negative) weights $a_i$ with
$\sum a_i=1$, the value $T^\ast=\sum_i \new{a_i}\del{a(i)} T(x_i)$ 
maximizes $\sum_i \new{a_i}\del{a(i)} \ln P(x_i)$.

\subsection{Proof of Proposition~\ref{prop:speed} and
Corollary~\ref{cor:KLspeed} (Speed of IGO)}

\begin{lem}
Let $X$ be a centered $L^2$ random variable with values in $\R^d$ and let $A$ be a
real-valued $L^2$ random variable. Then
\[
\norm{\E (AX)}\leq \sqrt{\lambda \Var A}
\]
where $\lambda$ is the largest eigenvalue of the covariance matrix of $X$
expressed in an orthonormal basis.
\end{lem}

\begin{proof}\del{[Proof of the lemma]}
Let $v$ be any vector in $\R^d$; its norm satisfies
\[
\norm{v}=\sup_{w,\,\norm{w}\leq 1} (v\cdot w)
\]
and in particular
\begin{align*}
\norm{\E (AX)}
&=
\sup_{w,\,\norm{w}\leq 1} (w \cdot \E (AX))
\\&=
\sup_{w,\,\norm{w}\leq 1} \E (A\,(w \cdot X))
\\&=
\sup_{w,\,\norm{w}\leq 1} \E ((A-\E A)\,(w \cdot X))
\quad\text{since $(w \cdot X)$ is centered}
\\&\leq
\sup_{w,\,\norm{w}\leq 1} \sqrt{\Var A} \,\sqrt{\E((w\cdot X)^2)}
\end{align*}
by the Cauchy--Schwarz inequality.

Now, in an orthonormal basis we have
\[
(w\cdot X)=\sum_i w_i X_i
\]
so that
\begin{align*}
\E((w\cdot X)^2)
&=
\E\left((\tsum_i w_i X_i)(\tsum_j w_j X_j)\right)
\\&=\tsum_i \tsum_j \E (w_i X_i w_j X_j)
\\&=\tsum_i \tsum_j w_i w_j \E (X_i X_j)
\\&=\tsum_i \tsum_j w_i w_j C_{ij}
\end{align*}
with $C_{ij}$ the covariance matrix of $X$. The latter expression is the
scalar product $(w\cdot Cw)$. Since $C$ is a symmetric positive-semidefinite
matrix, $(w\cdot Cw)$ is at most $\lambda \norm{w}^2$ with
$\lambda$ the largest eigenvalue of $C$.
\end{proof}

For the IGO flow we have $\frac{\d \theta^{t}}{\d t}=\E_{x\sim P_\theta}
W_\theta^f (x) \tilde\nabla_\theta \ln P_\theta(x)$\del{ or $\frac{\d \theta^{t}}{\d t}=\E_{x\sim P_\theta}
\left(W_\theta^f (x) - \E_{x\sim P_\theta}(W_\theta^f (x)) \right) \tilde\nabla_\theta \ln P_\theta(x) $ since removing a constant from $W_\theta^f (x)$ does not affect the IGO flow (see discussion in Section~\ref{sec:impl})}.

So applying the lemma, we get that the norm of $\frac{\d \theta}{\d t}$
is at most $\sqrt{\lambda \Var_{x\sim P_\theta} W_\theta^f (x)}$ where $\lambda$ is the
largest eivengalue of the covariance matrix of $\tilde\nabla_\theta \ln
P_\theta(x)$ (expressed in a coordinate system where the Fisher matrix at
the current point $\theta$ is
the identity).

By construction of the quantiles, we have $\Var_{x\sim P_\theta} W_\theta^f
(x)\leq \Var_{[0,1]} w$ (with equality unless there are ties). 
Indeed, for a given $x$, let $\mathcal{U}$ be a uniform random variable
in $[0,1]$ independent from $x$ and define the random variable
$Q=q^<(x)+(q^\leq(x)-q^<(x))\mathcal{U}$. Then $Q$ is uniformly distributed between the upper 
and lower quantiles $q^\leq(x)$ and $q^<(x)$ and thus  we can rewrite
$W_\theta^f(x) $  as $\E(w(Q)|x)$. By the Jensen inequality we have $\Var
W_\theta^f(x) = \Var \E(w(Q)|x) \leq \Var w(Q)$. In addition when $x$ is
taken under $P_\theta$, $Q$ is uniformly distributed in $[0,1]$\NDY{More
detailed proof of this? It's not completely trivial, but rather tedious I
think.}\note[Anne]{I would say that it's OK like that in particular because it's a result that appears "often" in the context of order statistics (if I am not mistaken)} and thus $ \Var w(Q) = \Var_{[0,1]} w$, i.e., $\Var_{x\sim P_\theta} W_\theta^f
(x)\leq \Var_{[0,1]} w$. 

Besides, consider the tangent space in $\Theta$-space at point
$\theta^t$, and let us choose an orthonormal basis in this tangent space for
the Fisher metric. Then, in this basis
we have $\tilde \nabla_i \ln P_\theta(x)=\partial_i \ln P_\theta(x)$. So
the covariance matrix of $\tilde \nabla \ln P_\theta(x)$ is $\E_{x\sim
P_{\theta}} (\partial_i \ln P_\theta(x) \partial_j \ln P_\theta(x))$,
which is equal to the Fisher matrix by definition. So this covariance
matrix is the identity, whose largest eigenvalue is $1$. This proves
Proposition~\ref{prop:speed}.

For Corollary~\ref{cor:KLspeed}, by the relationship~\eqref{eq:IKL} between
Fisher matrix and Kullback--Leibler divergence, if $v$ is the speed of
the IGO flow then the Kullback--Leibler divergence between $P_{\theta^t}$
and $P_{\theta^{t+\deltat}}$ (where $P_{\theta^{t+\deltat}}$ is the
trajectory of the IGO flow after a time $\deltat$) is equal to the square
norm of $\deltat.v$ in Fisher metric up to an $O(\norm{\deltat.v}^3)$
term. Now if $P_{\theta^{t+\deltat}}$ is obtained by a finite-population
IGO \emph{algorithm}, by Theorem~\ref{thm:consistency} the actual $v$
from the IGO algorithm differs from the speed of the IGO flow by an
$o(1)_{N\to\infty}$ term. Collecting terms, we find the expression in
Corollary~\ref{cor:KLspeed}.

\subsection{Proof of Proposition~\ref{prop:noisyIGO} (Noisy IGO)}

On the one hand, let $P_{\theta}$ be a family of distributions on $X$.
The IGO algorithm \eqref{eq:intrinsicIGOupdate} applied to a random function  $f(x)=\tilde
f(x,\omega)$ where $\omega$ is a random variable uniformly distributed in
$[0,1]$ reads
\begin{equation}\label{eq:igonoisyone}
\theta^{t + \deltat} = \theta^{t} + \deltat \sum_{i=1}^{N} \hat{w_{i}} \widetilde \nabla_{\theta} \ln P_{\theta}(x_{i})
\end{equation}
where $x_{i} \sim P_{\theta}$ and $\hat{w_{i}}$ is according to
\eqref{eq:wi} where ranking is applied to the values $\tilde
f(x_{i},\omega_{i})$, with $\omega_{i}$ uniform variables in $[0,1]$
independent from $x_{i}$ and from each other.

On the other hand, for
the IGO algorithm using $P_\theta\otimes U_{[0,1]}$ and applied to the deterministic function $\tilde f$,  $\hat{w_{i}}$ is computed using the ranking according to the $\tilde f$ values of the sampled points $\tilde x_{i} = (x_{i}, \omega_{i}) $, and thus coincides with the one in \eqref{eq:igonoisyone}.

Besides,
$$
\partial_{\theta} \ln P_{\theta\otimes U_{[0,1]}}(\tilde x_{i})
=
\partial_{\theta} \ln P_{\theta}(x_{i}) + 
\underbrace{\partial_{\theta} \ln U_{[0,1]}(\omega_{i})}_{=0}
$$
and thus, both the vanilla gradients and the Fisher matrix $I$ (given by
the tensor square of the vanilla gradients) coincide.
This proves that the IGO algorithm update on space $X\times [0,1]$, using the
family of distributions $\tilde P_\theta=P_\theta\otimes U_{[0,1]}$, applied to
the deterministic function $\tilde f$, coincides with \eqref{eq:igonoisyone}.

\bibliography{14-467}

\end{document}